\documentclass[a4paper,11pt,twoside]{article}
\usepackage[T1]{fontenc}
\usepackage[applemac]{inputenc}% Codage du fichier en ISO-Latin-1 (accents...)
\usepackage[english,francais]{babel}% Utilisation du Fran^oais comme langue principale, de l'anglais comme langue secondaire
\usepackage{aeguill}% Guillemets et accents
\usepackage{xspace} % Espaces ins^ocables avant les ponctuations doubles
\usepackage{mathrsfs} % pour pouvoir utiliser \mathscr{F} (peut-√¬™tre plus joli que \mathcal)
\usepackage{amsfonts}
\usepackage{marvosym} % Quelques symboles utiles, comme l'euro (\EUR)
\usepackage{amsmath,amssymb,amsthm}% Formules math^omatiques
\usepackage{pstricks}
\usepackage{color,graphicx}
\usepackage{epsfig}
\usepackage[matrix,arrow]{xy}
\psfigdriver{dvips}

\usepackage[matrix,arrow]{xy}%Utilisation des matrices (pour les diagrammes par exemple)
\usepackage{multicol}%Utilisation du multicolonage
\usepackage{array}%Utilisation avanc^oe des tableaux
\usepackage{multirow}%Utilisation du multicolonnage dans les tableaux

\usepackage{fancyhdr} % En-t^otes personnalis^os ; je les personnalise plus loin
\usepackage{fancybox} % Pour redimensionner ou tourner des blocs

\definecolor{puces}{cmyk}{0,0,0,.5} % D^ofinition d'une couleur "puces"

\newcommand{\titre}{$p$-Adic families of Siegel modular  cuspforms}
\newcommand{\auteur}{Fabrizio Andreatta, Adrian Iovita and  Vincent Pilloni}

\title{\titre}
\author{\auteur}

\numberwithin{equation}{subsubsection}

\newcommand{\isolong}{\buildrel{\sim}\over\longrightarrow}

\theoremstyle{plain}
\newtheorem{defi}{Definition}[subsection]
\newtheorem{prop}[defi]{Proposition}
\newtheorem{thm}[defi]{Theorem}

\newtheorem{lem}[defi]{Lemma}
\newtheorem{coro}[defi]{Corollary}

\newtheorem{opprob}{Open problem}

\newtheorem{defi2}{Definition}[subsubsection]
\newtheorem{prop2}[defi2]{Proposition}
\newtheorem{thm2}[defi2]{Theorem}

\newtheorem{lem2}[defi2]{Lemma}
\newtheorem{coro2}[defi2]{Corollary}

\newtheorem{defi3}{Definition}[section]

\newtheorem{thm3}[defi3]{Theorem}

\theoremstyle{remark}
\newtheorem{rem}[defi]{Remark}

\newtheorem{rem2}[defi2]{Remark}

\newenvironment{demo}{\noindent{\textbf{Proof}}}{\hfill \qedsymbol}
\newcommand{\C}{\mathbb{C}}
\newcommand{\R}{\mathbb{R}}
\newcommand{\qq}{\mathbb{Q}}
\newcommand{\Z}{\mathbb{Z}}
\newcommand{\N}{\mathbb{N}}

\newcommand{\Spec}{\mathrm{Spec}}
\newcommand{\Spf}{\mathrm{Spf}}
\newcommand{\Spm}{\mathrm{Spm}}

\newcommand{\HH}{\mathrm{H}}

\newcommand{\ocal}{\mathcal{O}}
\newcommand{\oscr}{\mathscr{O}}

\newcommand{\B}{\mathrm{B}}
\newcommand{\T}{\mathrm{T}}
\newcommand{\HT}{\mathrm{HT}}

\newcommand{\cU}{\mathcal{U}}
\newcommand{\cE}{\mathcal{E}}
\newcommand{\cX}{\mathcal{X}}
\newcommand{\cY}{\mathcal{Y}}
\newcommand{\M}{\mathrm{M}}
\newcommand{\U}{\mathrm{U}}
\newcommand{\NU}{\mathrm{N}}
\newcommand{\I}{\mathrm{I}}

\newcommand{\G}{\mathrm{G}}
\newcommand{\Ha}{\mathrm{Ha}}
\newcommand{\Hdg}{\mathrm{Hdg}}
\newcommand{\XI}{X_{\mathrm{Iw}}}
\newcommand{\cXI}{\cX_{\mathrm{Iw}}}

\newcommand{\GL}{\mathrm{GL}_g}

\newcommand{\colim}{\mathrm{colim}}

\newcommand{\Hom}{\mathrm{Hom}}

\newcommand{\gl}{\mathfrak{g}}
\newcommand{\tor}{\mathfrak{t}}
\newcommand{\mfrak}{\mathfrak{m}}
\newcommand{\Fil}{\mathrm{Fil}}
\newcommand{\Gr}{\mathrm{Gr}}

\newcommand{\gsp}{\mathrm{GSp}_{2g}}

\newcommand{\cusp}{\mathrm{cusp}}
\newcommand{\un}{\mathrm{un}}
\newcommand{\an}{\mathrm{an}}

\newcommand{\rig}{\mathrm{rig}}

\begin{document}
\selectlanguage{english}

\maketitle
\begin{abstract}
Let $p$ be an odd prime and $g\geq 2$ an integer. We prove that a finite slope Siegel cuspidal
eigenform of genus $g$   can be $p$-adically deformed over the $g$-dimensional weight space.  The proof of this
theorem relies on the construction of a family  of sheaves of locally analytic overconvergent modular forms.
\end{abstract}
\tableofcontents
\section{Introduction}

%The theory of $p$-adic families of automorphic forms is now well developed since the original breakthrough of
%Hida (\cite{Hida})
%in the early eighties who constructed $p$-adic families of ordinary elliptic eigenforms.
%Later Coleman  (\cite{Col}, \cite{Colm}) extended Hida's
%result to the finite slope case and Coleman and Mazur (\cite{C-M}) constructed the eigencurve,  which is a
%$p$-adic rigid analytic
%curve parametrizing $p$-adic
%families of elliptic modular eigenforms of finite slope.

%During the last ten years many  authors have contributed to set up a general theory  of $p$-adic  automorphic
%forms on  higher rank groups. Some of them used an  approach  based on  the cohomology of arithmetic groups
%initiated by Hida and Stevens.  We quote the work of Ash-Stevens (\cite{AS}) who constructed $p$-adic families
%of modular symbols for the groups $(\GL)_{/\Q}$ and the work of E. Urban  (\cite{Ur}) who constructed
%equidimensional
%eigenvarieties  of the expected dimension  for modular eigensymbols of finite slope associated to
%reductive groups $G$ over $\qq$ such that $G(\R)$ admits discrete series.
%{}

After its glorious start in $1986$ with H.~Hida's article \cite{Hida}, the theory of {\em $p$-adic families of modular forms} was developed in various directions
and was applied  in order to prove many strong results in Arithmetic Geometry. One of its first applications was in the proof of the weight two
Mazur-Tate-Teitelbaum conjecture by R.~Greenberg and G.~Stevens in \cite{greenberg_stevens} and in the proof of certain cases of the Artin conjecture by K.~Buzzard,
M.~Dickinson, N.~Shepherd-Barron, R.~Taylor in \cite{BDSBT}. An important turn in its history was marked by R.~Coleman's
 construction of finite slope $p$-adic families of elliptic modular forms (\cite{Col} and \cite{Colm})
and by the construction of the eigencurve by R.~Coleman and B.~Mazur in \cite{C-M}.  The eigencurve is a
$p$-adic rigid analytic curve which parametrizes overconvergent elliptic modular eigenforms of
finite slope.

During the last fifteen years many  authors have contributed to set up a general theory  of $p$-adic  automorphic forms on  higher rank groups. Some of them used an
approach  based on  the cohomology of arithmetic groups initiated by Hida and Stevens.  Hida's idea, later on developed by Emerton (\cite{Em}) was to amalgamate
(more precisely take the projective limit, or as in \cite{Em}, alternatively consider the inductive limit followed by $p$-adic completion of) cohomology groups with
trivial ($\Z_p$) coefficients of a chosen tower of Shimura varieties. One obtains a large $\qq_p$-Banach space with an action of an appropriate Galois group, of the
$\qq_p$ (or even adelic, depending on the choice of the tower) points of the group and of a certain Hecke algebra. These data were used by H.~Hida in order to
produce a construction of the ordinary part of the eigenvariety for a large class of Shimura varieties.

In \cite{Em} there is also a construction of  finite slope eigenvarieties but so far it could not be proved
that the eigenvarieties thus constructed have the right dimension except in the cases already known:
for elliptic modular forms and modular forms on Shimura curves. Nevertheless the very rich structure
of the completed cohomology of towers of Shimura varieties was successfully used to prove results about the $p$-adic
local (and global) Langlands correspondence.

Stevens' approach is different, namely he uses the cohomology of one Shimura variety (of level type $\Gamma_0(Np)$ with $(N,p)=1$) with complicated coefficients
(usually certain locally analytic functions or distributions on the $\Z_p$-points of the group) as the space (called overconvergent modular symbols or $p$-adic
families of modular symbols) on which the Hecke operators act. These data were used by A.~Ash and G.~Stevens to produce eigenvarieties for   $\GL/\qq$ in \cite{AS}.
Recently E.~Urban  (\cite{Ur}) used the same method to construct equidimensional eigenvarieties  of the expected dimension  for modular eigensymbols of finite slope
associated to reductive groups $G$ over $\qq$ such that $G(\R)$ admits discrete series.

Finally, in constructing the eigencurve some authors (including Hida and Coleman)  used  a geometric approach
based on Dwork's ideas and  Katz's theory of $p$-adic modular
forms and overconvergent modular forms. Namely they interpolated directly the classical modular forms seen
as sections of certain automorphic line bundles on the modular curve (of level prime to $p$)
by defining overconvergent modular forms and allowing  them
to have essential singularities in ``small $p$-adic disks of very supersingular points''.
So far
this geometric approach  has only been successful, in the case of higher rank groups,
in producing ordinary families (\cite{Hi}) or one dimensional
finite slope families. The main theme of our work is to
by-pass these restrictions. In the  articles \cite{AIS} and \cite{PiTW} we explained new points of view
on the construction of the eigencurve of Coleman and
Mazur. Namely, we showed that over the open modular curve (of level prime to $p$), which is  the complement
of a disjoint union of ``small disks of very supersingular points'', one can interpolate the classical automorphic
sheaves and even construct $p$-adic families of such sheaves. We showed that Coleman's $p$-adic families can be seen as
global sections of such $p$-adic families of sheaves.

The present paper is a development of these ideas for Siegel varieties. We  prove that a cuspidal
automorphic form which occurs in the $\HH^0$ of the
coherent cohomology of some automorphic vector bundle on a Siegel variety  can be $p$-adically deformed
over the $g$ dimensional  weight space. We believe that the methods used in this
paper would certainly apply to any $\mathrm{PEL}$ Shimura variety of type $A$ and $C$ and maybe even to those
of type $D$.

\medskip

We now give  a more precise description of our main result. Let $p$ be an odd prime, $g \geq 2$ and
$N \geq 3$
two integers. We assume that $(p,N)=1$. Let
$Y_{\mathrm{Iw}}$ be the Siegel variety of genus $g$, principal level $N$ and Iwahori level structure at $p$.
This is the moduli space over $\Spec~\qq$ of principally
polarized abelian schemes $A$, equipped with a symplectic isomorphism $A[N] \simeq (\Z/N\Z)^{2g}$ and a
totally
isotropic flag $\Fil_\bullet A[p]: 0 = \Fil_0 A[p]
\subset \ldots \subset \Fil_{g} A[p] \subset A[p]$  where $\Fil_i A[p]$ has rank $p^i$. To any $g$-uple
$\kappa = (k_1,\ldots,k_g) \in \Z^{g}$ satisfying $k_1 \geq
k_2 \ldots \geq k_g$, one attaches an automorphic locally free sheaf $\omega^\kappa$ on $Y_{\mathrm{Iw}}$. Its global sections
$\HH^0(Y_{\mathrm{Iw}}, \omega^\kappa)$ constitute the module of classical Siegel modular forms of
weight $\kappa$. It contains the sub-module of cuspidal forms $\HH_{\cusp}^0(Y_{\mathrm{Iw}}, \omega^\kappa)$ which vanish
at infinity. On these modules we have an action of the unramified  Hecke algebra $\mathbb{T}^{Np}$ and of the dilating
Hecke algebra $\mathbb{U}_p =
\Z[U_{p,1},\ldots,U_{p,g}]$ at $p$. Let $f$ be a cuspidal eigenform and $\Theta_{f}\colon  \mathbb{T}^{Np}\otimes
\mathbb{U}_p \rightarrow \bar{\qq}$ be the
associated character. Since $f$ has Iwahori level at $p$, $\Theta_{f}(U_{p,i}) \neq 0$  and $f$  is of finite slope. We fix an embedding of $\bar{\qq}$ in $\C_p$
and denote by $v$ the valuation on $\C_p$ normalized by $v(p)=1$.

\begin{thm3}\label{thm1} Let $f$ be a weight $\kappa$ cuspidal eigenform of Iwahori level at $p$.
Then there is an affinoid neighbourhood $\mathcal{U}$  of $\kappa \in \mathcal{W} = \mathrm{Hom}_{cont} (
(\Z_p^\times)^g, \C_p^\times)$,  a finite surjective map of rigid analytic varieties

$$w\colon  \cE_f \rightarrow \mathcal{U}$$
and  a faithful, finite $\oscr_{\cE_f}$-module  $\mathcal{M}$ which is   projective as an $\oscr_{\mathcal{U}}$-module,

such that:

\begin{enumerate}
\item $\cE_f$ is equidimensional of dimension $g$.
\item We have a character $\Theta\colon \mathbb{T}^{Np}\otimes \mathbb{U}_p \rightarrow \oscr_{\cE_f}$.
\item $\mathcal{M}$ is an $\oscr_{\cE_f}$-module consisting  of finite-slope locally analytic cuspidal overconvergent modular forms.
The modular form $f$ is an element of   $\mathcal{M}\otimes_{\oscr_{\mathcal{U}}}^\kappa \C_p$, where the notation means
that the tensor product is taken with respect to the .$\oscr_{\mathcal{U}}$-module structure on $\C_p$
given by the ring homomorphism $\oscr_{\mathcal{U}}\longrightarrow \C_p$ which is evaluation of
the rigid functions on $\mathcal{U}$ at $\kappa$.

\item There is a point $x_f \in  \cE_f$, with $w(x_f) = \kappa$ and such that the specialization
of $\Theta$ at $x_f$ is $\Theta_f$.
\item For all $\mu = (m_1,\ldots,m_g) \in \Z^g \cap \mathcal{U}$ satisfying $m_1 \geq m_2\ldots
\geq m_g$, $v(\Theta_f(U_{p,i})) < m_{g-i}-m_{g-i+1} + 1$ for $1 \leq i \leq g-1$ and
$v(\Theta_f(U_{p,g})) < m_{g} - \frac{g(g+1)}{2}$
the following hold
\begin{itemize}
\item There is an inclusion  $\mathcal{M}\otimes_{\oscr_{\mathcal{U}}}^\mu\C_p \hookrightarrow  \HH_{\cusp}^0(Y_{\mathrm{Iw}}, \omega^\mu)$,
\item For any point $y$ in the fiber $w^{-1} (\mu)$, the character $\Theta_{y}$ comes from a weight
$\mu$ cuspidal Siegel eigenform on $Y_{\mathrm{Iw}}$.
\end{itemize}
\end{enumerate}

\end{thm3}

%The theorem still holds true for a form $f$ of weight $\kappa$, with non-trivial Nebentypus $\chi\colon (\Z/p\Z^\times)^g \rightarrow \C^\times$.

The rigid space $\cE_f$ is a neighbourhood of the point $x_f$ in an eigenvariety  $\cE $.
We actually prove the following:

\begin{thm3}\label{mainthm2} \begin{enumerate}
\item There is an equidimensional eigenvariety $ \cE $ and a locally finite map  to the weight space $w \colon \cE \rightarrow \mathcal{W}$. For any $\kappa \in \mathcal{W}$, $w^{-1}(\kappa)$ is in bijection with the  eigensystems of $\mathbb{T}^{Np}\otimes_\Z \mathbb{U}_p$ acting on the space of finite slope locally analytic overconvergent cuspidal modular forms of weight $\kappa$.
\item  Let  $f$ be a  finite slope locally analytic overconvergent cuspidal eigenform of   weight $\kappa$ and $x_f$ be the point corresponding to $f$ in $\cE$. If $w $ is unramified at $x_f$, then there is a  neighbourhood $\cE_f$ of $x_f$ in $\cE$ and a  family $F$ of finite slope locally analytic overconvergent cuspidal eigenforms  parametrized by $\cE_f$ and passing through $f$ at $x_f$.
\end{enumerate}
\end{thm3}

A key step in the proof of these theorems is the construction of the spaces of analytic overconvergent modular forms of any
weight
$\kappa \in \mathcal{W}$. They are  global sections of  sheaves $\omega^{\dag\kappa}_w$ which are defined over
some
strict neighbourhood of the multiplicative ordinary locus of $\XI$, a toroidal compactification of
$Y_{\mathrm{Iw}}$.
These sheaves are locally in the \'etale topology isomorphic to the $w$-analytic induction, from a Borel of
$\GL$ to the
Iwahori subgroup, of the character $\kappa$. They are  particular examples of sheaves over rigid spaces
which we call
Banach sheaves, and whose properties are studied in the appendix. We view these sheaves as the rigid
analytic analogues of quasi-coherent sheaves in algebraic geometry; for example, the Banach sheaves are
acyclic for sheaf cohomology on affinoids.

One important feature of the sheaves $\omega^{\dag\kappa}_w$ is that they vary analytically with the weight $\kappa$. One can thus define families of analytic
overconvergent modular forms parameterized by the weight and construct Banach spaces of  analytic overconvergent modular forms of varying weight. We have been able
to show that the module of cuspidal families, i.e.~the $\oscr_{\mathcal{U}}$-module $\mathcal{M}$ appearing in the theorem above is a  projective module. Therefore
one can use Coleman's spectral theory to construct $g$ dimensionional families of cuspidal eigenforms proving theorems \ref{thm1} and \ref{mainthm2}; see section
\ref{sec:moduleM} for a more detailed discussion. The  fifth part of the theorem \ref{thm1} is a special case of the main result of   \cite{PiS2} where a classicity
criterion (small slope forms are classical) for  overconvergent modular forms   is proved for many Shimura varieties.

As mentioned above, E.~Urban has constructed an eigenvariety  using the cohomology of arithmetic groups. Following \cite{Che}, one can prove that the reduced
eigenvarieties constructed in \cite{Ur} and in our paper coincide.  One way to think about our theorem is that   every cuspidal eigenform  gives a point on an
equidimensional component of the eigenvariety of dimension $g$. In \emph{loc.~cit.} this is proved in general when the weight is cohomological, regular  and the
slope is non critical. One advantage of our construction is that it provides
 $p$-adic  deformations of the ``physical" modular eigenforms   and  of their $q$-expansions.
For the symplectic groups, these carry more information than the Hecke eigenvalues.

\medskip
%{}

The paper is organized as follows. In the second section, we gather some useful and now classical results
about the $p$-adic interpolation of the algebraic representations of the group $\GL$. The idea is to
replace algebraic induction from the Borel to $\GL$ of a character by analytic induction from the
Borel to the Iwahori subgroup.  This is important because   the automorphic sheaf $\omega^\kappa$ is locally
over $\XI$ the algebraic induction of the character $\kappa$. Thus, locally for the Zariski topology over
$\XI$,  interpolating the sheaves $\omega^\kappa$ for varying $\kappa$ is equivalent to  interpolating
algebraic representations of $\GL$.  The third and fourth sections are about canonical subgroups.
We recall results of \cite{AM}, \cite{AG} and
\cite{Far}.  Using canonical subgroups we construct Iwahori-like subspaces in  the $\GL$-torsor of
trivializations of the co-normal sheaf of the universal semi-abelian scheme. They are used in section
five where we produce the Banach sheaves $\omega^{\dag\kappa}_{w}$. Section six is about Hecke operators. We
show that they act on our spaces of analytic overconvergent modular forms and we also construct a compact
operator $U$. In section seven we relate classical modular forms and analytic overconvergent modular forms.
This section relies heavily  on the main result of \cite{PiS2}. In section eight we finally construct
families. We let the weight $\kappa$ vary in
$\mathcal{W}$ and study the variation  of the spaces of overconvergent analytic modular forms.

We were able to control  this variation  on the cuspidal part, i.e. we showed that the specialization in any
$p$-adic weight of a family of cuspforms is surjective onto the space of cuspidal overconvergent
forms of that weight. The proof of this result is the technical heart of the paper. The main difference
between the case $g\ge 2$ and $g=1$ (see section \ref{sec:moduleM} for more details)
is the fact that the strict neighbourhoods $\cXI(v)$ of width $v$ of the multiplicative ordinary
locus, in some (any) toroidal compactification of the Siegel modular variety of Iwahori level,
are not affinoids.
Therefore, inspired by \cite{Hi}, we studied the descent of our families of Banach sheaves
$\omega_w^{\dag\kappa^\un}$ from
the toroidal to the minimal compactification and showed that the direct image of the cuspidal sub-sheaf of
$\omega_w^{ \dag\kappa^\un}$ is still a Banach sheaf.
Now the image of the strict neighbourhood $\cXI(v)$ in the minimal compactification is an
affinoid and the acyclicity of Banach sheaves on affinoids mentioned above allows us to prove the desired results, namely
that one can apply Coleman's spectral theory to the modules of $p$-adic families of cusp forms and obtain eigenfamilies of
finite slope. Moreover that any overconvergent modular form of finite slope is the specialization of a $p$-adic family
of finite slope,
in other words that any overconvergent modular form of finite slope deforms over the weight space.

\bigskip
\noindent

{\bf Aknowledgements} We thank Jacques Tilouine, Benoit Stroh and Michael Harris for many inspiring discussions on subjects pertaining to this research. We are also
grateful to the referee of this article for the careful reading of the paper and useful suggestions which hopefully led to its improvement.

\section{Families of representations of the group $\GL$}\label{sect_GL}
We recall some classical results about Iwahoric induction using the BGG analytic resolution of \cite{Ow} (see also \cite{Ur}).

\subsection{Algebraic representations}\label{sec:algrep} Let $\GL$ be the linear algebraic  group of dimension $g$ realized
as the group of $g\times g$ invertible matrices.
Let $\B$ be the Borel subgroup of upper triangular matrices, $\T$ the maximal torus of diagonal matrices, and $\U$
the unipotent radical of $\B$. We let $\B^0$ and
$\U^0$ be the opposite Borel of lower triangular matrices and its unipotent radical. We denote by $X(\T)$ the group of characters
of $\T$ and by $X^+(\T)$  its cone
of dominant weights with respect to $\B$. We identify $X(\T)$ with $\Z^g$ via the map which associates to a $g$-uple $(k_1,\ldots,k_g) \in \Z^g$ the character
$$   \begin{pmatrix} % or pmatrix or bmatrix or Bmatrix or ...
      t_1& 0 & 0 \\
      0 & \ddots & 0 \\
      0 & 0 & t_g
   \end{pmatrix} \mapsto t_1^{k_1}\cdots t_g^{k_g}.$$
   With this identification, $X^+(\T)$ is the cone of elements $(k_1,\ldots, k_g) \in \Z^g$ such that
$k_1 \geq k_2 \geq \ldots\geq k_g$. Till the end of   this paragraph, all group schemes are considered over
$\Spec~\qq_p$. For any $\kappa \in X^+(\T)$ we set
   $$V_\kappa = \big\{ f\colon  \GL \rightarrow \mathbb{A}^1\, \hbox{{\rm morphism of schemes s.t. }}
f(gb)= \kappa(b) f(g) ~\forall (g,b)\in \GL\times \B \big\}.$$
This is a finite dimensional $\qq_p$-vector space.   The group $\GL$ acts on $V_\kappa$ by the formula
$\bigl(g\cdot f\bigr)(x) = f(g^{-1}\cdot x)$ for any $(g,f)\in\GL\times
V_\kappa$. If $L$ is an extension of $\qq_p$ we set $V_{\kappa,L} = V_\kappa \otimes_{\qq_p} L$.

\subsection{The weight space}\label{sect_ws}

Let $\mathcal{W}$ be the rigid analytic space over $\qq_p$ associated to the noetherian, complete algebra $\Z_p[[\T(\Z_p]]$,
where let us recall
$\T$ is the split torus of diagonal matrices in $\GL$.  Let us fix an isomorphism $\T \simeq \mathbb{G}_m^g$.
We obtain  an isomorphism  $\T(\Z_p) \tilde{\rightarrow} \T(\Z/p\Z) \times (1+ p\Z_p)^g$ which implies that we have natural isomorphisms
as $\Z_p$-algebras
$$
\Z_p[[\T(\Z_p)]]\tilde{\longrightarrow}\bigl(\Z_p[\T(\Z/p\Z)]\bigr)[[(1+p\Z_p)^g]]
\tilde{\longrightarrow}\bigl(\Z_p[\T(\Z/p\Z)]\bigr)[[X_1,X_2,...,X_g]],
$$
where the second isomorphism is obtained by sending $(1,1,...,1+p,1...,1)$ with $1+p$ on the $i$-th component for $1\le i\le g$, to $X_i$.

It follows that the $\C_p$-points of $\mathcal{W}$ are described by: $\mathcal{W}(\C_p) = \Hom_{\rm cont}(\T(\Z_p), \C_p^\times)$
and if  we denote by $\widehat{  \T(\Z/p\Z)}$  the character group of $\T(\Z/p\Z)$,
 the weight space is isomorphic to a disjoint union,
indexed by the elements of $\widehat{  \T(\Z/p\Z)}$, of $g$-dimensional open unit polydiscs.

More precisely we have the following explicit isomorphism:

 \begin{eqnarray*} \mathcal{W} &\tilde{\rightarrow} & \widehat{  \T(\Z/p\Z)} \times \prod_{i=1}^g B(1, 1^-) \\
\kappa & \mapsto& \big(\kappa\vert_{\T(\Z/p\Z)}, \kappa( (1+p,1,\ldots,1)),
\kappa((1,1+p,\ldots,1)),\ldots, \kappa((1,\ldots, 1, 1+p))\big).
 \end{eqnarray*}

The inverse of the above map  is defined as follows:  $(\chi, s_1,\cdots,s_g) \in  \widehat{  \T(\Z/p\Z)} \times \prod_{i=1}^g B(1, 1^-)$
is assigned to the character which maps $(\lambda, x_1,\cdots, x_g) \in \T(\Z/p\Z) \times (1+ p\Z_p)^g$ to

$$\chi(\lambda)\prod_{i=1}^gs_i^\frac{ \mathrm{log} (x_i)}{\mathrm{log}(1+p)}.$$

\smallskip
\noindent
{\em The universal character}

If we denote by $\oscr_{\mathcal{W}}$ the sheaf of rigid analytic functions on $\mathcal{W}$, we have a natural continuous group homomorphism,
obtained as the composition
$$
\kappa^{\rm un}: \T(\Z_p) \longrightarrow \bigl(\Z_p[[\T(\Z_p)]]\bigr)^\times\longrightarrow \oscr_{\mathcal{W}}(\mathcal{W})^\times,
$$
where the first map is the tautological one. We call $\kappa^{\rm un}$ the universal character. It can alternatively be seen as a
pairing $\kappa^{\rm un}:\mathcal{W}(\C_p)\times \T(\Z_p)\longrightarrow \C_p^\times$  satisfying
the property: for every $t\in \T(\Z_p), \kappa\in \mathcal{W}(\C_p)$
we have $\kappa^{\rm un}(t)(\kappa)=\kappa(t)$. If $\mathcal{U}=\Spm~A\subset \mathcal{W}$ is an admissible affinoid open, we obtain
a universal character for $\mathcal{U}$, $\T(\Z_p)\rightarrow A^\times$ which is the composition of
$\kappa^{\rm un}$ with the natural restriction  homomorphism
$\oscr_{\mathcal{W}}(\mathcal{W})^\times\rightarrow A^\times$. This character will be also denoted by
$\kappa^{\rm un}$ and it may be seen as an
$A$-valued weight, i.e. $\kappa^{\rm un}\in \mathcal{W}(A)$.

\begin{defi} Let $w \in \qq_{>0}$. We say that a character $\kappa \in \mathcal{W}(\C_p)$ is $w$-analytic if $\kappa$ extends to an analytic map
$$\kappa\colon  \T(\Z_p) (1 + p^w \ocal_{\C_p})^g \rightarrow \C_p^\times.$$
\end{defi}

It follows  from the  classical $p$-adic properties of the exponential and the logarithm  that any character $\kappa$ is $w$-analytic for some $w >0$.
In fact, we have the following proposition :
\begin{prop}[\cite{Ur}, lem. 3.4.6] \label{prop_cara} For any quasi-compact open subset $\mathcal{U}\subset \mathcal{W}$,
there exists $w_{\mathcal{U}} \in \R_{>0}$ such that the universal character
$\kappa^{\un}\colon \mathcal{U}\times \T(\Z_p) \rightarrow \C_p^\times$, extends to an analytic function
$\kappa^{\un}\colon  \mathcal{U} \times \T(\Z_p)(1 + p^{w_U} \ocal_{\C_p})^{g} \rightarrow \C_p^\times$.
\end{prop}

%\subsubsection{Back to the weight space}\label{sec:Wo}

\bigskip
\noindent
In what follows we  construct an admissible affinoid  covering $ \cup_{w>0} \mathcal{W}(w)$ of the weight space $\mathcal{W}$
with the property that for every $w$ the restriction of the universal character $\kappa^{\rm un}$ to
$\mathcal{W}(w)$ is $w$-analytic.

We start by fixing  $w \in ]n-1,n]\cap \qq$ and we choose a finite extension $K$ of $\qq_p$ whose ring of integers,  denoted $\ocal_K$
contains an element $p^w$ of valuation $w$.
We set $\mathfrak{W}(w)^o = \Spf~\ocal_K \langle\langle S_1,\ldots,S_g\rangle\rangle$, it is a
formal unit polydisc over $\Spf~\ocal_K$.

We  define the formal $w$-torus $\mathfrak{T}_w:=\Spf~\ocal_K\langle\langle 1+p^wX_1,\ldots, 1+p^wX_g\rangle\rangle$
and attach to $\mathfrak{W}(w)^o$ a formal universal character
\begin{eqnarray*}
 \kappa^{o\un} \colon  \mathfrak{T}_w \times \mathfrak{W}(w)^o & \rightarrow& \widehat {\mathbb{G}}_m \\
 (1+p^wX_1,\ldots, 1+p^wX_g, S_1,\ldots,S_g) &\mapsto &\prod_{i=1}^g (1+p^wX_i)^{S_ip^{-w + \frac{2}{p-1}}}
\end{eqnarray*}

%{}

Let $\mathcal{W}(w)^o$ be the rigid analytic generic fiber of $\mathfrak{W}(w)^o$. We define $\mathcal{W}(w)$
to be the fiber product:

$$ \mathcal{W} \times_{\Hom_{cont}\bigl((1+p^n\Z_p)^g, \C_p^\times\bigr) } \mathcal{W}(w)^o,$$
where the maps used to define the fiber product are the following:
$\mathcal{W}\longrightarrow \Hom_{cont}\bigl((1+p^n\Z_p)^g,\C_p^\times  \bigr)$ is restriction
and the map $\mathcal{W}(w)^o\longrightarrow \Hom_{cont}\bigl((1+p^n\Z_p)^g,\C_p^\times\bigr)$
is given by
$$
(s_1,s_2,\ldots,s_g)\longrightarrow \Bigl((1+p^nx_1,1+p^nx_2,\ldots,1+p^nx_g)\rightarrow\prod_{i=1}^g
(1+p^nx_i)^{s_ip^{-w+\frac{2}{p-1}}}\Bigr).
$$

Then,  $\mathcal{W} = \cup_{w>0}\mathcal{W}(w)$ is an increasing cover by affinoids. By construction,
the restriction of the universal character $\kappa^{\un}$ of
$\mathcal{W}$ to $\mathcal{W}(w)$ is $w$-analytic.

\subsection{Analytic representations}\label{sec:analrep}

Let $\I$ be the Iwahori sub-group of $\GL(\Z_p)$ of matrices whose reduction modulo $p$ is upper triangular.
Let $\NU^0$ be the subgroup of $\U^0(\Z_p)$ of matrices which reduce to the identity modulo $p$. The Iwahori
decomposition is an isomorphism: $\B(\Z_p)\times\NU^0 \rightarrow \I$. We freely identify $\NU^0$  with $(p\Z_p)^{\frac{g(g-1)}{2}}
\subset \mathbb{A}_{\an}^{\frac{g(g-1)}{2}}$, where $\mathbb{A}_{\an}$ denotes the rigid analytic affine line defined over $\qq_p$.
For $\epsilon >0$, we let $\NU^0_\epsilon$ be the rigid analytic space $$\bigcup_{x \in (p\Z_p)^{\frac{g(g-1)}{2}}}
B(x, p^{-\epsilon}) \subset \mathbb{A}_{\an}^{\frac{g(g-1)}{2}}.$$

Let $L$ be an extension of $\qq_p$ and  $\mathcal{F}(\NU^0, L)$ the ring of $L$-valued functions on $\NU^0$. We
say that a function $f \in \mathcal{F}(\NU^0, L)$ is $\epsilon$-analytic if it is the restriction to $\NU^0$ of a necessarily unique analytic function
on $\NU^0_\epsilon$. We denote by $\mathcal{F}^{\epsilon-\an}(\NU^0,L)$ the set of $\epsilon$-analytic functions. A function is analytic if it is
$p^{-1}$-analytic.  We simply denote by $\mathcal{F}^{\an}(\NU^0,L)$ the set of analytic functions. We let
$\mathcal{F}^{l-\an}(\NU^0, L)$ be the set of locally analytic functions on $\NU^0$
i.e. the direct limit of the sets
$\mathcal{F}^{\epsilon-\an}(\NU^0,L)$ for  all $\epsilon>0$.

Let $\epsilon >0$ and   $\kappa \in \mathcal{W}(L)$ be an $\epsilon$-analytic character. We set
$$V_{\kappa,L}^{\epsilon-\an} = \big\{ f\colon  \I \rightarrow  L,~ f(ib) = \kappa(b) f(i)~ \forall (i,b) \in \I \times \B(\Z_p), ~f\vert_{\NU^0} \in \mathcal{F}^{\epsilon-\an}(\NU^0,L) \big\}.$$
  We define  similarly $V_{\kappa,L}^{\an}$ and $V_{\kappa,L}^{l-\an}$. They are all representations of the Iwahori  group $\I$.
\subsection{The BGG resolution}\label{subsect_BGGGL} Let $W$ be the Weyl group of $\GL$, it acts on $X(\T)$. We set $\gl$ and $\tor$ for the  Lie algebras of
$\GL$ and $\T$. The choice of $\B$ determines a system of simple positive roots
$\Delta \subset X(\T)$. To any $\alpha \in \Delta$ are associated an element $H_\alpha \in \tor$,
elements $X_\alpha \in \gl_\alpha$ and $X_{-\alpha} \in \gl_{-\alpha}$ such that $[ X_\alpha, X_{-\alpha}] = H_\alpha$ and a
co-root $\alpha^\vee$.  We let $s_\alpha \in W$ be the symmetry $\lambda \mapsto \lambda- \langle \lambda,\alpha^\vee
\rangle\alpha$. For any $w\in W$ and $\lambda\in X(\T)$, we set $w\bullet \lambda =  w(\lambda +\rho) - \rho$, where $\rho$ is half the sum
of the positive roots. By the main result of \cite{Ow}, for all $\kappa \in X^+(\T)$, and any field extension $L$ of $\qq_p$,
we have an exact sequence of $\I$-representations:
\begin{eqnarray}\label{exactsequence}  0 \longrightarrow V_{\kappa,L} \stackrel{d_0}{\longrightarrow} V_{\kappa,L}^{\an}
\stackrel{d_1}{\longrightarrow }\bigoplus_{\alpha \in \Delta} V_{s_\alpha\bullet \kappa,L}^{\an}
\end{eqnarray}
Let us make  explicit the differentials. The map $d_0$ is the natural inclusion, the map $d_1$ is the sum of maps $\Theta_\alpha\colon  V^{\an}_{\kappa,L}
\rightarrow V^{\an}_{s_\alpha\bullet \kappa,L}$ whose definitions we now recall. We let $\I$ act on the space of analytic functions on $\I$ by the
formula $(i\star f)(j)
= f(j \cdot i)$ for any analytic function $f$ and $i,j\in I$. By differentiating we obtain an action of $\gl$ and hence of the
enveloping algebra $U(\gl)$ on the space of
analytic functions on $\I$. If $f \in V_{\kappa,L}^{\an}$ we set $\Theta_\alpha(f) = X_{-\alpha}^{\langle\kappa,\alpha^\vee\rangle +1}\star f$. We now show that
$\Theta_\alpha(f) \in V_{s_\alpha\bullet \kappa,L}^{\an}$. First of all let us check that $\Theta_\alpha(f)$ is
$\U(\Z_p)$-invariant. It will be enough to prove that
$X_\beta\star \Theta_\alpha(f)=0$ for all $\beta \in \Delta$. If $\beta \neq \alpha$, this follows easily for $[X_\beta,X_{-\alpha}]=0$.
If $\beta=\alpha$, we have
to use the relation
$$[X_\alpha,X_{-\alpha}^{\langle\kappa,\alpha^\vee\rangle}] = (\langle\kappa,\alpha^\vee\rangle +1) X_{-\alpha}
(H_\alpha- \langle\kappa, \alpha^\vee\rangle).$$
We now have
\begin{eqnarray*}
X_\alpha\star \Theta_\alpha(f)&=& [X_\alpha,X_{-\alpha}^{\langle\kappa, \alpha^\vee\rangle+1}] \star f \\
&=&  (\langle\kappa,\alpha^\vee\rangle +1) X_{-\alpha}(H_\alpha- \langle\kappa, \alpha^\vee\rangle)\star f \\
&=& 0.
\end{eqnarray*}
Let us  find the weight of $\Theta_\alpha(f)$. For any $t\in \T(\qq_p)$,  We have
\begin{eqnarray*}
t\star \Theta_{\alpha}(f) &=& \mathrm{Ad}(t)(X_{-\alpha}^{\langle\kappa,\alpha^\vee\rangle+1})t\star f \\
&=& \alpha^{-\langle\kappa,\alpha^\vee\rangle-1}(t) \kappa(t) \Theta_\alpha(f).
\end{eqnarray*}
Since we have $\alpha^{-\langle\kappa,\alpha^\vee\rangle-1} \kappa= s_\alpha\bullet\kappa$, the map $\Theta_\alpha$ is well defined.
\subsection{A classicity criterion}\label{sect_class_GL} For $1\leq i \leq g-1$ we set
$d_i =       \begin{pmatrix} % or pmatrix or bmatrix or Bmatrix or ...
     p^{-1} \mathrm{1}_{g-i} & 0 \\
      0 & \mathrm{1}_{i} \\
   \end{pmatrix} \in \GL(\qq_p)$. The adjoint action of $d_i$ on $\GL/\qq_p$ stabilizes the  Borel subgroup $\B$.
The formula $(\delta_i\cdot f)(g) := f(d_i g d_i^{-1})$ defines an action on the space $V_\kappa$ for any $\kappa \in X^+(\T)$.  We now define the action on the spaces
$V_{\kappa,L}^{\epsilon-\an}$ for any $\kappa \in \mathcal{W}(L)$. We have a well-defined adjoint action of $d_i$ on the group $\NU^0$. Let $f\in
V_\kappa^{\epsilon-\an}$ and $j\in \I$. Let $j= n  \cdot b$ be  the Iwahori decomposition of $j$. We set $\delta_if(j) := f( d_i n d_i^{-1} b)$. We hence get
operators $\delta_i$ on  $V_{\kappa,L}^{\epsilon-\an}$ and $V_{\kappa,L}^{l-\an}$.  Let $z_{k,l}$ be the $(k,l)$-matrix coefficient on $\GL$. If we use the
isomorphism $V_{\kappa,L}^{\epsilon-\an} \rightarrow \mathcal{F}^{\epsilon-\an}(\NU^0, L )$ given by the restriction of functions to $\NU^0$, then the operator
$\delta_i$ is given by
   \begin{eqnarray*}
   \delta_i\colon   \mathcal{F}^{\epsilon-\an}(\NU^0, L ) &\rightarrow &  \mathcal{F}^{\epsilon-\an}(\NU^0, L ) \\
   f & \mapsto & [ (z_{k,l})_{k<l} \mapsto f( p^{n_{k,l}}z_{k,l})]
   \end{eqnarray*}
   where  $n_{k,l} = 1$ if $k\geq g- i+1$ and $l \leq g-i$ and $n_{k,l}=0$ otherwise. The operator $\delta_i$ is norm decreasing and the operator
$\prod_{i}\delta_i$ on  $V_\kappa^{\epsilon-\an}$ is completely continuous.

If $\kappa \in X(\T)^+$ the map $d_0$ in the exact sequence (\ref{exactsequence}) is $\delta_i$-equivariant. Regarding the map $d_1$ we have the following variance formula
$$\delta_i \Theta_\alpha = \alpha(d_i)^{\langle\kappa ,\alpha^\vee\rangle+1} \Theta_\alpha \delta_i.$$
Indeed for any $f \in V_{\kappa,L}^{\epsilon-\an}$ we have
\begin{eqnarray*}
\delta_i \Theta_\alpha (f) &= &d_i\cdot ( d_i^{-1} X_{-\alpha}^{\langle\kappa,\alpha^\vee\rangle+1} \star f)\\
&=& \alpha(d_i)^{\langle\kappa,\alpha^\vee\rangle+1} d_i\cdot (X_{-\alpha}^{\langle\kappa,\alpha^\vee\rangle+1}d_i^{-1}\star f)\\
&=& \alpha(d_i)^{\langle\kappa,\alpha^\vee\rangle+1} \Theta_{\alpha}(\delta_if).
\end{eqnarray*}
Let $\underline{v}= (v_1,\ldots,v_{g-1}) \in \R^{g-1}$. We let $V_\kappa^{l-\an, < \underline{v}}$ be the union of the generalized eigenspaces where $\delta_i$ acts by eigenvalues of valuation strictly smaller than $v_i$. We are now able to give the classicity criterion.

\begin{prop} Let $\kappa =(k_1,\ldots,k_g) \in X^+(\T)$. Set $v_{g-i} = k_i-k_{i+1}+1$ for $1\leq i\leq g-1$. Then any element
$f \in V_{\kappa,L}^{l-\an,<\underline{v}}$ is in $V_{\kappa,L}$.
\end{prop}

\begin{demo} One easily checks that any element $f \in V_{\kappa,L}^{l-\an, <\underline{v}}$ is actually analytic because the operators $\delta_i$ increase
the radius of analyticity. Using the exact sequence (\ref{exactsequence}), we need to see that $d_1.f =0$. Let $\alpha$ be the simple positive
root given by the character $(t_1,\ldots, t_g)\mapsto t_i. t_{i+1}^{-1}$. Since $\delta_{g-i} \Theta_\alpha(f) =p^{k_{i+1}-k_i-1}
\Theta_\alpha \delta_{g-i}(f)$ we see that $\Theta_{\alpha}(f) $ is a generalized eigenvector for $\delta_{g-i}$ for eigenvalues
of negative valuation. But  the norm of $\delta_{g-i}$ is less than $1$ so $\Theta_{\alpha}(f)$ has to be zero.
\end{demo}

\section{Canonical subgroups over complete discrete valuation rings}
\subsection{Existence of canonical subgroups}\label{sect--Fargues}
Let $p>2$ be a prime integer and $K$ a complete valued extension of $\qq_p$ for a valuation
$v\colon  K \rightarrow \R \cup\{\infty\}$ such that $v(p)=1$.  Let $\bar{K}$ be an algebraic closure of $K$.
We denote by $\ocal_K$ the ring of elements of $K$ having non-negative valuation and set $v \colon \ocal_K/p\ocal_K
\rightarrow [0,1]$ to be the truncated valuation defined as follows: if $x \in \ocal_K/p\ocal_K$ and $\hat{x}$ is a (any)
lift of $x$ in $\ocal_K$, set  $v(x) = \inf\{ v(\hat{x}), 1\}$.  For any $w \in v(\ocal_K) $ we set $\mfrak(w)=
\{ x \in K, v(x) \geq w\}$ and $\ocal_{K,w} = \ocal_K/\mfrak(w)$. If $M$ is an $\ocal_K$-module, then $M_w$ denotes
$M\otimes_{\ocal_K} \ocal_{K,w}$. If $M$ is a torsion  $\ocal_K$-module of finite presentation,  there is an integer
$r$ such that  $M \simeq \bigoplus_{i=1}^r \ocal_{K,a_i}$ for real numbers $a_i \in v(\ocal_{K})$. We set
$\deg M = \sum_i v(a_i)$.

Let  $H$ be a group scheme over  $\ocal_K$ and let   $\omega_H$ denote  the co-normal sheaf along the unit section of $H$. If $H$ is a finite flat group
scheme, $\omega_H$ is a torsion $\ocal_K$-module of finite presentation and the degree of $H$, denoted $\deg H$, is by definition the degree
of $\omega_H$ (see \cite{Fa} where the degree is used to define the Harder-Narasimhan filtration of finite flat group schemes).

Let $G$ be a Barsotti-Tate group over $\Spec~\ocal_K$ of dimension $g$ (for example the Barsotti-Tate
group associated to an abelian scheme of dimension $g$). Consider the $\ocal_{K,1}$-module $\mathrm{Lie}~ G[p]$. We denote by $\sigma$ the
Frobenius endomorphism of  $\ocal_{K,1}$.  The module  $\mathrm{Lie}~ G[p]$ is equipped with a $\sigma$-linear Frobenius endomorphism whose determinant,
called the Hasse invariant of $G$ is denoted $\Ha(G)$.
The Hodge height of $G$, denoted $\Hdg(G)$ is the truncated valuation of $\Ha(G)$.

Canonical subgroups have been constructed by Abbes-Mokrane, Andreatta-Gasbarri, Tian and Fargues.
In the sequel we quote mostly results of Fargues.

\begin{thm}[\cite{Far}, thm. 6]\label{thm_can_sg}  Let $n \in \N$. Assume that $\Hdg(G) < \frac{1}{2p^{n-1}}$ ( resp.~$ \frac{1}{3p^{n-1}}$ if $p=3$).
Then the  $n$-th step of the Harder-Narasimhan filtration of $G[p^n]$, denoted $H_n$ is called the canonical subgroup of level $n$ of $G$.
It enjoys the following properties.
\begin{enumerate}
\item   $H_n(\bar{K}) \simeq (\Z/p^n\Z)^g$.

\item  $\deg H_n = ng -\frac{p^n-1}{p-1} \Hdg(G)$.

\item  For any $1\leq k \leq n$, $H_n[p^k] $ is the canonical subgroup of level $k$ of $G$.
\item In $G\vert_{\Spec~\ocal_{K,1-\Hdg(G)}}$ we have that ${H_1}\vert_{\Spec~\ocal_{K,1-\Hdg(G)}}$ is the Kernel of Frobenius.
\item For any $1\leq k < n$, $\Hdg(G/H_k) = p^k\Hdg(G)$ and $H_n/H_k$ is the canonical subgroup of level $n-k$ of $G/H_k$.

\item Let $G^D$ be the dual Barsotti-Tate group of $G$. Denote by $H_n^\bot$ the annihilator of $H_n$ under the natural pairing $G[p^n]\times G^D[p^n]\to \mu_{p^n}$. Then $\Hdg(G^D)= \Hdg(G)$ and $H_n^\bot$ is the canonical subgroup of level $n$ of $G^D$.
\end{enumerate}
\end{thm}

The theorem states in particular that if the Hodge height of $G$ is small, there is a  (canonical) subgroup of high
degree  and rank $g$ inside $G[p]$. The converse is also true.

\begin{prop}\label{prop_laver} Let $H \hookrightarrow G[p]$  be a finite flat subgroup scheme of $G[p]$ of rank $g$. The following are equivalent:
\begin{enumerate}
\item $\deg H > g- \frac{1}{2}$  if $p \neq 3$ or $\deg H > g- \frac{1}{3}$ if $p=3$,
\item $\Hdg(G) < \frac{1}{2}$ if $p\neq 3$ or $\Hdg(G) < \frac{1}{3}$ if $p=3$, and $H$ is the canonical subgroup of level $1$ of  $G$.
\end{enumerate}
\end{prop}

\begin{demo} In view of theorem \ref{thm_can_sg}, we only need to show that the first point implies the second.  Set $v = g- \deg H$.  It is enough to prove that $v< \frac{1}{2}$  if $p \neq 3$ or $v<  \frac{1}{3}$ if $p=3$ implies that $\Hdg(G) \leq v$. Indeed, by theorem \ref{thm_can_sg}, $G$ will admit a canonical subgroup of level $1$, which is a step of the Harder-Narasimhan filtration of $G[p]$. On the other hand, proposition 15 of \cite{Far} shows  that $H$ is a step of the Harder-Narasimhan filtration of $G[p]$. It follows that $H$ is the canonical subgroup of level $1$ of $G$.

Let $\bar{H}$ and $\bar{G}[p]$ denote the restrictions of $H$ and $G[p]$ to $\Spec~\ocal_{K,1}$. Note that there are   canonical identifications $\omega_{\bar{G}[p]}\simeq \omega_{G[p]} \simeq \omega_{G}/p \omega_{G} $ and $\omega_{H} \simeq \omega_{\bar{H}}$.
We use the superscripts $\underline{~}^{(p)}$ to denote base change  by the Frobenius map $\sigma:\ocal_{K,1}\rightarrow \ocal_{K,1}$.
We have a functorial Vershiebung morphism $V \colon \bar{H}^{(p)} \rightarrow \bar{H}$ and $V \colon \bar{G}[p]^{(p)} \rightarrow \bar{G}[p]$.
Taking the induced map on the co-normal sheaves   we obtain the following commutative diagram with exact rows:
\begin{eqnarray*}
\xymatrix{ 0 \ar[r] & \omega_{\bar{G}[p]/\bar{H}} \ar[r] \ar[d]^{V^\ast} & \omega_{\bar{G}[p]} \ar[r]^{\phi}  \ar[d]^{V^\ast} & \omega_{\bar{H}}
\ar[r] \ar[d]^{V^\ast} & 0 \\
& \omega_{\bar{G}[p]/\bar{H}}^{(p)} \ar[r]  & \omega_{\bar{G}[p]}^{(p)} \ar[r]^{\phi \otimes 1}  & \omega_{\bar{H}}^{(p)} \ar[r]  & 0}
\end{eqnarray*}

We have an isomorphism $\ocal_{K,1}^g \simeq \omega_{G[p]}$ and $\mathrm{Ker}\phi  \subset p^{1-v} \omega_{G[p]}$ since $\deg G[p]/H = v$ by \cite{Fa}, lem. 4. As a result, there is a surjective map $\omega_{\bar{H}} \rightarrow \ocal_{K,1-v}^g$. We thus obtain a surjective map $\omega_{\bar{H}}^{(p)} = \omega_{\bar{H}}\otimes_{\ocal_{K,1},\sigma} \ocal_{K,1} \rightarrow \ocal_{K,1-v}^g \otimes_{\ocal_{K,1},\sigma} \ocal_{K,1} \simeq \ocal_{K,1}^g$.  The map $\phi\otimes 1 \colon \omega_{\bar{G}[p]}^{(p)} \rightarrow \omega_{\bar{H}}^{(p)}$ is a surjective map between two  finite $\ocal_{K,1}$-modules which are isomorphic, so it is an isomorphism. As $\Hdg(G)$ can also be  computed as the truncated valuation of the determinant of $V^\ast$ on $\omega_{\bar{G}[p]}^{(p)}$, we conclude that  $\Hdg(G) =  \deg (\omega_{\bar{G}[p]}^{(p)}/V^\ast \omega_{\bar{G}[p]}) = \deg (\omega_{\bar{H}}^{(p)}/V^\ast \omega_{\bar{H}})$. We are thus reduced to compute the map $V^\ast$  at the level of the group $\bar{H}$.

After possibly extending $K$, we can find  an increasing filtration of $H_K$  by finite flat subgroups $\{\Fil_i H_K\}_{0 \leq i \leq g}$  where $\Fil_i H_K$ has rank $p^i$. Taking schematic closures we obtain an increasing filtration of $H$  by finite flat subgroups $\{\Fil_i H\}_{0 \leq i \leq g}$  where $\Fil_i H$ has rank $p^i$.  We set $\Gr_k  H= \Fil_k  H / \Fil_{k-1}H$. This is a finite flat group scheme of order $p$ for every $k$.  We let $\{\Fil_i \bar{H}\}_{0 \leq i \leq g}$ be the filtration  of $\bar{H}$ obtained via base change to $\Spec~\ocal_{K,1}$ and $\{\Fil_i \bar{H}^{(p)}\}_{0 \leq i \leq g}$ be the filtration of $\bar{H}^{(p)}$ induced by pull-back under $\sigma$. We obtain a decreasing filtration on the differentials by setting   $\Fil^i \omega_{\bar{H}} = \mathrm{Ker} ( \omega_{\bar{H}} \rightarrow \omega_{\Fil_i \bar{H}})$. Taking differentials in the exact sequence $0 \rightarrow \Fil_{k-1} H \rightarrow \Fil_{k} H \rightarrow \Gr_k H \rightarrow 0$  provides an isomorphism $\Gr^k \omega_{\bar{H}} := \Fil^{k-1} \omega_{\bar{H}}/\Fil^{k} \omega_{\bar{H}} \simeq \omega_{\Gr_k \bar{H}}$.  Similarly, we set $\Fil^i \omega_{\bar{H}}^{(p)} = \mathrm{Ker} ( \omega_{\bar{H}}^{(p)} \rightarrow \omega_{\Fil_i \bar{H}}^{(p)})$ and there is a surjective map  $\omega_{\Gr_k \bar{H}}^{(p)} \rightarrow \Gr^k (\omega_{\bar{H}}^{(p)})$. But as before, it is easy to see that both modules are isomorphic to $\ocal_{K,1}$ and this map is an isomorphism.

The map $V^\ast$ respects these filtrations and a straightforward calculation using Oort-Tate theory shows that $\deg (\omega_{\Gr^k \bar{H}}^{(p)}/V^\ast \omega_{\Gr^k \bar{H}}) = 1 - \deg (\Gr^k {H})$. Hence,  $\deg (\omega_{\bar{H}}^{(p)}/V^\ast \omega_{\bar{H}})\leq \sum_k \deg \Gr^k \omega_{\bar{H}}^{(p)} /V^\ast \Gr^k \omega_{\bar{H}} = g -\sum_k \deg(\Gr^k {H})$.  Since $\sum_k \deg(\Gr^k {H}) = \deg {H}$, we conclude that $\Hdg(G)=\deg (\omega_{\bar{H}}^{(p)}/V^\ast \omega_{\bar{H}}) \leq g- \deg H=v$ as claimed.
\end{demo}

\subsection{The Hodge-Tate maps for  $H_n$ and $G[p^n]$}

In this section we work under the hypothesis of theorem  \ref{thm_can_sg}, i.e let us recall that $G$ was a Barsotti-Tate group of dimension
$g$ such that $v:=\Hdg(G) < \frac{1}{2p^{n-1}}$ ( resp.~$ \frac{1}{3p^{n-1}}$ if $p=3$) and we denoted $H_n\subset G[p^n]$ its level $n$ canonical subgroup.
We now define
the  Hodge-Tate map for  $H_n^D$ (viewed as a map of abelian sheaves on the $fppf$-topology):
$$ \mathrm{HT}_{H_n^D} \colon H_n^D \rightarrow \omega_{H_n},$$by sending an $S$-valued point $x\in H_n^D(S)$, i.e.,
a homomorphism  of $S$-group schemes $x\colon H_{n,S} \to \mu_{p^n,S}$, to the pull--back
$x^\ast(dt/t)\in  \omega_{H_n}(S)$ of the invariant differential $dt/t$ of $\mu_{p^n,S}$.

Following the conventions of section \ref{sect--Fargues} we write  $\omega_{G[p^n], w}$, resp.~$\omega_{H_n, w}$ for
$\omega_{G[p^n]}\otimes_{\ocal_K} \ocal_{K,w}$, resp.~$\omega_{H_n}\otimes_{\ocal_K} \ocal_{K,w}$.

 \begin{prop}\label{prop-HdgT} \begin{enumerate}

 \item The  differential of the inclusion $H_n \hookrightarrow G[p^n]$ induces an isomorphism
     $$\omega_{G[p^n], n- v \frac{p^n-1}{p-1}}  \isolong \omega_{H_n, n- v \frac{p^n-1}{p-1}}.$$
 \item The linearized Hodge-Tate map  $$\mathrm{HT}_{H_n^D} \otimes 1 \colon H_n^D(\bar{K}) \otimes_\Z
\ocal_{\bar{K}}\rightarrow \omega_{H_n}\otimes_{ \ocal_K} \ocal_{\bar{K}}$$
 has cokernel of degree $\frac{v}{p-1}$.
\end{enumerate}
\end{prop}
\begin{demo}  We have an exact sequence:

$$ 0 \rightarrow H_n \rightarrow G[p^n] \rightarrow G[p^n]/H_n \rightarrow 0$$ which induces an exact sequence:

$$ 0 \rightarrow \omega_{ G[p^n]/H_n} \rightarrow \omega_{G[p^n]} \rightarrow \omega_{H_n} \rightarrow 0$$

We know that $\omega_{G[p^n]} \simeq \ocal_{K,n}^g$ and that $\deg G[p^n]/H_n = \frac{p^n-1}{p-1} v$ so the first claim follows.

There is a  commutative diagram:
\begin{eqnarray*}
\xymatrix{  H_n^D(\bar{K})\otimes_\Z \ocal_{\bar{K}} \ar[r]^{\mathrm{HT}_{H_n^D}\otimes1} \ar[d]&  \omega_{H_n}\otimes_{\ocal_K} \ocal_{\bar{K}} \ar[d]\\
H_1^D(\bar{K})\otimes_\Z \ocal_{\bar{K}} \ar[r]^{\mathrm{HT}_{H_1^D}\otimes 1} & \omega_{H_1}\otimes_{\ocal_K} \ocal_{\bar{K}}}
\end{eqnarray*}
We know by the proof  of theorem 4 of \cite{Far} that $\HT_{H_1^D}\otimes 1$  has cokernel of degree $\frac{v}{p-1}$. The same holds for the Hodge-Tate map  $\mathrm{HT}_{H_n^D}\otimes1$.
\end{demo}

\bigskip

Although  the results of proposition \ref{prop-HdgT} are all that we need for later use, we'd like to go further and analyze the Hodge-Tate map
for the group $G[p^n]$:

$$\HT_{n} \colon  G[p^n](\ocal_{\bar{K}}) \rightarrow \omega_{G^D[p^n]}\otimes_{\ocal_K} \ocal_{\bar{K}}.$$ The following result is implicit in \cite{Far} (see also \cite{AG}, sect. 13.2  when $n=1$).

\begin{prop}\label{prop--HT} We use the notations of proposition \ref{prop-HdgT}. Assume that $v < \frac{p-1}{p(p^n-1)}$. Then the following natural sequence:
is exact.
$$0 \rightarrow  H_n(\bar{K}) \rightarrow G[p^n](\bar{K}) \stackrel{\mathrm{HT}_{G[p^n]}}{\rightarrow} \omega_{G^D[p^n]}
\otimes_{\ocal_K} \ocal_{\bar{K}}.$$

Furthermore, the cokernel of the map $\mathrm{HT}_{G[p^n]} \otimes 1\colon  G[p^n](\bar{K})\otimes_\Z \ocal_{\bar{K}} \rightarrow  \omega_{G^D[p^n]} \otimes_{\ocal_K} \ocal_{\bar{K}} $ is of degree $\frac{v}{p-1}$.
\end{prop}

\begin{demo}  An easy calculation using Oort-Tate theory shows that for any group scheme $H \rightarrow \Spec~\ocal_{\bar{K}}$ of order $p$  and degree at least $1- \frac{1}{p}$, the Hodge-Tate map $\HT_H \colon H(\bar{K}) \rightarrow \omega_{H^D}$ is  zero.
By hypothesis we have $\deg H_n \geq ng- \frac{1}{p}$ and we can thus filter $H_n$ by group schemes such that each
graded quotient is of order $p$ and has degree at least $1- \frac{1}{p}$. A straightforward d\'evissage now proves that
the map $\mathrm{HT}_{H_n} \colon H_n(\bar{K}) \rightarrow \omega_{H_n^D}\otimes_{\ocal_K} \ocal_{\bar{K}}$ is zero and so
the map $\mathrm{HT}_{G[p^n]} \colon H_n(\bar{K}) \rightarrow \omega_{G^D[p^n]}$ is also zero. The rest of the proposition follows from the proof  of theorem 4 of \cite{Far} as in proposition \ref{prop-HdgT}.
\end{demo}

\bigskip

Applying this proposition to $G^D$ and using the fact that $H_n^\bot$ is the canonical subgroup of $G^D$ we obtain a map
$$\hat{ \mathrm{HT}}\colon  H_n^D(\bar{K}) \rightarrow \omega_{G[p^n]} \otimes_{\ocal_{K}}{\ocal_{\bar{K}}},$$which is a lift of the map $\HT_{H_n^D}$.

\begin{rem}\label{rem--HT} There is a rationality issue with the map $\hat{\HT}$. If $K'$ is the finite  extension of $K$ fixed by an open
sub-group $\Gamma$ of $\mathrm{Gal}(\bar{K}/K)$ we obtain an induced map:

$$\hat{ \mathrm{HT} }\colon H_n^D(K') \rightarrow (\omega_{G[p^n]} \otimes_{\ocal_{K}}\ocal_{\bar{K}})^{\Gamma}.$$

 There is an injection $\omega_{G[p^n]} \otimes_{\ocal_{K}}\ocal_{{K'}} \hookrightarrow (\omega_{G[p^n]} \otimes_{\ocal_{K}}\ocal_{\bar{K}})^{\Gamma}$ which may be strict and there is no reason for   $\hat{\mathrm{HT}}( H_n^D(K'))$ to lie in $\omega_{G[p^n]} \otimes_{\ocal_{K}}\ocal_{{K'}}$.
 But if we reduce modulo $p^{n- v\frac{p^n-1}{p-1}}$, then $\hat{\HT}$ coincides with $\HT_{H_n^D}$ in
$\omega_{G[p^n], n- v \frac{p^n-1}{p-1}} = \omega_{H_n, n- v \frac{p^n-1}{p-1}}$.

%??Add a reference to Fargues for the last equality??%%
\end{rem}

\subsection{Canonical subgroups for semi-abelian schemes}
We will need to apply the results of the last section in the setting of semi-abelian schemes. Let $S$ be a noetherian  scheme and $U$ a dense open subset.  We will  use  the notions of $1$-motives  over $U$  and Mumford $1$-motives over
$U \hookrightarrow S$ as follows.
\begin{defi}[\cite{Del}, def. 10.1.1, \cite{Str}, def. 1.3.1]  A $1$-motive over $U$ is a complex of \emph{fppf}
abelian sheaves $[Y\rightarrow \tilde{G}] $ concentrated in degree $-1$ and $0$ where
 \begin{enumerate}
\item $\tilde{G} \rightarrow U $  is a semi-abelian scheme which  is an extension
$$ 0 \rightarrow T \rightarrow \tilde{G} \rightarrow A \rightarrow 0$$ with $T$ a torus and $A$ an abelian scheme,
\item  $Y \rightarrow U $ is an isotrivial sheaf.
\end{enumerate}

\noindent A Mumford $1$-motive over $U \hookrightarrow S$ is the data of:
\begin{enumerate}
\item A semi-abelian scheme $\tilde{G} \rightarrow S$  which is an extension
$$ 0 \rightarrow T \rightarrow \tilde{G} \rightarrow A \rightarrow 0$$ with $T$ a torus over $S$ and $A$ an abelian scheme over $S$,
\item $Y\rightarrow S$ an isotrivial sheaf,
\item $[Y_U \rightarrow \tilde{G}_U]$ a $1$-motive over $U$.
\end{enumerate}
\end{defi}

Given $M = [Y \rightarrow \tilde{G}] $ a  $1$-motive over $U$ and an integer $n$, we define the $n$-torsion of $M$ as the $\HH^{-1}$ of the cone of the multiplication by $n$ map $M \rightarrow M$. It comes with a filtration $\Fil_\bullet$. The group $\tilde{G}$ is an extension $ 0 \rightarrow T \rightarrow \tilde{G} \rightarrow A \rightarrow 0$, and    $\Fil_0 = T[n]$, $\Fil_1= \tilde{G}[n]$, $\Fil_2 =  M[n]$.

Assume that  $S = \Spec~\ocal_K$, that $U = \Spec~K$.  We say that a Mumford $1$-motive  $M = [Y \rightarrow \tilde{G}] $ over $U \hookrightarrow S$ has a canonical subgroup of level $n$ if $\tilde{G}$ has a canonical subgroup of level $n$.
% ?? Reference to \cite[\S 2.1]{AIS}?? I'm not happy about this section and rmks 4.1, 4.2: results are stated but are somehow vague and without a proof. I would just make a precise statement of the analogues of Thm 3.1 and Prop. 3.1, with the definition of the canonical subgroup and a precise reference to lemma 2.8, AIS for HT%

\section{Canonical subgroups in families and applications}
\subsection{Families of canonical subgroups}
\label{sec-familiescansg}

We start by introducing some categories of $\ocal_K$-algebras. We let $\mathbf{Adm}$ be the category of admissible
$\ocal_K$-algebras, by which we mean   flat $\ocal_K$-algebras which are  quotients of  rings of restricted power series $\ocal_K\langle X_1,\ldots,X_r\rangle$,
for some $r\ge 0$.
 We let $\mathbf{NAdm}$ be the category of normal admissible $\ocal_K$-algebras.

Let $R$ be an object of $\mathbf{Adm}$.
% A rig-ideal of $R$ will be the ideal of definition of a rig-point
%($\mathfrak{P}$ is a rig-ideal if $R/\mathfrak{P}$ is a flat normal $\ocal_K$-algebra of dimension $1$).
We have a supremum semi-norm on $R[1/p]$ denoted by $\vert \ \vert$. If $R$ is in $\mathbf{NAdm}$ then $\vert \  \vert$ is a norm and
the unit ball for this norm is precisely $R$.

For any object $R$ in $\mathbf{Adm}$, we let $R-\mathbf{Adm}$ be the category of $R$-algebras which
are admissible as $\ocal_K$-algebras. We define similarly $R-\mathbf{NAdm}$.

If $w \in v(\ocal_K)$ we set as before $R_w = R \otimes_{\ocal_K} \ocal_{K,w}$ and for any $R$-module $M$, $M_w$ means $M \otimes_R R_w$.

\medskip

Till the end of this section  we fix an object  $R$ of $\mathbf{NAdm}$.
We set $S= \Spec~R$, and $S_{\rig}$ is the rigid analytic space associated to $R[\frac{1}{p}]$. We will study the
$p$-adic properties of certain semi-abelian schemes over $S$ and their canonical subgroups. We make  the following
assumptions.

Let  $U$ be a dense open sub-scheme  of $S$ and $G$ a semi-abelian scheme over $S$ such
that $G|_U$ is abelian. We assume that there exists  $\tilde{G}$, a semi-abelian scheme over $S$ with constant toric
rank, $Y$  an isotrivial sheaf over $S$ and $M = [Y \rightarrow \tilde{G}]$ a Mumford $1$-motive over
$U \hookrightarrow S$ such that
$M[p^n] \simeq G[p^n]$ over $U$ and that $\tilde{G}[p^n] \hookrightarrow G[p^n]$.  For  $x \in S_{\rig}$ we write $\mathrm{Hdg}(x)$ for  $\mathrm{Hdg}\bigl(\tilde{G}_x[p^\infty]\bigr)$.

 \begin{rem}  The group $G[p^n]$ is not  finite flat in general (unless $G$ has constant toric rank   over $S$)
but under the hypothesis above it has a finite
flat  sub-group  $\tilde{G}[p^n]$ which we  use as a good substitute.
 \end{rem}

\begin{rem} In our applications $R$ will come from the $p$-adic completion of an \'etale affine open sub-set of the
toroidal compactification of the Siegel variety. If this open sub-set does not meet   the boundary, then the
semi-abelian scheme $G$ will be abelian and the situation is simple. On the other hand we can cover the boundary
by \'etale affine open sub-sets such that $G$ comes by approximation from a semi-abelian scheme constructed out of
a $1$-motive $M$ (by Mumford's construction). In the approximation process it is possible to preserve the
$p^n$-torsion of $M$ as explained in \cite{Str}, section 2.3.
\end{rem}

Using the previous notations and assumptions on $S$ and $G$ we now make further assumptions on the Hodge height.   Let $v < \frac{1}{2p^{n-1}}$
(resp.~$v < \frac{1}{3p^{n-1}}$ if $p=3$)  such that for any $x \in S_{\rig}$, $\mathrm{Hdg}(x) < v$.  For any point $x \in S_{\rig}$, $G$
has a canonical sub-group of order $n$. By the properties of the Harder-Narasimhan filtration  there  is a finite flat sub-group $H_{n,K} \subset
G\vert_{S_{\rig}}$ interpolating the canonical sub-groups of level $n$ for all the points $x\in S_{\rig}$.

\begin{prop} The canonical subgroup extends to a finite flat subgroup scheme $H_n \hookrightarrow G[p^n]$ over $S$.
\end{prop}
\begin{demo} We first assume $n=1$.    Let $Gr \rightarrow S=\Spec~R$ be the proper scheme which parametrizes all finite
flat subgroups of $\tilde{G}[p]$  of rank $p^g$ over $S$. We have a section $s \colon S_K \rightarrow Gr_K$ given by the canonical
subgroup.  Let $T$ be the schematic closure of $s(S_K)$ in $Gr$. The map $T \rightarrow S$ is proper. We let
$H \rightarrow T$ be the universal subgroup. We first show that $T \rightarrow S$ is finite. Let $k$ be the residue field of $\ocal_K$.  It is enough to prove that for all $x \in T_k$, $H_x$ is the kernel of the Frobenius morphism; indeed  this will imply that $T \rightarrow S$ is quasi-finite, hence finite. So let $x \in T_k$ and let $x_1 \rightsquigarrow x_2 \ldots \rightsquigarrow x$ be a sequence of immediate specializations of maximal length. Clearly, $x_1\in T_K$ since $T$ is the closure of its generic fiber. So let $x_j$ and $x_{j+1}$ be such that $x_j \in T_K$ and $x_{j+1} \in T_k$. Let $V$ be the closure of $x_j$ in $T$, $V'$ be the localization of $V$ at $x_{j+1}$ and $V''$ be the normalization of $V'$. Then $V''$ is a discrete valuation ring of mixed characteristic. So $H_{V''}$ is generically the canonical subgroup and by the general theory over discrete valuation rings (see theorem \ref{thm_can_sg}), $H_{x_k}$ is the kernel of Frobenius. As a result $H_{x}$ is the kernel of Frobenius as well. Now set $T = \Spec~B$. By construction $B$ is torsion free, and hence it is  flat, as $\ocal_K$-module. Furthermore it is a finite $R$-module  and $R_K = B_K$. Since $R$ is normal, $B=R$.

By induction, we assume that the proposition is known for $n-1$ and prove it for $n\geq 2$.   We define $H_n$ by the cartesian square (where $H_n/H_1$ is the canonical subgroup of level $n-1$ for $\tilde{G}[p^n]/H_1$ by theorem \ref{thm_can_sg}):

\begin{eqnarray*}
\xymatrix{ H_n \ar[r] \ar[d] & \tilde{G}[p^n]   \ar[d]\\
H_n/H_1 \ar[r] & \tilde{G}[p^n]/H_1}
\end{eqnarray*}
all vertical maps are finite flat. Since $H_n/H_1$ is finite flat over $S$, we are done.
\end{demo}

\subsection{The Hodge-Tate map in families}
In this paragraph we investigate the properties of the map of  $fppf$ abelian sheaves $\HT_{H_n^D} \colon H_n^D \rightarrow \omega_{H_n}$.
We work using the notations and assumptions of section \S \ref{sec-familiescansg}.

\begin{prop}\label{prop:Filomega1}  Let $w \in v(\ocal_K)$ with $w<n-v\frac{p^n-1}{p-1}$. The morphism of coherent
sheaves  $\omega_{G} \rightarrow \omega_{H_n}$ induces an isomorphism $\omega_{G,w} \rightarrow \omega_{H_n,w}$.
\end{prop}

\begin{demo}  Possibly after replacing $R$ with an open affine formal covering, we may assume that $\omega_G$ is a free $R$-module. Fix an isomorphism $\omega_G\cong R^g$. Consider the surjective map $\alpha\colon R^g\cong \omega_{G} \rightarrow \omega_{H_n,w}$ given by the inclusion $H_n\subset G$. It suffices to show that any element $(x_1,\ldots,x_g)\in {\rm Ker}(\alpha)$ satisfies $x_i\in p^w R$ for every $i=1,\ldots,g$. As $R$ is normal, it suffices to show that for every codimension $1$ prime ideal  $\mathfrak{P}$ of $R$  containing  $(p)$ we have $x_i\in p^w R_{\mathfrak{P}}$ or equivalently $x_i\in p^w \widehat{R}_{\mathfrak{P}}$. Here, $R_{\mathfrak{P}}$ is a discrete valuation ring of mixed characteristic and $\widehat{R}_{\mathfrak{P}}$ is its $p$-adic completion. We are then reduced to prove the claim over a complete dvr and this is the content of proposition \ref{prop-HdgT}.
\end{demo}

\begin{prop}\label{prop-HTann} Assume that $H_n^D(R) \simeq (\Z/p^n\Z)^g$. The cokernel of the map

$$ \HT_{H_n^D} \otimes 1 \colon H_n^D(R) \otimes_\Z R \rightarrow \omega_{H_n}$$

is killed by $p^{\frac{v}{p-1}}$.

\end{prop}

\begin{demo}  Possibly after localization on $R$, we may assume that $\omega_G$ is a free $R$-module of rank $g$. We have a surjection $R^g \simeq \omega_{G} \rightarrow \omega_{H_n}$. If we fix a basis of $H_n^D(R)$, we also have a surjection $ R^g \rightarrow H_n^D(R) \otimes_\Z R \simeq R_n^g$.  In these presentations, the map $\HT_{H_n^D} \otimes 1$ is given by a matrix $\gamma \in \mathrm{M}_{g}(R)$.    Let ${d}\in R$ be the determinant of the matrix  $\gamma$ . Then, $d$ annihilates the cokernel of $\gamma$. It suffices to prove that $p^{\frac{v}{p-1}}\in d R$. As $R$ is normal, it suffices to prove that $p^{\frac{v}{p-1}}\in d R_{\mathfrak{P}}$ for every codimension $1$ prime ideal $\mathfrak{P}$ of $R$ containing $p$. It follows from proposition \ref{prop-HdgT} that $p^{\frac{v}{p-1}}\in d R_{\mathfrak{P},n-v\frac{p^n-1}{p-1}}$. As $\frac{v}{p-1}<n-v\frac{p^n-1}{p-1}$, we conclude that $p^{\frac{v}{p-1}} \in d R_{\mathfrak{P}}$ as wanted.
\end{demo}

\subsection{The locally free sheaf  $\mathcal{F}$}
\label{sec-sheafF}

We work in the hypothesis of section \S \ref{sec-familiescansg}, i.e. let us recall that
we have fixed $R\in \mathbf{NAdm}$ and a semi-abelian scheme $G$ over $S:=\mathrm{Spec}(R)$
such that the restriction of $G$ to a dense open sub-scheme $U$ of $S$ is abelian.
We also fix  a rational number $v$ such that $v < \frac{1}{2p^{n-1}}$
(resp.~$v < \frac{1}{3p^{n-1}}$ if $p=3$)  with the property that for any $x \in S_{\rig}$, $\mathrm{Hdg}(x) < v$.
Here $\mathrm{Hdg}(x):=\mathrm{Hdg}\bigl(G_x[p^{\infty}]\bigr)$. Let $H_n$ denote the canonical
subgroup of $G$ of level $n$ over $S$.
From now on, we also assume that $H_n^D(R) \simeq (\Z/p\Z)^g$. We then have the following fundamental proposition.

\begin{prop}\label{prop-FHT} There is a  free sub-sheaf of $R$-modules $\mathcal{F}$  of
$\omega_{G}$ of rank $g$ containing $p^{\frac{v}{p-1}} \omega_{G}$  which is equipped,
for all $w \in ]0, n-v\frac{p^n}{p-1} ]$, with a map

$$ \mathrm{HT}_{w}\colon    H_n^D(R[1/p]) \rightarrow \mathcal{F}\otimes_{R} R_w$$ deduced from $\HT_{H_n^D}$
which induces  an isomorphism:

$$ \mathrm{HT}_{w} \otimes 1\colon    H_n^D(R[1/p])\otimes_\Z R_{w} \rightarrow \mathcal{F}\otimes_{R} R_w.$$

%(2) Assume that we have an ascending filtration $\Fil_i H_n^D\subset H_n^D$ by subgroup schemes such that  $\Fil^i H_n^D(R[1/p])=\oplus_{j=1}^i (\Z/p^n\Z)x_j$. If $\Fil_\bullet \omega_G$ is a $w$-compatible filtration, see definition \ref{defi:wcompfilt}. Then $\Fil_\bullet \mathcal{F}:=\Fil_\bullet \omega_G \cap \mathcal{F}$ is an ascending filtration  by free $R$--submodules such that $\Gr_i \mathcal{F} $ is free of rank $1$  and $ \mathrm{HT}_{w}$ induce isomorphisms:

%$$ \mathrm{HT}_{w} \otimes 1\colon    \Fil_i H_n^D(R[1/p])\otimes_\Z R_{w} \rightarrow \Fil_i \mathcal{F}\otimes_{R} R_w.$$

\end{prop}

\begin{demo}  Set $w_0 = n-v\frac{p^n-1}{p-1}$. Let $x_1,\ldots,x_g$ be a $\Z/p^n\Z$-basis of $H_n^D(R[1/p])$. Let $\tilde{{\HT}}_{H_n^D}(x_i)$ be
lifts to $\omega_{G}$ of ${\HT}_{H_n^D}(x_i) \in \omega_{H_n}$. We set $\mathcal{F}$ to be the sub-module of $\omega_G$ generated by $$\{\tilde{{\HT}}_{H_n^D}(x_1),\ldots,\tilde{{\HT}}_{H_n^D}(x_g) \}.$$ This module is free of rank $g$. Indeed, let $\sum_{i=1}^g  \lambda_i \tilde{{\HT}}_{H_n^D}(x_i) = 0 $ be a
non-zero relation with coefficients in $R$. We may assume that there is an index $i_0$ such that $\lambda_{i_0} \notin p^{{w_0}-\frac{v}{p-1}} R$. Projecting this relation in $\omega_{G,{w_0}} = \omega_{H_n,w_0}$ (see proposition \ref{prop:Filomega1}) we  contradict the proposition \ref{prop-HTann}. By  proposition \ref{prop-HTann}, $p^{\frac{v}{p-1}} \omega_{G} \subset \mathcal{F}$ and the  module $\mathcal{F}$ is independent of the choice of a particular  lifts $\tilde{\HT}_{H_n^D}(x_i)$. Let $r \colon \omega_{H_n} \rightarrow \omega_{G,w_0}$ denote the projection.  The map ${\HT}_{H_n^D}\circ r  $ factors through $\mathcal{F}/\mathcal{F} \cap p^{w_0} \omega_G$. For all $w \in ]0, n-v\frac{p^n}{p-1}]$, we have $\mathcal{F} \cap p^{w_0}  \omega_G \subset p^w \mathcal{F}$.  We can thus define $\HT_w$ as the composite of ${\HT}_{H_n^D}\circ r$ and the projection  $ \mathcal{F}/\mathcal{F} \cap p^{w_0} \omega_G \rightarrow \mathcal{F}/p^{w}\mathcal{F}$.  Finally, the last claim follows because the map $\HT_w \otimes 1$ is a surjective map between two free modules of rank $g$ over $R_w$ and so has to be an isomorphism.
\end{demo}

\begin{rem}\label{rmk-FHT}  The sheaf $\mathcal{F}$ is  independent   of $n \geq 1$, it is functorial in $R$ and  it coincides with the
sheaf constructed using $p$-adic Hodge theory in  \cite{AIS}, prop. 2.6. where it was  denoted $F_0$.
\end{rem}
\begin{rem} Let $\Omega$ be an algebraic closure of ${\rm Frac}(R)$. Let $\bar{R}$ be the inductive limit of all finite,
\'etale  $R$-algebras contained in $\Omega$ and let  $\widehat{\bar{R}}$ denote its $p$-adic completion. Assume
that $G$ is ordinary. Let $H_{\infty} \subset  G$ be the
canonical subgroup of order ``$\infty$" and  $T_p(H_\infty^D)(\widehat{\bar{R}})$   be the  Tate module of its dual $H_\infty^D$.
We  have an isomorphism:
$$ \HT_{H_{\infty}^D} \otimes 1 \colon T_p(H_\infty^D)\bigl(\widehat{\bar{R}}\bigr)\otimes_\Z \widehat{\bar{R}} \rightarrow \omega_{G}
\otimes_R \widehat{\bar{R}}.$$
The proposition \ref{prop-FHT} is a good substitute for this isomorphism in the non ordinary case.
\end{rem}

\subsection{Functoriality in $G$}\label{sect-functo}

We assume the hypothesis of section \S \ref{sec-sheafF}.

 Moreover we suppose that we have an isogeny $\phi\colon  G \rightarrow G'$ over $S$, where $G'$ is a second semi-abelian scheme over $S$
satisfying the same assumptions as $G$,  i.e. for all $x \in S_{\rig}$, $\Hdg(G_x), \Hdg(G'_x) \leq v$.

By functoriality of the Harder-Narasimhan filtration the isogeny induces a map
$\phi\colon  H_n \rightarrow H'_n$ where $H_n$ and $H'_n$ are the canonical subgroups of level $n$ of $G$ and $G'$.

We  assume further that $H_n^D(R[1/p]) \simeq (\Z/p^n\Z)^g$ and that $ {H_n'}^D(R[1/p]) \simeq (\Z/p^n\Z)^g$.  We let $\mathcal{F}$ and $\mathcal{F'}$ be the sub-sheaves of $\omega_{G}$ and $\omega_{G'}$ constructed in proposition \ref{prop-FHT}.

\begin{prop}\label{prop-compatt}  Let $w \in ]0, n- v \frac{p^n}{p-1}]$. The isogeny $\phi$ gives rise to the following diagram:
\begin{eqnarray*}
\xymatrix{ \omega_{G'} \ar[r]^{\phi^\ast} & \omega_{G} \\
\mathcal{F'} \ar[r]\ar[u]\ar[d] & \mathcal{F} \ar[u] \ar[d] \\
\mathcal{F'}/p^w\mathcal{F'} \ar[r] & \mathcal{F}/p^w\mathcal{F} \\
{H'_n}^D(R[1/p]) \ar[r]^{\phi^D} \ar[u]^{\HT_w} & H_n^D(R[1/p]) \ar[u]^{\HT_w} }
\end{eqnarray*}
\end{prop}
\begin{demo}  Set $w_0 = n- v \frac{p^n-1}{p-1}$. We  check that $\phi^\ast( \mathcal{F}') \subset \mathcal{F}$. Let $\omega \in \omega_{G'}$ such that $\omega \mod p^{w_0}$ belongs to the $R$-span $\langle\HT_{w_0}( {H'_n}^D(R[1/p]) \rangle$.
Then $\phi^\ast \omega \mod p^{w_0}$ belongs to the $R$-span $\langle\HT_{w_0}(\phi^D{H'_n}^D(R[1/p]) \rangle$. The rest now follows easily.
\end{demo}

\subsection{The main construction}\label{sect-main}
In this section we work in the hypothesis of section \S \ref{sec-sheafF} and
  we make the further assumptions that $v < \frac{1}{2p^{n-1}}$ (resp.~$v< \frac{1}{3p^{n-1}}$ if $p=3$),
that $H_n^D(R[1/p]) \simeq (\Z/p^n\Z)^g$ and that $w \in ]0, n- v \frac{p^n}{p-1}]$.

Let $\mathcal{GR}_{\mathcal{F}} \rightarrow S$ be the Grassmannian parametrizing all flags
$\Fil_0 \mathcal{F} = 0 \subset \Fil_1 \mathcal{F} \ldots \subset \Fil_g \mathcal{F} = \mathcal{F}$ of the free module
$\mathcal{F}$; see \cite[\S I.1.7]{Ko} for the construction.  Let $\mathcal{GR}^+_{\mathcal{F}}$ be the $\T$-torsor over $\mathcal{GR}_{\mathcal{F}}$ which parametrizes flags $\Fil_\bullet \mathcal{F}$ together with basis $\omega_i$ of the graded pieces $\Gr_i \mathcal{F}$.

We fix an isomorphism $ \psi\colon (\Z/p^n\Z)^g  \simeq H_n^D(R[1/p])$ and call $x_1,\ldots,x_g$ the $\Z/p^n\Z$-basis of $H_n^D(R[1/p])$ corresponding to the canonical basis of  $(\Z/p^n\Z)^g$. Out of $\psi$, we obtain a flag $\Fil_\bullet^{\psi}  =
\{ 0 \subset \langle x_1\rangle \subset \langle x_1,x_2\rangle \ldots \subset \langle x_1,\ldots,x_g\rangle =
H_n^D(R[1/p]) \}$.  We also have a basis $x_i \mod \Fil_{i-1}^\psi$ of the graded piece $\Gr_i^\psi$.

Let $R'$ be an object in $R-\mathbf{Adm}$. We say that an element  $\Fil_\bullet \mathcal{F}\otimes_R R' \in \mathcal{GR}_{\mathcal{F}}(R')$ is $w$-compatible with $\psi$ if
$ \Fil_\bullet \mathcal{F} \otimes_R R'_w = \HT_w ( \Fil_\bullet^{\psi}) \otimes_\Z R'_w$.

We say that an element  $(\Fil_\bullet \mathcal{F}\otimes_R R', \{w_i\}) \in \mathcal{GR}^+_{\mathcal{F}}(R')$ is
$w$-compatible with $\psi$ if
$ \Fil_\bullet \mathcal{F} \otimes_R R'_w = \HT_w ( \Fil_\bullet^{\psi}) \otimes_\Z R'_w$ and
$w_i \mod p^w\mathcal{F} \otimes_R R' + \Fil_{i-1} \mathcal{F} \otimes_R R' = \HT_w (x_i \mod \Fil_{i-1}^\psi)$.

We now define  functors
\begin{eqnarray*} \mathfrak{IW}_w: R-\mathbf{Adm} &\rightarrow& SET \\
R' & \mapsto & \big \{ w-\textrm{compatible}~\Fil_\bullet \mathcal{F} \otimes_R R' \in \mathcal{GR}_{\mathcal{F}}(R')\}\\
\mathfrak{IW}^+_w: R-\mathbf{Adm} &\rightarrow& SET \\
R' & \mapsto & \big \{ w-\textrm{compatible}~(\Fil_\bullet \mathcal{F} \otimes_R R', \{w_i\}) \in \mathcal{GR}_{\mathcal{F}}^+(R')\}
\end{eqnarray*}

These two functors are representable by affine formal schemes which can be described as follows. Let $f_1,\ldots,f_g$ be an
$R$-basis of $\mathcal{F}$ lifting the vectors $\HT_w(x_1),\ldots,\HT_w(x_g)$.

The given basis identifies $\mathcal{GR}_{\mathcal{F}}$ with $\mathrm{GL}_g/\B\times S$ and $\mathfrak{IW}_w$   with the set of matrices:
   $$  \begin{pmatrix} % or pmatrix or bmatrix or Bmatrix or ...
       1& 0 &\ldots & 0\\
      p^w \mathfrak{B}(0,1) & 1 & \ldots& 0 \\
   \vdots & \vdots & \ddots & \vdots \\
     p^w\mathfrak{B}(0,1) & p^w\mathfrak{B}(0,1) &\ldots& 1\\

   \end{pmatrix} \times_{\Spf~\ocal_K} \Spf~R$$
where we have denoted by $\mathfrak{B} (0,1) = \Spf~\ocal_K\langle X\rangle$ the formal unit ball.

Similarly, the given basis identifies $\mathcal{GR}^+_{\mathcal{F}}$ with $\mathrm{GL}_g/\U \times S$ and
$\mathfrak{IW}^+_w$   with the set of matrices:
   $$  \begin{pmatrix} % or pmatrix or bmatrix or Bmatrix or ...
       1 + p^w \mathfrak{B}(0,1)& 0 &\ldots & 0\\
      p^w \mathfrak{B}(0,1) & 1 + p^w \mathfrak{B}(0,1)& \ldots& 0 \\
   \vdots & \vdots & \ddots & \vdots \\
     p^w\mathfrak{B}(0,1) & p^w\mathfrak{B}(0,1) &\ldots& 1 + p^w \mathfrak{B}(0,1)\\

   \end{pmatrix} \times_{\Spf~\ocal_K} \Spf~R.$$

We let $\mathfrak{T} \rightarrow \Spf~\ocal_K$ be the formal completion of $\T$ along its special fiber.
Let $\mathfrak{T}_w$ be the formal torus  defined by
$$\mathfrak{T}_w (R') = \mathrm{Ker} \big( \T(R') \rightarrow \T( R'/p^wR') \big) $$ for any object $R' \in \mathbf{Adm}$.  The formal scheme $\mathfrak{IW}^+_w$ is a torsor over  $\mathfrak{IW}_w$ under $\mathfrak{T}_w$.

All these constructions are functorial in $R$. They do not depend on $n$ but only on $w$.
We denote by $\mathcal{IW}_w$ and $\mathcal{IW}_w^+$ the rigid analytic generic fibers of these formal schemes.
They are admissible opens of the rigid spaces associated to $\mathcal{GR}_{\mathcal{F}}$ and $\mathcal{GR}^+_{\mathcal{F}}$
respectively.

\section{The overconvergent modular sheaves}\label{sect-ovmods}

\subsection{Classical Siegel modular schemes and modular  forms}
\label{sec-classicalmf}

We fix an integer $N \ge 3$  such that $(p,N)=1$. Recall that $K$ denotes  a finite extension of  $\qq_p$,
$\ocal_K$ its ring of integers and $k$ its residue field. The valuation $v$ of $K$ is normalized such that $v(p)=1$.

{\em The Siegel variety of prime to $p$ level.}
Let $Y$ be the moduli space of principally polarized abelian
schemes $(A,\lambda)$ of dimension $g$ equipped with a  principal level $N$ structure $\psi_N$ over
$\Spec~\ocal_K$. Let $X$ be a toroidal compactification of $Y$ and $G \rightarrow X$ be the semi-abelian scheme
extending the universal abelian scheme (see \cite{FC}).

{\em The Siegel variety of Iwahori level.}
Let $Y_{\mathrm{Iw}} \rightarrow \Spec~\ocal_K$ be the moduli space parametrizing principally polarized abelian schemes $(A,\lambda)$ of dimension $g$, equipped with a level $N$ structure $\psi_N$ and an Iwahori structure at $p$: this is the data of a full  flag $\Fil_\bullet A[p]$ of the group $A[p]$ satisfying $\Fil_\bullet^\bot = \Fil_{2g-\bullet}$. Let   $\XI$ be a toroidal compactification of this moduli space (see \cite{Str}). We choose the polyhedral decompositions occurring in the constructions of   $X$ and $\XI$ in such a way that the forgetful map $Y_{\mathrm{Iw}} \rightarrow Y$ extends to a   map $ \XI \rightarrow X$.

{\em The classical modular sheaves.}
Let $\omega_G$ be the co-normal sheaf of $G$ along its unit section,
$\mathcal{T} =\mathrm{Hom}_X ( \oscr_{X}^g, \omega_G)$ be the space of $\omega_G$ and $\mathcal{T}^\times = \mathrm{Isom}_X ( \oscr_{X}^g, \omega_G)$ be the $\GL$-torsor of trivializations of $\omega_G$. We define a left action  $\GL \times \mathcal{T} \rightarrow \mathcal{T}$ by sending $\omega\colon  \oscr_{X}^g \rightarrow \omega_G$ to $\omega\circ h^{-1}$ for any $h \in \GL$.

We define an automorphism $\kappa \mapsto \kappa'$ of $X(\T)$ by sending any $\kappa = (k_1,\ldots,k_g) \in X(\T)$ to $\kappa'=(-k_g, - k_{g-1},\ldots, -k_1) \in X(\T)$. This automorphism stabilizes the dominant cone $X^+(\T)$. Let $\pi\colon  \mathcal{T}^\times \rightarrow X$ be the projection. For any $\kappa \in X^+(\T)$, we let $\omega^{\kappa}$ be the sub-sheaf of $\pi_\ast\oscr_{\mathcal{T}^\times}$ of $\kappa'$-equivariant functions for the action of $\B$ (with $\GL$ acting on the left on $\pi_\ast\oscr_{\mathcal{T}^\times}$ by $f (\omega) \mapsto f (\omega g)$ for any section $f$ of  $ \pi_\ast\oscr_{\mathcal{T}^\times} $ viewed has a function over the trivializations $\omega$, and any $g \in \GL$). The global sections $\HH^0(X,\omega^\kappa)$ form the module of Siegel modular forms of weight $\kappa$ over $X$.

\subsection{Application of the main construction: the sheaves $\omega^{\dagger,\kappa}_w$}\label{sec:blowup}
We denote by $\mathfrak{X}$ the formal scheme obtained by completing $X$ along its special fiber $X_{k}$ and by $X_{\rig}$ the associated rigid space. We have a Hodge height  function $\Hdg\colon  X_{\rig} \rightarrow [0,1]$
(see section \ref{sect--Fargues}). Let  $v \in [0,1]$, we set  $\cX(v) = \{ x\in X^{\rig},~\Hdg (x) \leq v\}$, this is an open subset of $X_{\rig}$.
Let  $v \in v(\ocal_K)$. Consider the blow-up $\tilde{\mathfrak{X}}(v)={\bf Proj} \oscr_{\mathfrak{X}}[X,Y]/\bigl(\Ha X+ p^v Y\bigr)$ of $\mathfrak{X}$ along the ideal $(\Ha, p^v)$.  Let $\mathfrak{X}(v)$ be the $p$-adic completion of the normalization of the  greatest open formal
sub-scheme of $\tilde{\mathfrak{X}}(v)$ where the ideal $(\Ha, p^v)$ is generated by $\Ha$. This is a formal model
of $\cX(v)$.

%We denote by $\mathfrak{X}_{\mathrm{Iw}}$ the formal completion of $\XI$ along ${\XI}_k$ and by ${\XI}_{\rig}$ the associated rigid space.

Let $n \in \N_{>0}$ and  $v < \frac{1}{2p^{n-1}} \in v(\ocal_K)$ (rep. $v< \frac{1}{3p^{n-1}} \in v(\ocal_K)$ if $p=3$). We have a  canonical subgroup $H_n$ of level $n$ over $\cX(v)$. Let $\cX_1(p^n)(v) = \mathrm{Isom}_{\cX(v)} ( (\Z/p^n\Z)^g,H_n^D) $ be the  finite \'etale cover of $\cX(v)$ parametrizing trivializations of $H_n^D$. We let $\psi$ be the universal trivialization over $\cX(v)$. Let $\mathfrak{X}_1(p^n)(v)$ be the normalization of $\mathfrak{X}(v)$ in $\cX_1(p^n)(v)$. The group $\GL(\Z/p^n\Z)$ acts on $\mathfrak{X}_1(p^n)(v)$. We let $\mathfrak{X}_{\mathrm{Iw}}(p^n)(v)$ be the quotient  $\mathfrak{X}_1(p^n)(v)/ \B(\Z/p^n\Z)$. It is also the normalization of $\mathfrak{X}(v)$ in $\cX_1(p^n)(v)/\B(\Z/p^n\Z)$.  We also denote by $\mathfrak{X}_{\mathrm{Iw}^+}(p^n)(v)$ the quotient $\mathfrak{X}_1(p^n)(v)/ \U(\Z/p^n\Z)$.

%When $n=1$, we obtain a formal scheme   $\mathfrak{X}_{\mathrm{Iw}}(p)(v)$ whose generic fiber we denote $\cXI(p)(v)$ is an open subset of ${\XI}_{\rig}$.

\subsubsection{Modular properties} The formal schemes $\mathfrak{X}_1(p^n)(v)$ and $\mathfrak{X}_{\mathrm{Iw}}(p^n)(v)$
have nice modular interpretations away from the boundary.  Let $\mathfrak{Y}_1(p^n)(v)$ and  $\mathfrak{Y}_{\mathrm{Iw}}(p^n)(v)$ be the open formal sub-schemes that are the complements of the boundaries in $\mathfrak{X}_1(p^n)(v)$ and $\mathfrak{X}_{\mathrm{Iw}}(p^n)(v)$ respectively.

\begin{prop2} For any object $R \in \mathbf{NAdm}$,
\begin{enumerate}
\item $\mathfrak{Y}_1(p^n)(v) (R)$ is the set of isomorphism classes of quadruples $(A,\lambda, \psi_N, \psi)$ where $(A \rightarrow \Spf~R, \lambda)$ is a  principally polarized  formal abelian scheme of dimension $g$ such that for all rig-point $x$ in $R$, we have $\Hdg\bigl(A_x[p^\infty]\bigr) \leq v$; $\psi_N$ is a principal level $N$ structure; $\psi\colon  \Z/p^n\Z^g \rightarrow H_n^D$ is a trivialization of the dual canonical subgroup of level $n$ over $R[1/p]$.
\item  $\mathfrak{Y}_{\mathrm{Iw}}(p^n)(v) (R)$ is the set of isomorphism classes of quadruples $(A,\lambda, \psi_N, \Fil_\bullet)$ where $(A \rightarrow \Spf~R, \lambda)$ is a  principally polarized formal abelian scheme of dimension $g$ such that for all rig-point $x$ in $R$, we have $\Hdg\bigl(A_x[p^\infty]\bigr) \leq v$; $\psi_N$ is a principal level $N$ structure; $\Fil_\bullet$ is a full flag of locally free $\Z/p^n\Z$-modules of  the dual canonical subgroup of level $n$ over $R[1/p]$.
\end{enumerate}
\end{prop2}
\begin{demo} The proof is similar to the proof of  \cite{AIS}, lemma 3.1.
\end{demo}

\subsubsection{The modular sheaves $\omega^{\dagger,\kappa}_w$}\label{sec:ThesheafcalF} Let $ w \in v(\ocal_K) \cap ] n-1 + \frac{v}{p-1}, n-v \frac{p^n}{p-1}]$.
By proposition \ref{prop-FHT} there is a rank $g$ locally free sub-sheaf  $\mathcal{F}$ of $\omega_{G}/_{\mathfrak{X}_1(p^n)(v)}$. It is equipped with an  isomorphism:

$$ (\HT_w \circ \psi ) \otimes 1\colon   (\Z/p^n\Z)^g \otimes_\Z \oscr_{\mathfrak{X}_1(p^n)(v)}/p^w \oscr_{\mathfrak{X}_1(p^n)(v)}\simeq \mathcal{F}
\otimes_{\ocal_K} \ocal_{K,w}.$$

\begin{rem2} The  hypothesis  $w\in ] n-1 + \frac{v}{p-1}, n-v \frac{p^n}{p-1}]$ is motivated by proposition
\ref{prop-rig-int}. All this paragraph would make sense under the hypothesis $0 < w  < n- v\frac{p^n}{p-1}$
but in this way we
normalize $n$ and our construction only depends on $\kappa$, $w$ and $v$.  Remark  that if  $w < n-1 - v \frac{p^{n-1}}{p-1}$ we could use $\mathfrak{X}_1(p^{n-1})(v)$ as a base.
\end{rem2}

By section \ref{sect-main} we have  a chain of  formal schemes:

$$ \mathfrak{IW}^+_w  \stackrel{\pi_1}{\rightarrow} \mathfrak{IW}_w  \stackrel{\pi_2} {\rightarrow} \mathfrak{X}_1(p^n)(v) \stackrel{\pi_3}{\rightarrow} \mathfrak{X}_{\mathrm{Iw}}(p^n)(v) \stackrel{\pi_4} \rightarrow \mathfrak{X}_{\mathrm{Iw}}(p)(v).$$

We recall that $ \mathfrak{IW}_w$ parametrizes flags in the locally free sheaf $\mathcal{F}$  that are $w$-compatible with $\psi$ and that $\mathfrak{IW}^+_w$
 parametrizes  flags and bases of the graded pieces that are $w$-compatible with $\psi$.

We recall that $\mathfrak{IW}^+_w$ is a torsor over  $\mathfrak{IW}_w$ under the formal torus $\mathfrak{T}_w$. We also have an action of the group $\B(\Z/p^n\Z)$ on  $\mathfrak{X}_1(p^n)(v)$ over $\mathfrak{X}_{\mathrm{Iw}}(p^n)(v)$.  We let $\mathfrak{B}_w$ be the formal group defined by

$$ \mathfrak{B}_w (R) = \mathrm{Ker} \big(\B(R) \rightarrow \B(R/p^wR)\big) $$ for all
$R \in \mathbf{Adm}$.

There is a surjective map $\mathfrak{B}_w \rightarrow \mathfrak{T}_w$ with kernel the ``unipotent radical" $\mathfrak{U}_w$. All these actions fit together in an action of $B(\Z_p)\mathfrak{B}_w$ on $\mathfrak{IW}_w^+$ over $\mathfrak{X}_{\mathrm{Iw}}(p^n)(v)$ (the unipotent radical $\mathfrak{U}_w$ acts trivially).

The morphisms $\pi_1, \pi_2, \pi_3$ and $\pi_4$ are affine. Set $\pi = \pi_1 \circ \pi_2 \circ \pi_3 \circ \pi_4$. Let $\kappa \in \mathcal{W}(K)$ be a $w$-analytic character. The involution $\kappa \mapsto \kappa'$ of $X(\T)$ extends to an involution of $\mathcal{W}$, mapping $w$-analytic characters to $w$-analytic characters. The character $\kappa'\colon  \T(\Z_p) \rightarrow \ocal_K^\times$ extends  to a character $\kappa'\colon  \T(\Z_p)\mathfrak{T}_w \rightarrow \widehat{\mathbb{G}_m}\ocal_K^\times$ and to a character $\kappa'\colon  \B(\Z_p)\mathfrak{B}_w \rightarrow \widehat{\mathbb{G}_m}$ with $\U(\Z_p) \mathfrak{U}_w$ acting trivially.

Using the notion of formal Banach sheaf given in definition \ref{def:formalBanach}, we give the following:

\begin{defi2}\label{def:formalintegralsheaf} The formal Banach sheaf of $w$-analytic, $v$-overconvergent modular forms of weight $\kappa$ is
$$\mathfrak{w}^{\dag\kappa}_w = \pi_\ast \oscr_{\mathfrak{IW}^+_w}[\kappa'].$$
\end{defi2}

\subsubsection{Integral overconvergent modular forms}
\begin{defi2}The space of  integral $w$-analytic, $v$-overconvergent modular forms of  genus $g$,
weight $\kappa$, principal level $N$    is
$$ \M^{\dag\kappa}_w(\mathfrak{X}_{\mathrm{Iw}}(p)(v)) = \HH^0(\mathfrak{X}_{\mathrm{Iw}}(p)(v), \mathfrak{w}_w^{\dag\kappa})$$
\end{defi2}

An element $f$ of the module $$\HH^0(\mathfrak{Y}_{\mathrm{Iw}}(p)(v), \mathfrak{w}_w^{\dag\kappa}),$$
called {\em weakly modular form} is a  rule which
associates functorially to an object $R$ in $\mathbf{NAdm}$, a quadruple $(A,\lambda,
\psi_N, \psi) \in \mathfrak{Y}_{1}(p^n)(v)(R)$, a $w$-compatible flag $\Fil_\bullet \mathcal{F}_R$ and a
$w$-compatible basis $w_i$ of each $\Gr_i \mathcal{F}_R$ an element

$$ f ( R, A,\lambda, \psi_N, \psi, Fil_\bullet \mathcal{F}_R, \{\omega_i\} ) \in R$$

satisfying the functional equation:
$$ f ( R, A,\lambda, \psi_N, b.\psi, Fil_\bullet \mathcal{F}_R,  b\cdot\{  \omega_i\} ) = \kappa'(b) f ( R, A,\lambda, \psi_N, \psi, Fil_\bullet \mathcal{F}_R, \{  w_i\} )$$
for all $b \in \B(\Z_p)\mathfrak{B}_w(R).$

\subsubsection{Locally analytic overconvergent modular forms}

Let $\kappa \in \mathcal{W}(K)$ and $v, w >0$. Suppose that  $\kappa$ is $w$-analytic. If there is $n \in \N$ satisfying
$ v < \frac{1}{2p^{n-1}}$ (resp.~$v< \frac{1}{3p^{n-1}}$ if $p=3$) and $w \in ]n -1 + \frac{v}{p-1}, n - v\frac{p^n}{p-1}]$, then we have
constructed a sheaf $\mathfrak{w}^{\dag \kappa}_w$ on $\mathfrak{X}_{\mathrm{Iw}}(v)$.
As a result for all $\kappa$, if we take  $v>0$ sufficiently small and  $w \notin \N$ big enough (so that $\kappa$ is $w$-analytic),
there is  a unique $n \in \N$ satisfying the conditions above and so the modular  sheaf $\mathfrak{w}^{\dag \kappa}_w$ exists.
Let $\kappa, n, v, w$ satisfying all the required conditions for the existence of the sheaf $\mathfrak{w}^{\dag \kappa}_w$. Clearly, if $v' < v$
then $\kappa, n , v', w$ satisfy also the conditions and the sheaf $\mathfrak{w}^{\dag \kappa}_w$
on
$\mathfrak{X}_{\mathrm{Iw}}(v')$ is the restriction of the sheaf on $\mathfrak{X}_{\mathrm{Iw}}(v)$.

 If $\kappa$ is $w$-analytic, it is also $w'$ analytic for any $w' > w$. Let $n' \in \N$ and $v > 0$ such that $\kappa, n, v, w$ and
$\kappa, n', v, w'$ satisfy the conditions, so that we have two sheaves $\mathfrak{w}^{\dag \kappa}_w$
and $\mathfrak{w}^{\dag \kappa}_{w'}$ over $\mathfrak{X}_{\mathrm{Iw}}(v)$.

 There is a natural inclusion:

 $$\mathfrak{IW}^+_{w'} \hookrightarrow \mathfrak{IW}^+_w \times_{ \mathfrak{X}_1(p^n)(v)} \mathfrak{X}_1(p^{n'})(v)$$

 which follows from the fact that $w'$-compatibility implies $w$-compatibility.

This induces a natural  map   $\mathfrak{w}^{\dag \kappa}_w \rightarrow \mathfrak{w}^{\dag \kappa}_{w'}$ and thus a map
$\mathrm{M}^{\dag\kappa}_{w}(\mathfrak{X}_{\mathrm{Iw}}(v)) \rightarrow \mathrm{M}^{\dag\kappa}_{w'}
(\mathfrak{X}_{\mathrm{Iw}}(v))$. We are  led to the following definition.

\begin{defi2} Let $\kappa \in \mathcal{W}$. The space of integral locally analytic overconvergent modular forms of
weight $\kappa$ and principal level $N$ is the inductive limit:
$$\mathrm{M}^{\dag\kappa}(\mathfrak{X}_{\mathrm{Iw}}(v)) = \lim_{v \rightarrow 0, w \rightarrow \infty} \mathrm{M}^{\dag\kappa}_{w}(\mathfrak{X}_{\mathrm{Iw}}(p)(v)).$$
\end{defi2}

\subsection{Rigid analytic interpretation }\label{sect-rigide}

We let $\mathcal{T}_{\an}$, $\mathcal{T}^{\times}_{\an}$, $\mathrm{GL}_{g,\an}$ be the rigid analytic spaces associated to the $\ocal_K$-schemes  $\mathcal{T}$, $\mathcal{T}^{\times}$ and $\GL$. We also let $\mathfrak{T}$, $\mathfrak{T}^\times$ and $\widehat{\GL}$ be the completions of  $\mathcal{T}$, $\mathcal{T}^{\times}$ and $\GL$  along their special fibers and $\mathcal{T}_{\rig}$, $\mathcal{T}_{\rig}^\times$ and $\mathrm{GL}_{g,rig}$ be their rigid analytic fibers. We have actions $\mathrm{GL}_{g,\an} \times \mathcal{T}_{\an} \rightarrow \mathcal{T}_{\an}$, $\mathrm{GL}_{g, rig} \times \mathcal{T}_{\rig} \rightarrow \mathcal{T}_{\rig}$ and $\widehat{\GL} \times \mathfrak{T} \rightarrow \mathfrak{T}$. When the context is clear, we just write $\GL$ instead of $\mathrm{GL}_{g,\an}$, $\mathrm{GL}_{g,rig}$ or $\widehat{\GL}$. For example, we have $\mathcal{T}_{\rig}^\times / \B = \mathcal{T}_{\an}^\times/\B$, because $\mathcal{T}^\times/\B$ is complete. Over $\mathcal{T}_{\rig}^\times/\B$, we have a diagram:
\begin{eqnarray*}
\xymatrix{ \mathcal{T}_{\rig}^\times/\U \ar[r]\ar[d] & \mathcal{T}_{\an}^\times/\U \ar[ld]\\
 \mathcal{T}_{\rig}^\times/\B & }
\end{eqnarray*}
where $\mathcal{T}_{\rig}^\times/\U$ is a torsor under $\U_{\rig}/\B_{\rig} = \T_{\rig}$ and $\mathcal{T}_{\an}^\times/\U$ is a torsor under $\U_{\an}/\B_{\an} = \T_{\an}$. %Finally, let us consider $\mathcal{T}_{\rig}^\times/\B^+ \rightarrow \mathcal{T}_{\rig}^\times/\B$ the fibration in $g$-dimensional affine spaces, which parametrizes flags of differential forms with a basis of each graded piece. It contains the dense open subset   $\mathcal{T}_{\rig}^\times/\mathrm{U}$.

We let $\mathcal{IW}_w^+$ and $\mathcal{IW}_w$ be the rigid spaces associated to $\mathfrak{IW}_w^+$ and $\mathfrak{IW}_w$ respectively.  We have a chain of rigid spaces:

$$ \mathcal{IW}_w^+ \rightarrow \mathcal{IW}_w \rightarrow {\cX}_1(p^n)(v) \rightarrow {\cX}_{\mathrm{Iw}}(p) (v) $$

The natural injection $\mathcal{F} \hookrightarrow {\omega_{G}}_{ \mathfrak{X}_1(p^n)(v)}$ is an isomorphism on the rigid fiber. More precisely, we can cover $\mathfrak{X}_1(p^n)(v)$ by affine open formal schemes $\Spf~R$ such that $\mathcal{F}$ and $\omega_G$ are free $R$ modules of rank $g$. We choose a  basis for $\mathcal{F}$ compatible with $\psi$ and a basis for $\omega_{G}$ such that  the inclusion is given by an upper triangular  matrix $M \in \mathrm{M}_g(R)$ with diagonal given by  $\mathrm{diag}(\beta_1,\ldots, \beta_g)$ where $\beta_i \in R$ and $v(\prod_i \beta_i(x)) = \frac{1}{p-1}\Hdg (G_x)$ for all closed points $x$ of $\Spec~R[1/p]$ (see proposition \ref{prop-HdgT}).

We thus obtain an open immersion:

$$ \mathcal{IW}_w \hookrightarrow \mathcal{T}^\times_{\rig}/\B_{ \cX_1(p^n)(v)}.$$

This immersion is  locally isomorphic to the inclusion:

$$  M \begin{pmatrix} % or pmatrix or bmatrix or Bmatrix or ...
      1  & 0 & \cdots & 0   \\
   p^w \B(0,1) &   1 & \cdots & 0\\
 \vdots & \vdots & \ddots & 0 \\
p^w B(0,1) & \cdots & p^w \B(0,1) & 1 \\

   \end{pmatrix} \hookrightarrow (\GL)_{\rig}/\B.$$

 Similarly, we have an open immersion:

$$ \mathcal{IW}^+_w \hookrightarrow \mathcal{T}^\times_{\an}/\B_{ \cX_1(p^n)(v)}$$

which is locally isomorphic to the inclusion

$$ M \begin{pmatrix} % or pmatrix or bmatrix or Bmatrix or ...
      1 + p^w \B(0,1)  & 0 & \cdots & 0   \\
  p^w \B(0,1) & 1 + p^w\B(0,1) & \cdots & 0\\
 \vdots & \vdots & \ddots & 0 \\
p^w B(0,1) & \cdots & p^w \B(0,1) & 1 +  p^w \B(0,1) \\

   \end{pmatrix}\hookrightarrow (\GL)_{\an}/\U$$

Let $\cXI(p^n)(v) = \cX_1(p^n)(v)/\B(\Z/p^n\Z)$ be the rigid space associated to $\mathfrak{X}_{\mathrm{Iw}}(p^n)(v)$ and let $\cX_{\mathrm{Iw}^+}(p^n)(v) = \cX_1(p^n)(v)/\U(\Z/p^n\Z)$. The map $\cX_{\mathrm{Iw}^+}(p^n)(v) \rightarrow \cXI(p^n)(v)$ is an \'etale cover with group $\T(\Z/p^n\Z)$.
The action of $\B(\Z/p^n\Z)$ on $\cX_1(p^n)(v) $ lifts to an action on the
rigid space  $\mathcal{IW}_w \rightarrow \cX_1(p^n)(v) $ since the notion of $w$-compatibility of the flag only depends on $\psi \mod \B(\Z/p^n\Z)$. Taking quotients we get a rigid space $\mathcal{IW}^o_w \rightarrow \cXI(p^n)(v)$.  Similarly, the action of $\U(\Z/p^n\Z)$ on $\cX_1(p^n)(v) $ lifts to an
action on the rigid space  $\mathcal{IW}^+_w \rightarrow \cX_1(p^n)(v) $  since the notion of $w$-compatibility of the flag and the bases of the
graded pieces only depends on $\psi \mod \U(\Z/p^n\Z)$. Taking quotients we obtain  a rigid space $\mathcal{IW}^{o+}_w \rightarrow
\cX_{\mathrm{Iw}^+}(p^n)(v)$.
From the open immersions above  we get  open immersions:

$$ \mathcal{IW}^o_w \hookrightarrow \mathcal{T}^\times_{\rig}/\B_{ \cXI(p^n)(v)} ~\textrm{and} ~\mathcal{IW}^{o+}_w \hookrightarrow
\mathcal{T}^\times_{\an}/\U_{ \cX_{\mathrm{Iw}^+}(p^n)(v)},$$where $\B_{ \cXI(p^n)(v)}$ and $\U_{ \cX_{\mathrm{Iw}^+}(p^n)(v)}$ are the
base changes of the algebraic groups $\B$ and $\U$ to $\cX_{\mathrm{Iw}^+}(p^n)(v)$.

\begin{prop}\label{prop-rig-int} Since $w > n-1 + \frac{v}{p-1}$, the compositions $\mathcal{IW}^o_w \hookrightarrow \mathcal{T}^\times_{\rig}/\B_{ \cXI(p^n)(v)} \rightarrow \mathcal{T}^\times_{\rig}/\B_{ \cXI(p)(v)}$ and $\mathcal{IW}^{o+}_w \hookrightarrow \mathcal{T}^\times_{\an}/\U_{ \cXI(p^n)(v)} \rightarrow \mathcal{T}^\times_{\an}/\U_{ \cX_{\mathrm{Iw}}(p)(v)}$ are open immersions.
\end{prop}

\begin{demo} We can work \'etale locally over $\cX_{\mathrm{Iw}}(p)(v)$. Let $S$ be a set of representatives in the Iwahori subgroup $\I(\Z_p) \subset \GL(\Z_p)$ of $\I(\Z/p^n\Z)/\B(\Z/p^n\Z)$. Over a suitable open affinoid $U$ of $\cX_1(p^n)(v)$, the map  $(\mathcal{IW}^o_w)\vert_U \rightarrow (\mathcal{T}^\times_{\rig}/\B_{ \cXI(p)(v)})\vert_U$  is isomorphic to the following  projection:

$$h \colon   \coprod_{\gamma \in S} M \begin{pmatrix} % or pmatrix or bmatrix or Bmatrix or ...
      1   & 0 & \cdots & 0   \\
       p^w \B(0,1) &  1 & \cdots & 0\\
 \vdots & \vdots & \ddots & 0 \\
p^w B(0,1) & \cdots & p^w \B(0,1) & 1 \\

   \end{pmatrix}.\gamma \rightarrow (\GL)_{\rig}/\B$$ There is a matrix $M'$ with integral coefficients such that $M'\cdot M = p^{\frac{v}{p-1}} \mathrm{Id}_g$.  It is trivial to check  that $M' \circ h$ is injective and so $h$ is injective.
The proof of the second part of the proposition is similar.
\end{demo}

\medskip

We have a diagram:

\begin{eqnarray*}
\xymatrix{ \mathcal{T}^\times_{\an}/\U_{ \cXI(p)(v)} \ar[d]& \mathcal{IW}^{o+}_w \ar[r]^{f_1} \ar[d]^{g_1} \ar[l]_{i_1}  &    \cX_{\mathrm{Iw}^+}(p^n)(v) \ar[d] \\
\mathcal{T}^\times_{\rig}/\B_{ \cXI(p)(v)} \ar[rd] &\mathcal{IW}^o_w \ar[r]^{f_2} \ar[d]^{g_2}  \ar[l]_{i_2} &  \cX_{\mathrm{Iw}}(p^n)(v) \ar[ld]  \\
&\cXI(p)(v) & }
\end{eqnarray*}

The maps $i_1$ and $i_2$ are open immersions, the maps $f_1$ and $f_2$ have geometrically connected fibers. The torus $\T(\Z_p)$ acts on $\mathcal{IW}^{o+}_w$ over $\mathcal{IW}^o_w$. This action is compatible with the maps $f_1$ and $f_2$ and the action of $\T(\Z/p^n\Z)$ on $\cX_{\mathrm{Iw}^+}(p^n)(v) $ over $\cX_{\mathrm{Iw}}(p^n)(v)$. It is also compatible with the maps $i_1$ and $i_2$ and the action of $\T_{\an}$ on  $\mathcal{T}^\times_{\an}/\U_{ \cXI(p)(v)}$ over $\mathcal{T}^\times_{\rig}/\B_{ \cXI(p)(v)}$.

Set $g = g_2 \circ g_1$. Let $\kappa$ be a $w$-analytic character. Then $\omega^{\dag\kappa}_w = g_\ast \oscr_{\mathcal{IW}^{o+}_w}[\kappa']$ is the projective
Banach sheaf of $w$-analytic, $v$-overconvergent weight $\kappa$ modular forms over $\cXI(p)(v)$ (this notion is defined in the appendix). It is the Banach sheaf
associated to the formal Banach sheaf $\mathfrak{w}^{\dag\kappa}_w$.

\begin{rem} Let $\T_w$ be the rigid-analytic torus which is the rigid analytic fiber of $\mathfrak{T}_{w}$. For example, we have $\T_w(\C_p) = (1 + p^w \ocal_{\C_p})^g$. The rigid space $\mathcal{IW}^{o+}_w$ is a $\T(\Z_p)\T_w$-torsor over $\mathcal{IW}^o_w$. By lemma  2.1 of \cite{PiTW}, it makes no difference in the definition of $\omega^{\dag\kappa}_w$ to take $\kappa'$-equivariant functions for the action of $\T(\Z_p)$ or of the bigger group $\T(\Z_p)\T_w$.
\end{rem}

\begin{defi}\label{def:overconvmodforms} Let $\kappa \in \mathcal{W}$. The space of $w$-analytic, $v$-overconvergent modular forms of weight $\kappa$ is:
$$\mathrm{M}_w^{\dag\kappa}(\cXI(p)(v)) = \HH^0(\cXI(p)(v), \omega^{\dag\kappa}_w).$$

The space of locally analytic overconvergent modular forms of weight $\kappa$ is:

$$\mathrm{M}^{\dag\kappa}(\XI(p)) = \colim_{v \rightarrow 0, w \rightarrow \infty} \mathrm{M}^{\dag\kappa}_{w}(\cXI(p)(v)).$$
\end{defi}

The space $\mathrm{M}_w^{\dag\kappa}(\cXI(p)(v))$ is a Banach space, for the norm induced by the supremum norm on $\mathcal{IW}^{o+}_w$. Its unit ball is the space $\mathrm{M}_w^{\dag\kappa}(\mathfrak{X}_{\mathrm{Iw}}(p)(v))$ of integral forms.

\begin{prop}\label{prop-faisceau-induction} If $\kappa \in X_+(\T)$, then there is a canonical restriction map:

$$\omega^\kappa\vert_{\cXI(p)(v)} \hookrightarrow \omega^{\dag\kappa}_{w}$$ induced by the open immersion:
$\mathcal{IW}^{o+}_w \hookrightarrow \mathcal{T}_{\an}/\U_{ \cX_{\mathrm{Iw}}(p)(v)}$.

This map is locally for the \'etale topology isomorphic to the inclusion

$$ V_{\kappa'} \hookrightarrow  V_{\kappa'}^{w-\an}$$ of the algebraic induction into the analytic induction.

\end{prop}

\begin{coro} For any $\kappa \in X_+(\T)$, we have an inclusion:

$$ \HH^0(\XI, \omega^\kappa) \hookrightarrow \mathrm{M}^{\dag,\kappa}_w( \cXI(p)(v))$$
from the space of classical forms of weight $\kappa$ into  the space of $w$-analytic, $v$-overconvergent modular forms of weight $\kappa$.
\end{coro}
\subsection{Overconvergent and $p$-adic modular forms}

We compare the notion $p$-adic modular forms introduced by Katz and used by Hida (\cite{Hi}) to construct ordinary eigenvarieties to the notion of overconvergent
locally analytic modular forms. Let $\mathfrak{X}_1(p^\infty)(0)$ be the projective limit of the formal schemes $\mathfrak{X}_1(p^n)(0)$. It is a  pro-\'etale cover
of $\mathfrak{X}_{\mathrm{Iw}}(p)(0)$ with group the Iwahori subgroup of $\GL(\Z_p)$, denoted by $\I$. In particular we have an action of $\B(\Z_p)$ on the space
$\HH^0(\mathfrak{X}_1(p^\infty)(0),\oscr_{\mathfrak{X}_{1}(p^\infty)(0)})$. Any character $\kappa \in \mathcal{W}$ can be seen as a character of $\B(\Z_p)$, trivial
on the unipotent radical.

\begin{defi}  Let $\kappa \in \mathcal{W}(K)$ be an $\ocal_K$-valued  character of the group $\T(\Z_p)$. The space of $p$-adic modular forms of weight $\kappa$ is:
$$\M^{\infty \kappa} := \HH^0(\mathfrak{X}_1(p^\infty)(0),\oscr_{\mathfrak{X}_{1}(p^\infty)(0)}) [\kappa']$$
\end{defi}
\medskip

Over $\mathfrak{X}_1(p^\infty)(0)$, we have a universal trivialization $\psi \colon \Z_p^g \simeq T_p\big(G^D[p^\infty]\big)^{\rm et}$ of the
$p$-adic \'etale Tate module of $G^D[p^\infty]$ and   a  comparison theorem:
$$ \HT_{G_\infty^D} \colon \Z_p^g \otimes_{\Z} \oscr_{\mathfrak{X}_{1}(p^\infty)(0)} \stackrel{\sim}{\rightarrow} \omega_{G} \otimes_{\oscr_{\mathfrak{X}}} \oscr_{\mathfrak{X}_{1}(p^\infty)(0)}.$$
As a result, for all $w\in ]n-1,n]$, we have a diagram:

\begin{eqnarray*}
\xymatrix{\mathfrak{X}_1(p^\infty)(0) \ar[r]^{i}\ar[d] & \mathfrak{IW}^+_w \ar[ld]\\
\mathfrak{X}_1(p^n)(0)&}
\end{eqnarray*}

Let $\mathfrak{U}_n(\Z_p)$ be the sub-group of $\U(\Z_p)$ of matrices which are trivial modulo $p^n\Z_p$. The map $i$ factorizes through an immersion $\mathfrak{X}_1(p^\infty)(0)/\mathfrak{U}_n(\Z_p) \hookrightarrow \mathfrak{IW}^+_w$ which is equivariant under the action of $\B(\Z_p)$.
This provides a map:
$$ \M^{\dag\kappa}_w(\mathfrak{X}_{Iw}(p)(0)) \rightarrow \M^{\infty \kappa}.$$

\begin{rem} A space analogue to $\M^{\dag\kappa}_w(\mathfrak{X}_{\mathrm{Iw}}(p)(0))$ appears in the work  \cite{SU} of Skinner-Urban for $\mathrm{GSp}_4$ . The space of semi-ordinary modular forms is a direct factor of $\M^{\dag\kappa}_w(\mathfrak{X}_{\mathrm{Iw}}(p)(0))$ cut out by a projector.
\end{rem}

\begin{prop} There is a natural injective map:
$$ \M^{\dag\kappa}(\mathfrak{X}_{\mathrm{Iw}}(p)) \hookrightarrow \M^{\infty\kappa}.$$
As a result, locally analytic overconvergent modular forms are $p$-adic modular forms.
\end{prop}
\begin{demo} This map is obtained as the limit  of the maps  $\M^{\dag\kappa}_w(\mathfrak{X}_{\mathrm{Iw}}(p)(v)) \rightarrow \M^{\dag\kappa}_w(\mathfrak{X}_{\mathrm{Iw}}(p)(0)) \rightarrow \M^{\infty\kappa}$. All spaces are torsion free $\ocal_K$-module so the injectivity can be checked after inverting $p$.  The injectivity of $\M^{\dag\kappa}_w(\cX_{\mathrm{Iw}}(p)(v)) \rightarrow \M^{\dag\kappa}_w(\cX_{\mathrm{Iw}}(p)(0))$ for $v$ small enough follows from the  the surjectivity  of the map  on the connected components $\Pi_0(\cX_{\mathrm{Iw}}(p)(0)) \rightarrow \Pi_0(\cX_{\mathrm{Iw}}(p)(v))$. The injectivity of the map $\M^{\dag\kappa}_w({\cX}_{\mathrm{Iw}}(p)(0)) \rightarrow \M^{\infty\kappa}[1/p]$ can be checked locally over $\cX_{\mathrm{Iw}}(p)(0)$. This boils down to the injectivity of the restriction map: $V^{w-\an}_{\kappa'}  \rightarrow \mathcal{F}^{0} (\I)$ where $\mathcal{F}^{0}(\I)$ is the space of $\mathcal{C}^0$, $\ocal_K$-valued  functions on $\I$.
\end{demo}

\subsection{Independence of the compactification}\label{part_indep_compact}
We will see that our modules of overconvergent modular forms are in fact independent of the compactifications.

If $S$ is a   rigid space, we say that a function $f$ on $S$ is bounded if the supremum norm  $\sup_{x \in S} \vert f(x) \vert$ is finite. If $S$ is
quasi-compact, then this property is automatically satisfied (see \cite{Bos}, p. 23).  We now recall the following  result:
\begin{thm}[\cite{Luk}, thm 1.6.1]\label{lucky} Let  $S$ be a smooth, quasi-compact
rigid space and $Z$ a co-dimension $\geq 1$ Zariski-closed subspace. Then any bounded function on
$S\backslash Z$ extends uniquely to $S$.
\end{thm}
Let $Y_{\mathrm{Iw}}^{\an}$ be the analytification of $Y_{\mathrm{Iw}}$ (see for example \cite{Bos}, sect. 1.13).
Set $\cXI(p)(v)\cap Y^{\an}_{\mathrm{Iw}} = \cY_{\mathrm{Iw}}(v)$. The space  $\cY_{\mathrm{Iw}}(v)$ does not depend on the compactification.  We say that $f \in  \HH^0( \cY_{\mathrm{Iw}}(v), \omega_{{w}}^{\dag\kappa})$ is bounded if it is bounded  when considered as a   function on the rigid space $\mathcal{IW}^{o+}_w \times_{\cX_{\mathrm{Iw}}(v)} \cY_{\mathrm{Iw}}(v)$.

\begin{prop} The module of $w$-analytic and $v$-overconvergent modular forms  is  exactly the submodule  of bounded sections of
 $ \HH^0( \cY_{\mathrm{Iw}}(v), \omega_{{w}}^{\dag\kappa})$.  In particular, $\M^{\dag\kappa}_w(\cXI(p)(v))$ is independent on the choice
of the toroidal compactification.
\end{prop}

\begin{demo}  The map $\M^{\dag\kappa}_w(\cXI(p)(v)) \rightarrow \HH^0( \cY_{\mathrm{Iw}}(v), \omega_{{w}}^{\dag\kappa})$ is clearly injective.
Let $$f \in \HH^0( \cY_{\mathrm{Iw}}(v), \omega_{{w}}^{\dag\kappa})$$ be a bounded section. This is a bounded function on
$\mathcal{IW}^{o+}_w \times_{\cX_{\mathrm{Iw}}(v)} \cY_{\mathrm{Iw}}(v)$, homogeneous for the action of the torus $\T(\Z_p)$.
By theorem \ref{lucky}, it extends to  a function on $\mathcal{IW}^{o+}_w$, which is easily seen to be homogeneous of the same weight.
\end{demo}

\bigskip

\begin{rem} The module $\M^{\dag\kappa}_w(\cXI(p)(v))$ could thus have been defined without reference  to any compactification. Nevertheless,
compactifications will turn out to be quite useful in the next section allowing us to prove properties of these modules.
\end{rem}

\subsection{Dilations}\label{sec:dilations}

For our study of Hecke operators it is useful to define slight generalizations of the spaces $\mathcal{IW}^{o+}_w$. This section is technical and may
be skipped at
the first reading.  Let $n \in \N, v < \frac{1}{2p^{n-1}}$ (resp.~$v< \frac{1}{3p^{n-1}}$ if $p=3$) and   let $\underline{w} = (w_{i,j})_{1 \leq j\leq i \leq g} \in
]\frac{v}{p-1}, n- v\frac{p^n}{p-1}]^{\frac{g(g+1)}{2}}$ satisfying $w_{i+1, j} \geq w_{i,j}$ and $w_{i, j-1} \geq w_{i,j}$.  We will call $(w_{i,j})$ a dilation
parameter. Let  $\mathcal{IW}^{o+}_{\underline{w}}$ be the open subset of $\mathcal{T}_{\an}/\U_{ \cXI(p)(v)}$ such that for any finite extension $L$ of $K$, an element
in $\mathcal{IW}^{o+}_{\underline{w}}(L)$ is the data of:
 \begin{itemize}
\item an $\ocal_L$-point of $\cXI(p)(v)$, coming from  a semi-abelian scheme $G \rightarrow \ocal_L$, with $H_n$ its canonical subgroup of level $n$,  and a flag
$\Fil_\bullet H_n$,

\item a flag of differential forms $\Fil_\bullet \mathcal{F} \in \mathcal{T}_{\rig}/\B(\ocal_L)$ and  for all $1 \leq j \leq g$, an element $\omega_j \in \Gr_j \mathcal{F}$ such that  there is   a trivialization $\psi\colon  H^D_n(\bar{K}) \simeq \Z/p^n\Z^g$  where $\psi$ is compatible with  $\Fil_\bullet H_1$,  and the following holds:

Denote by $e_1,\ldots,e_g$ the canonical basis of $\Z/p^n\Z^g$, set $w_0 = n- v\frac{p^n}{p-1}$   and by abuse of notation set  $\HT_{w_0}$  for the map $\HT_{w_0} \circ \psi$; then for all $1 \leq i \leq g$, we have:
$$ \omega_i  \mod \Fil_{i-1} \mathcal{F} + p^{w_0} \mathcal{F}   = \sum_{ j \geq i} a_{j,i} \HT_{w_0}(e_j) $$  where $a_{j,i} \in \ocal_L$ and $v(a_{j,i}) \geq w_{j,i}$    if $j > i $ and $v(a_{i,i}-1) \geq w_{i,i}$.

\end{itemize}

\medskip

When $w_{i,j} = w$ for all $1 \leq j \leq i \leq g$ and there exists $n \in \N$ such that $w \in ]n -1 + \frac{v}{p-1}, n - v\frac{p^n}{p-1}]$, we have $\mathcal{IW}^{o+}_{\underline{w}} = \mathcal{IW}^{o+}_{w}$. The spaces $\mathcal{IW}^{o+}_{\underline{w}}$ are dilations of the space $\mathcal{IW}^{o+}_{w_0}$, in the sense that we relax the $w_0$-compatibility with $\psi$ and impose a weaker condition.
The rigid space  $\mathcal{IW}^{o+}_{\underline{w}}$ is locally  for the \'etale topology over $\cXI(p)(v)$  isomorphic to

$$   \begin{pmatrix} % or pmatrix or bmatrix or Bmatrix or ...
     1 + p^{w_{1,1}} \B(0,1)  & 0 & \cdots & 0   \\
      p^{w_{2,1}} \B(0,1) &  1 + p^{w_{2,2}}\B(0,1) & \cdots & 0\\
 \vdots & \vdots & \ddots & 0 \\
p^{w_{g,1}} \B(0,1) & \cdots & p^{w_{g,g-1}} \B(0,1) & 1 + p^{w_{g,g}} \B(0,1) \\

   \end{pmatrix}.\I/\U(\Z_p). $$

If $\kappa$ is a $\inf_i \{w_{i,i}\}$-analytic character, we  can define Banach sheaves $\omega^{\dag\kappa}_{\underline{w}}$ and the space $\M^{\dag\kappa}_{\underline{w}}\big(\cXI(p)(v)\big)$ of $\underline{w}$-analytic, $v$-overconvergent modular forms as  in section \ref{sect-rigide}.

\begin{rem} If $\underline{w}$ and $\underline{w'}$ are two dilation parameters satisfying $w_{i,j} = w'_{i,j}$ as soon as $i \neq j$, and if $\kappa$ is a $\inf_{i,j} \{ w_{i,i}, w'_{j,j}\}$ analytic character then the sheaves $\omega_{\underline{w}}^{\dag\kappa}$ and $\omega_{\underline{w}'}^{\dag\kappa}$ are canonically isomorphic.
\end{rem}

\section{Hecke operators}\label{sect_Hecke}
In this section we define  an action of the Hecke operators on the space of overconvergent modular forms and we single out
one of these operators  which is  compact.
\subsection{Hecke operators outside $p$}\label{part_Hecke_p_prime}
Let $q$ be a prime integer with $(q,p)=1$ and let $\gamma \in \gsp(\qq_q) \cap \mathrm{M}_{2g}(\Z_q)$.
Let $C_{\gamma} \subset Y_{\mathrm{Iw}} \times Y_{\mathrm{Iw}}\times \Spec~K$ be the moduli space over $K$ classifying pairs $(A, A')$ of principally polarized abelian schemes of dimension $g$, equipped with  level $N$ structures $(\psi_N, \psi_N')$, flags $\Fil_\bullet A[p]$ and $\Fil_\bullet A'[p]$ of $A[p]$ and $A'[p]$,  and an isogeny $\pi\colon  A \rightarrow A'$ of type $\gamma$, compatible with the level structures, the flags and the polarizations (see \cite{FC}, chapter 7). We have two finite \'etale projections $p_1$, $p_2\colon C_{\gamma} \rightarrow Y_{\mathrm{Iw},K}$. They extend to projections on the analytifications $p_1$, $p_2\colon C_{\gamma}^{\an} \rightarrow Y_{\mathrm{Iw},K}^{\an}$. There is an issue with the boundary: in general
it is not possible to find toroidal compactifications for $C_{\gamma}$ and $Y_{\mathrm{Iw}}$ in such a way that the
projections $p_1$ and $p_2$ extend to finite morphisms.  Moreover if one varies $\gamma$ it is not possible to find
toroidal compactifications of $Y$ and the $C_{\gamma}$'s  such that all projections extend to the compactifications.
Therefore we will  define Hecke operators on  $ \HH^0( \cY_{\mathrm{Iw}}(v), \omega_{{w}}^{\dag\kappa})$ and show  that
these Hecke operators  map bounded functions to bounded functions thus defining an action on
$\M^{\dag\kappa}_w(\cXI(p)(v))$ (see section \S\ref{part_indep_compact}).

The isogeny $\pi$ induces a map $\pi^\ast\colon  \omega_{A} \rightarrow \omega_{A'}$, hence a map $\pi^\ast\colon  p_2^\ast \mathcal{T}_{\an}^\times/\U \rightarrow p_1^\ast \mathcal{T}_{\an}^\times/\U$ which is an isomorphism.   Let $n\in \N$, $v < \frac{1}{2p^{n-1}}$ (resp $< \frac{1}{3p^{n-1}}$ if $p=3$. Set $C_{\gamma}^{\an}(v) = C_{\gamma}^{\an} \times_{p_1} \cY_{\mathrm{Iw}}(v) = C_{\gamma}^{\an} \times_{p_2} \cY_{\mathrm{Iw}}(v) $. The last equality follows from the facts that level prime-to-$p$ isogenies are \'etale in characteristic $p$ and that the Hodge height is preserved under \'etale isogenies. Let $w \in ]n-1 +\frac{v}{p-1}, n-v \frac{p^n}{p-1}]$.
\begin{lem}\label{lem_pi_B} The map $\pi^\ast$ induces an isomorphism  $$\pi^\ast\colon    p_2^\ast \mathcal{IW}^{o+}_w\vert_{\cY_{\mathrm{Iw}}(v)} \simeq  p_1^\ast \mathcal{IW}^{o+}_w\vert_{\cY_{\mathrm{Iw}}(v)}.$$
\end{lem}

\begin{proof} We can check the equality at the level of points. Let $L$ be a finite extension of $K$, and $A$, $A'$ be two semi-abelian schemes over $\Spec~\ocal_K$ with canonical subgroup of level $n$, say $H_n$ and $H_n'$. Enlarging
$L$, we may assume that $H_n(L) \simeq H_n'(L) \simeq \Z/p^n\Z^g$.  Let $\pi\colon  A \rightarrow A'$ be an isogeny of type $\gamma$. It induces a group isomorphism  $H_n({L}) \tilde{\rightarrow} H_n'({L})$. We have a commutative diagram (see section \ref{sect-functo}):
\begin{eqnarray*}
\xymatrix{ \mathcal{F}' \ar[r]^{\pi^\ast}\ar[d]  & \mathcal{F} \ar[d] \\
\mathcal{F}'/p^w \ar[r] & \mathcal{F}/p^w \\
(H_n')^D(L) \ar[r]^{\pi^D}\ar[u]^{\mathrm{HT}_w}& H_n^D(L)\ar[u]^{\mathrm{HT}_w}}
\end{eqnarray*}

Since the bottom line is an isomorphism and the maps $\HT_w \otimes 1$ are isomorphisms it follows that $\pi^\ast$ is an isomorphism.

%Now let $\omega'_1,\ldots,\omega'_g$ be a basis of differential forms of $A'$ and $x'_1,\ldots, x'_g$ be a basis of $(H_n')^D(K)$ such that $\mathrm{HT}(x'_i) = \omega'_i\vert_{ \omega_{H_n'}}$. We obtain from the diagram that $\mathrm{HT}(\pi^Dx'_i) = \pi^\ast \omega'_i\vert_{ \omega_{H_n}}$. This shows that $\pi^\ast p_2^\ast B_w^\times \subset p_1^\ast B_w^\times.$ Reversing the argument, one shows the other inclusion.
\end{proof}

We let $\pi^{\ast-1}$ be the inverse of the isomorphism given by the proposition. We can now define the Hecke operator $T_{\gamma}$ as the composition:

$$T_\gamma\colon   \HH^0(\cY_{\mathrm{Iw}}(v), \omega_{w}^{\dag\kappa}) \stackrel{p_2^\ast}\rightarrow \HH^0( C_\gamma(v), p_2^\ast \omega_{w}^{\dag\kappa}) \stackrel{\pi^{\ast-1}} \rightarrow \HH^0( C_\gamma(v), p_1^\ast \omega_{w}^{\dag\kappa}) \stackrel{\mathrm{Tr}p_1}\rightarrow \HH^0(\cY_{\mathrm{Iw}}(v), \omega_{w}^{\dag\kappa}).$$
\subsection{Hecke operators at $p$}
We now define an action of the dilating Hecke algebra at $p$. For $i=1,\ldots,g$, let $C_i$ be the moduli scheme
over $K$ parametrizing principally polarized abelian schemes $A$, a level $N$ structure $\psi_N$, an self-dual  flag $\Fil_\bullet A[p]$ of subgroups  of $A[p]$ and a lagrangian sub-group $L \subset A[p^2]$ if $i=1,\ldots, g-1$ or $L\subset A[p]$ if $i=g$, such that $L[p] \oplus \Fil_i A[p]= A[p]$. There are two projections $p_1$, $p_2\colon C_i \rightarrow Y_{\mathrm{Iw},K}$. The first projection is defined by forgetting $L$. The second projection is defined by mapping $(A, \psi_N, \Fil_\bullet A[p])$ to $(A/L, \psi_N', \Fil_\bullet A/L[p])$ where $\psi_N'$ is the image of the level $N$ structure and $\Fil_\bullet A/L[p] $ is defined as follows:

\begin{itemize}
\item For $j = 1,\ldots, i$, $\Fil_j A/L[p]$ is simply the image of $\Fil_j A[p]$ in $A/L$,
\item For $j= i+1,\ldots, g $, $\Fil_j A/L[p]$ is the image in $A/L$ of  $\Fil_j A[p] + p^{-1}( \Fil_j A[p] \cap pL)$.
\end{itemize}
As before we consider the analytifications $p_1$, $p_2\colon C_i^{\an} \rightarrow Y_{\mathrm{Iw}}^{\an}$.
\subsubsection{The operator $U_{p,g}$}\label{sec:Upg}
We start by recalling the following result:
\begin{prop2}[\cite{Far}, prop. 17]\label{prop_U_pppp} Let $G$ be a semi-abelian scheme of dimension $g$ over $\ocal_K$, generically abelian. Assume that $\Hdg(G) < \frac{p-2}{2p-2}$. Let $H_1$ be the canonical sub-group of level $1$ of $G$ and let $L$ be a sub-group of $G_K[p]$ such that $H_1 \oplus L = G_K[p]$. Then $\Hdg(G/L) = \frac{1}{p} \Hdg(G)$, and $G[p]/L$ is the canonical sub-group of level $1$ of $G/L$.
\end{prop2}
Let $\mathcal{C}_g(v) = C_g^{\an} \times_{p_1, Y_{\mathrm{Iw}}^{\an}} \cY_{\mathrm{Iw}}(v)$. If $v <  \frac{p-2}{2p-2}$, by the previous proposition, we have a diagram:
\begin{eqnarray*}
\xymatrix{& \mathcal{C}_g(v)\ar[ld]^{p_1} \ar[rd]_{p_2} & \\
\cY_{\mathrm{Iw}}(v) &  & \cY_{\mathrm{Iw}}(\frac{v}{p})}
\end{eqnarray*}
 Let  $\pi\colon  A \rightarrow A'$ be the universal isogeny over $\mathcal{C}_g(v)$. It induces a map $\pi^\ast\colon  \omega_{A'} \rightarrow \omega_{A}$ and a map $\pi^\ast\colon  p_2^\ast \mathcal{T}_{\an}^\times/\U \rightarrow p_1^\ast \mathcal{T}_{\an}^\times/\U$. This map is an isomorphism.  Let $n \in \N$, and $v < \inf\{ \frac{1}{4p^{n-1}}\}$. Let $w \in ]n -1 + v\frac{1}{p-1}, n-v \frac{p^n}{p-1}]$. We have the lemma whose proof is identical to the proof of lemma \ref{lem_pi_B}:
\begin{lem2} The map $\pi^\ast$ induces an isomorphism  $$\pi^\ast\colon    p_2^\ast \mathcal{IW}^{o+}_w\vert_{\cY_{\mathrm{Iw}}(v)}  \simeq  p_1^\ast \mathcal{IW}^{o+}_w\vert_{\cY_{\mathrm{Iw}}(v)}.$$
\end{lem2}

We let $\pi^{\ast-1}$ be the inverse of this map. Let $\kappa$ be a $w$-analytic character. We now define the Hecke operator $U_{p,g}$ as the composition:
$$  \HH^0(\cY_{\mathrm{Iw}}(\frac{v}{p}), \omega_{w}^{\dag\kappa}) \stackrel{p_2^\ast}\rightarrow \HH^0( \mathcal{C}_g(v), p_2^\ast \omega_{w}^{\dag\kappa}) \stackrel{\pi^{\ast-1}} \longrightarrow \HH^0( \mathcal{C}_g(v), p_1^\ast \omega_{w}^{\dag\kappa}) \stackrel{{p}^\frac{-g(g+1)}{2}\mathrm{Tr}p_1}\longrightarrow \HH^0(\cY_{\mathrm{Iw}}(v), \omega_{w}^{\dag\kappa}).$$
The operator $U_{p,g}$ hence improves the radius of overconvergence. Remark also that we normalize the trace of the map $p_1$ by a factor $p^{-\frac{g(g+1)}{2}}$
which is an inseparability degree (see \cite{PiH}, sect. A.1). By a slight abuse of notation we also denote by $U_{p,g}$
the endomorphism of  $\HH^0(\cY_{\mathrm{Iw}}(v), \omega_{w}^{\dag\kappa})$ defined as the composition of the operator we just defined with the restriction map
$\HH^0(\cY_{\mathrm{Iw}}(v), \omega_{w}^{\dag\kappa}) \rightarrow \HH^0(\cY_{\mathrm{Iw}}(\frac{v}{p}), \omega_{w}^{\dag\kappa})$.

\subsubsection{The operators $U_{p,i}$, $i=1,\ldots,g-1$}\label{sec:Upi}

Let $F$ be a finite extension of $K$ and  $( A,\psi_N, \Fil_\bullet A[p], L)$ be an $F$-point of $C_i$. Let $(A'= A/L, \psi'_N, \Fil_\bullet A'[p])$ be the image by $p_2$ of  $( A,\psi_N, \Fil_\bullet A[p], L)$. Set $\pi\colon  A \rightarrow A/L$ for the isogeny.

\begin{prop2}\label{laverdun} If $\Hdg\bigl(A[p^\infty]\bigr) < \frac{p-2}{2p^2-p}$ and $ \Fil_g A[p]$ is the canonical sub-group of level $1$, then $\Hdg\bigl(A[p^\infty]/L) \leq  \Hdg\bigl(A[p^\infty]\bigr)$ and $\Fil_g A'[p]$ is the canonical sub-group of level $1$ of $A'$.
\end{prop2}
\begin{proof}  We assume that $\Hdg(A[p^\infty])\leq \frac{p-2}{2p^2-p}$ and we are reduced by  proposition  \ref{prop_laver} to show that $\Fil_g A'[p]$ has   degree  greater or equal to  $g-\Hdg(A[p^\infty])$. Let $H_2$ be the canonical subgroup of level $2$ in $A$ and  $x_1,\ldots,x_g$ be a basis of $H_2(\bar{K})$ as a $\Z/p^2\Z$-module. We complete it to a basis $x_1,\ldots,x_{2g}$  of $A[p^2](\bar{K})$. We can assume that $\Fil_\bullet A[p]$ is given by
$0 \subset \langle px_1\rangle \subset \ldots \subset \langle px_1,\ldots, px_g\rangle = H_1$ and that $L$ is given by
$\langle px_{i+1},\ldots, px_{2g-i}, x_{2g-i+1},\ldots, x_{2g}\rangle$. Set $\tilde{H}_1= \langle p x_1,\ldots,p x_i, x_{i+1},\ldots,
x_{g}\rangle$.  With these notations $\Fil_g A'[p] = \tilde{H}_1/L$.  We will show that $\deg~\tilde{H}_1/L \geq
g-\Hdg(A[p^\infty])$. We have a generic isomorphism $$H_2/H_1\stackrel{ \mathrm{diag}(p1_i, 1_{g-i})} \rightarrow
\tilde{H}_1/\langle px_{i+1},\ldots, px_g\rangle,$$ which implies  that $\deg \tilde{H}_1/\langle
px_{i+1},\ldots, px_g\rangle \geq g-p\Hdg(A[p^\infty])$. By proposition \ref{prop_laver}, $$\Hdg(A[p^\infty]/\langle
px_{i+1},\ldots, px_g\rangle) \leq p \Hdg(A[p^\infty])$$ and $\tilde{H}_1/\langle px_{i+1},\ldots, px_g\rangle$ is the canonical
sub-group of level $1$ of $A/\langle px_{i+1},\ldots, px_g\rangle$. At the level of the generic fiber, we have
$$A/\langle px_{i+1},\ldots, px_g\rangle [p] =  \langle px_{g+1},\ldots, px_{2g}\rangle \oplus \tilde{H}_1/
\langle px_{i+1},\ldots, px_g\rangle.$$ By proposition \ref{prop_U_pppp} we obtain $\deg \tilde{H}_1/\langle px_{i+1},\ldots,
p x_{2g}\rangle  \geq g- \Hdg(A[p^\infty])$. We conclude, since the map $\tilde{H}_1/\langle px_{i+1},\ldots, p x_{2g}\rangle
\rightarrow \tilde{H}_1/L$ is a generic isomorphism.
\end{proof}

We  set $\mathcal{C}_i(v) = C_i^{\an} \times_{p_1, Y_{\mathrm{Iw}}^{\an}} \cY_{\mathrm{Iw}}(v)$. If $v \leq v_1$, we have a diagram:
\begin{eqnarray*}
\xymatrix{& \mathcal{C}_i(v)\ar[ld]^{p_1} \ar[rd]_{p_2} & \\
\cY_{\mathrm{Iw}}(v) &  & \cY_{\mathrm{Iw}}(v)}
\end{eqnarray*}

Let $\pi\colon  A \rightarrow A/L$ be the universal isogeny over $\mathcal{C}_i(v)$. We  have a map $\pi^\ast\colon  \omega_{A/L} \rightarrow \omega_{A}$. It induces a map $\tilde{\pi}^\ast\colon  p_2^\ast \mathcal{T}_{\an}^\times \rightarrow p_1^\ast\mathcal{T}_{\an}^\times $ which sends a basis $\omega'_1,\ldots, \omega'_g$ of $\omega_{A/L}$ to $p^{-1}\pi^\ast \omega'_1,\ldots,p^{-1}\pi^\ast \omega_{g-i},  \pi^\ast \omega'_{g-i+1}, \ldots,  \pi^\ast \omega'_{g}$. This map is an isomorphism, we call $\tilde{\pi}^{\ast-1}$ its inverse, and denote by the same symbol the quotient map  $\tilde{\pi}^{\ast-1}\colon  p_1^\ast \mathcal{T}_{\an}^\times/\U\rightarrow p_2^\ast\mathcal{T}_{\an}^\times/\U.$

Let $n\in \N$, $v < \inf\{\frac{1}{3p^{n-1}}, \frac{p-2}{2p^2-p}\}$ and  $\underline{w}=(w_{i,j})_{1\leq j \leq k \leq g}$ be a dilation parameter such that $w_{k,j} \in ]0, n-2-v\frac{ p^{n}}{p-1}]$.
\begin{prop2}\label{prop_Upi} We have $\tilde{\pi}^{\ast-1}  p_1^\ast \mathcal{IW}^{o+}_{\underline{w}} \subset p_2^\ast \mathcal{IW}^{o+}_{\underline{w'}}$ where
\begin{eqnarray*}
w'_{k,j} &= &w_{k,j} ~\textrm{if}~ j\leq k  \leq i, \\
 w'_{k,j} &= &1 + w_{k,j} ~\textrm{if}~ j\leq i~\textrm{and} ~k  \geq i+1,\\
 w'_{k,j}& =& w_{k,j}~\textrm{if}~j \geq i+1.
 \end{eqnarray*}
\end{prop2}
\begin{proof}  Let $(A,\Fil_\bullet A[p],\psi_N,L)$ be an $F$-point of $\mathcal{C}_i(v)$. We set $A'= A/L$ and  assume that
$F$ is large enough to trivialize the  group schemes $H_n$, $H_n^D$, $H'_n$ and ${H'_n}^D$.    There are
$\Z/p^n\Z$-basis  $e_1,\ldots,e_g$ for $H_n(F)$ and $e'_1,\ldots,e'_g$ for $H_n'(F)$ such that the flags on $H_1(F)$
and $H'_1(F)$ are given by $\Fil_j = <p^{n-1}e_g,p^{n-1}e_{g-1},\ldots,p^{n-1}e_{g-j+1}>$ and $\Fil'_j = <p^{n-1}e'_g, p^{n-1}e'_{g-1},\ldots,p^{n-1}e'_{g-j+1}>$ and  the isogeny $\pi$ induces a map $H_n(F) \rightarrow H'_n(F)$ given by $\mathrm{diag} ( p \mathrm{Id}_{g-i}, \mathrm{Id}_i)$ in the basis.  Let $x_1,\ldots,x_g$ and $x'_1,\ldots,x'_g$ be the dual basis of $H_n^D(F)$ and ${H'_n}^D(F)$ (for a choice of a primitive $p^{n}$-th root of unity). The flags on $H_1^D(F)$ and ${H'_1}^D(F)$ are given by $\Fil_j = <x_1,x_{2},\ldots,x_{j}>$ and   $\Fil'_j = <x'_1,x'_{2},\ldots,x'_{j}>$ and the map $\pi^D\colon  {H'_n}^D(F) \rightarrow H_n^D(F)$ is given by $\mathrm{diag} ( p \mathrm{Id}_{g-i}, \mathrm{Id}_i)$.  Set $w_0 = n- v \frac{p^n}{p-1}$. We  have a commutative diagram:

\begin{eqnarray*}
\xymatrix{ \mathcal{F}' \ar[r]^{\pi^\ast}\ar[d]  & \mathcal{F}\ar[d] \\
\mathcal{F}'/p^{w_0} \ar[r] & \mathcal{F}/p^{w_0} \\
(H_n')^D(L) \ar[r]^{\pi^D}\ar[u]^{\mathrm{HT}_{w_0}}& {H}_n^D(L)\ar[u]^{\mathrm{HT}_{w_0}}}
\end{eqnarray*}

Let $(\Fil_\bullet \mathcal{F}', \{\omega_i' \in \Gr_i \mathcal{F}' \})$ be an element of $p_2^\ast \mathcal{T}_{\an}^\times/\U$ over $A'$. We assume that $\tilde{\pi}^{\ast} (\Fil_\bullet \mathcal{F}', \{\omega_i' \in \Gr_i \mathcal{F}' \}) =  (\Fil_\bullet \mathcal{F}, \{\omega_i \in \Gr_i \mathcal{F} \})  \in p_1^\ast\mathcal{IW}^{o+}_{\underline{w}}$.

This means that:

\begin{eqnarray*}
p^{-1}\pi^\ast \omega_j' &=& \sum_{k= j}^{g} a_{k,j} \mathrm{HT}_{w_0}(x_k) \mod p^{w_0}\mathcal{F} + \Fil_{j-1} \mathcal{F} ~\textrm{for}~ 1 \leq j \leq g- i,  \\
\pi^\ast \omega_j'  &=& \sum_{k= j}^{g} a_{k,j} \mathrm{HT}_{w_0}(x_k) \mod p^{w_0}\mathcal{F} + \Fil_{j-1} \mathcal{F} ~\textrm{for} ~g-i+1\leq j \leq g.
\end{eqnarray*}

where $(a_{k,j})_{1\leq j \leq k \leq g} \in \ocal_L^{\frac{g(g+1)}{2}}$ satisfy $v(a_{k,k}-1) \geq w_{k,k}$ and $v(a_{k,j}) \geq w_{k,j}$ for $k> j$.  We obtain:

\begin{eqnarray*}
\pi^\ast \omega_j' &=& \sum_{k= j}^{g-i} a_{k,j} \mathrm{HT}_{w_0}(\pi^Dx'_k) + \sum_{k= g-i+1}^g pa_{k,j} \mathrm{HT}_{w_0}(\pi^Dx'_k)
 \mod p^{w_0}\mathcal{F}+ \Fil_{j-1} \mathcal{F}   \\ \textrm{for}~ 1 \leq j \leq g- i,&& \\
\pi^\ast \omega_j' &= &\sum_{k= j}^{g} a_{k,j} \mathrm{HT}_{w_0}(\pi^Dx'_k)  \mod p^{w_0}\mathcal{F} + \Fil_{j-1} \mathcal{F}~\textrm{for} ~g-i+1\leq j \leq g.
\end{eqnarray*}

Since $p\mathcal{F} \subset \pi^\ast \mathcal{F'}$, we now get that:

\begin{eqnarray*}
 \omega_j' &=& \sum_{k= j}^{g-i} a_{k,j} \mathrm{HT}_{w_0}(x'_k) + \sum_{k= g-i+1}^g pa_{k,j} \mathrm{HT}_{w_0}(x'_k) \mod p^{w_0-1}\mathcal{F}' + \Fil_{j-1} \mathcal{F}'  \\
\textrm{for}~ 1 \leq j \leq g- i, && \\
 \omega_j' &= &\sum_{k= j}^{g} a_{k,j} \mathrm{HT}_{w_0}(x'_k)  \mod p^{w_0-1}\mathcal{F}' + \Fil_{j-1} \mathcal{F}'~\textrm{for} ~g-i+1\leq j \leq g.
\end{eqnarray*}

\end{proof}

Let $\underline{w}$ and $\underline{w}'$ be as in the proposition. Let $\kappa$ be a $\inf_j\{w_{j,j}\}$-analytic character.  We now define the Hecke operator $U_{p,i}$ as:
$$U_{p,i}\colon   \HH^0(\cY_{\mathrm{Iw}}(v), \omega_{\underline{w}'}^{\dag\kappa}) \stackrel{p_2^\ast}\rightarrow \HH^0( C_\gamma(v), p_2^\ast \omega_{\underline{w}'}^{\dag\kappa}) \stackrel{\tilde{\pi}^{-1\ast}} \rightarrow \HH^0( C_\gamma(v), p_1^\ast \omega_{\underline{w}}^{\dag\kappa}) \stackrel{p^{-i(g+1)}\mathrm{Tr}p_1}\rightarrow \HH^0(\cY_{\mathrm{Iw}}(v), \omega_{\underline{w}}^{\dag\kappa}).$$

This Hecke operator improves analyticity.  Note the normalization of the trace map by a factor $p^{-(g+1)i}$ which is an inseparability degree.  We also denote by $U_{p,i}$ the endomorphism of $\M^{\dag\kappa}_{\underline{w}}\big(\cXI(p)(v)\big)$ obtained by composing the above operator with the restriction $$\M^{\dag\kappa}_{\underline{w}}\big(\cXI(p)(v)\big) \rightarrow \M^{\dag\kappa}_{\underline{w'}}\big(\cXI(p)(v)\big).$$

\subsubsection{The relationship between  $U_{p,i}$ and $\delta_i$}

In this paragraph we establish the relationship between the operators $U_{p,i}$ and the operators $\delta_i$ of section \S\ref{sect_class_GL}.
Let $1 \leq i \leq g-1$  and consider the correspondence $p_1$, $p_2\colon \mathcal{C}_i(v) \rightarrow  \cXI(v)$.

\begin{prop2}\label{prop_Ud} Let $L$ be a finite extension of $K$,  $x,y \in \cY_{\mathrm{Iw}}(v)(L)$ such that  $y \in p_2 (p_1^{-1}) \{x\}$. Let $w >0$ and $\kappa$ be a $w$-analytic character.  There exists a commutative diagram where the vertical maps
are isomorphisms:

\begin{eqnarray*}
\xymatrix{ (\omega_{{w}}^{\dag\kappa})_y \ar[r]^{\tilde{\pi}^{\ast-1}} & (\omega_{{w}}^{\dag\kappa})_x\\
V_{\kappa',L}^{w-\an} \ar[r]^{\delta_i}\ar[u] & V_{\kappa',L}^{w-\an} \ar[u]}
\end{eqnarray*}
\end{prop2}
\begin{proof} This follows from the definition (see also lemma 5.1 of \cite{PiH}).
\end{proof}
\subsubsection{A compact operator}
Let $n\in \N$,  $v< \inf\{\frac{1}{3p^{n-1}}, \frac{p-2}{2p^2-p}\}$,  $w \in ]\frac{v}{p-1}, n-g -1 - v \frac{p^{n}}{p-1}]$ and $\kappa$ a $w$-analytic character.

The  composite  operator $\prod_{i=1}^g U_{p,i}$ induces  a map from $ \M^{\dag\kappa}_{\underline{w'}}(\cXI(\frac{v}{p})) \rightarrow   \M^{\dag\kappa}_{{w}}(\cXI(v))$ where $\underline{w}'= (w_{i,j}')$ is defined by:
$$ w'_{i,j} = i-j + w .$$
The natural restriction map $res\colon   \M^{\dag\kappa}_{{w}}(\cXI(v)) \rightarrow \M^{\dag\kappa}_{\underline{w}'}(\cXI(\frac{v}{p}))$ is compact. We let $U =  \prod_{i=1}^g U_{p,i}\circ res$. This is a compact endomorphism of $\M^{\dag\kappa}_{{w}}(\cXI(v))$.

\subsection{Summary}
For all $q \nmid pN$, let $\mathbb{T}_q$ be the spherical Hecke algebra $$\Z[ \mathrm{GSp}_{2g}(\qq_q)/\mathrm{GSp}_{2g}(\Z_q)].$$ Let
$\mathbb{T}^{Np}$ be the restricted tensor product of the algebras $\mathbb{T}_q$. We have defined an action of $\mathbb{T}^{Np}$ on the Frechet
space $\mathrm{M}^{\dag\kappa}(\XI)$. Consider the dilating Hecke algebra, $\mathbb{U}_p$, defined as the polynomial algebra over $\Z$ with indeterminates
$X_1,\ldots, X_g$. We have also defined an action of  $\mathbb{U}_p$,  sending $X_i$ to  $U_{p,i}$.   We proved that the operator $U= \prod_i U_{p,i}$ is compact.  Let us denote  by $\mathbb{T}^{\dag\kappa}$ the image of $\mathbb{T}^{Np}\otimes_\Z \mathbb{U}_p$ in $ \mathrm{End}\big(\mathrm{M}^{\dag\kappa}(\XI)\big)$ and  call it the overconvergent Hecke algebra of weight $\kappa$.

\section{Classicity}

\subsection{Statement of the main result}

 Let $\kappa = (k_1,\ldots, k_g)\in X^+(\T)$. We have a series of natural restriction maps:

$$ \HH^0(\XI,\omega^\kappa) \stackrel{r_1}\rightarrow \HH^0(\cXI(p)(v), \omega^\kappa) \stackrel{r_2}\rightarrow \HH^0(\cXI(p)(v), \omega^{\dag\kappa}_{w})$$
and we establish a criterion for an element in $\HH^0(\cXI(p)(v), \omega^{\dag\kappa}_{w})$ to be in the image of $r_2 \circ r_1$.
Let $\underline{a}=(a_1,\ldots, a_g) \in \R_{\geq 0}^g$. We set $\mathrm{M}^{\dag\kappa}_{w}(\cXI(p)(v))^{< \underline{a}}$ for the union of the generalized eigenspaces where $U_{p,i}$ has slope $<a_i$ for $1 \leq i \leq g$.

\begin{thm}\label{thm-class} Let  $\underline{a} = (a_1,\ldots, a_g)$ with   $a_i = k_{g-i}-k_{(g-i)+1} + 1$ when $1 \leq i \leq g-1$ and $a_g = k_g - \frac{g(g+1)}{2}$. Then we have:

$$ \mathrm{M}^{\dag\kappa}_{w}(\cXI(v))^{< \underline{a}} \subset \HH^0(\XI, \omega^\kappa).$$
\end{thm}

The proof of this theorem is split in two parts. We first show that $ \mathrm{M}^{\dag\kappa}_{w}(\cXI(p)(v))^{< \underline{a}} \subset \HH^0(\cXI(p)(v), \omega^\kappa)$. This is a classicity statement  at the level of sheaves and it is   easily deduced from the results of section \ref{sect_GL}, see proposition
\ref{prop_BGGrel}.

We  conclude by applying    the main theorem of \cite{PiS2} as follows. Since $U_{p,g}$ is a compact operator on $\HH^0(\cXI(p)(v), \omega^\kappa)$,
for all $a_g \in \R_{\geq 0}$ we can define $\HH^0(\cXI(p)(v), \omega^\kappa)^{<a_g}$  which is the sum of generalized eigenspaces for $U_{p,g}$
with eigenvalues of slope less that $a_g$.

\begin{thm}[\cite{PiS2}]  Let $a_g = k_g - \frac{g(g+1)}{2}$. Then $\HH^0(\cXI(p)(v), \omega^\kappa)^{<a_g}$ is a space of classical forms.
\end{thm}

\subsection{Relative BGG resolution}

We now take  $w\in ]\frac{v}{p-1}, 1-v\frac{p}{p-1} ]$. We remark that for such a $w$, the fibers of  the morphism $\pi\colon  \mathcal{IW}^o_w \rightarrow \cXI(p)(v)$ are connected. Consider the cartesian  diagram:
\begin{eqnarray*}
\xymatrix{ \mathcal{IW}^{o+}_w\times \mathcal{T}_{\an}^\times \ar[d]^{\pi_1}\ar[r]&\mathcal{T}_{\an}^\times \ar[d]\\
\mathcal{IW}^{o+}_w \ar[r] \ar[rd]^{\pi_2}&\mathcal{T}_{\an}^\times/\U\ar[d] \\
&\cXI(p)(v)}
\end{eqnarray*}
We have an action of the Iwahori sub-group $\I$ on $\mathcal{IW}^{o+}_w\times \mathcal{T}_{\an}^\times$ and by differentiating we obtain an action of the enveloping
algebra $U(\gl)$ on $$(\pi_2\circ \pi_1)_\ast\oscr_{\mathcal{IW}^{o+}_w\times \mathcal{T}_{\an}^\times }$$ denoted $\star$. We have already defined an inclusion
$d_0\colon  \omega^\kappa \rightarrow \omega^{\dag\kappa}_{w}$. For all $\alpha \in \Delta$ we now define a map $\Theta_\alpha\colon  \omega^{\dag\kappa}_w
\rightarrow    \omega^{\dag s_\alpha \bullet \kappa}_w$. We first define an endomorphism of  $$(\pi_2\circ \pi_1)_\ast\oscr_{\mathcal{IW}^{o+}_w\times
\mathcal{T}_{\an}^\times }$$  by sending a section $f$  to $X_{-\alpha}^{<\kappa,\alpha^\vee>+1}\star f$. It follows from    section \ref{subsect_BGGGL} that this map
restricted to $\omega^{\dag\kappa}_w$ produces the expected map $\Theta_\alpha$. We then set $d_1\colon  \oplus \Theta_\alpha\colon  \omega^{\dag\kappa}_w
\rightarrow \bigoplus_{\alpha \in \Delta}   \omega^{\dag s_\alpha \bullet \kappa}_w$. We have the following   relative BGG resolution, which is a relative version
of the theory recalled in section \ref{subsect_BGGGL}:

\begin{prop}\label{prop_BGGrel} There is an exact sequence of sheaves over $\cXI(p)(v)$:

$$ 0\rightarrow \omega^\kappa \stackrel{d_0}\rightarrow \omega^{\dag\kappa}_w \stackrel{d_1} \rightarrow \bigoplus_{\alpha \in \Delta}   \omega^{\dag s_\alpha \bullet \kappa}_w$$
\end{prop}

\begin{demo} Tensoring-completing the exact sequence \ref{exactsequence} (or more precisely its $w$-analytic version, see the remark below) by $\oscr_{\cXI(p)(v)}$ we get a  sequence :

\begin{eqnarray*}  0 \longrightarrow V_{\kappa'} \hat{\otimes} \oscr_{\cXI(p)(v)}\stackrel{d_0\otimes 1}{\longrightarrow} V_{\kappa'}^{w-\an}  \hat{\otimes} \oscr_{\cXI(p)(v)} \stackrel{d_1\otimes 1}{\longrightarrow }\bigoplus_{\alpha \in \Delta} V_{s_\alpha\bullet \kappa'}^{w-\an}  \hat{\otimes} \oscr_{\cXI(p)(v)}
\end{eqnarray*}

The exactness of this sequence  follows from lemma \ref{lem-completion}, noting that the image of $d_1$  is closed in $\bigoplus_{\alpha \in \Delta} V_{s_\alpha\bullet \kappa'}^{\an}$ by the main theorem of  \cite{Ow}.
By proposition \ref{prop-faisceau-induction}, this exact sequence is locally for the \'etale topology isomorphic to the
sequence of the proposition which is exact.
\end{demo}

\begin{rem} The definition of the map $\Theta_\alpha$ does not require the condition  $w\in ]\frac{v}{p-1}, 1-v\frac{p}{p-1}]$ but it is needed for the exactness of the sequence.
\end{rem}
The maps $\Theta_\alpha$ do not commute with the action of the Hecke operators $U_{p,i}$, for $i=1,\ldots, g-1$. Precisely, we have the following

\begin{prop} For $1 \leq i \leq g-1$ we have a commutative diagram
\begin{eqnarray*}
\xymatrix{ \HH^0(\cXI(p)(v), \omega^{\dag\kappa}_{w}) \ar[d]^{U_{p,i}}\ar[r]^{\Theta_\alpha} & \HH^0(\cXI(p)(v), \omega^{\dag s_\alpha \bullet\kappa}_{w})\ar[d]^{\alpha(d_{g-i})^{<\kappa, \alpha^\vee> +1}U_{p,i}} \\
\HH^0(\cXI(p)(v), \omega^{\dag\kappa}_{w}) \ar[r]^{\Theta_\alpha} & \HH^0(\cXI(p)(v), \omega^{\dag s_\alpha \bullet\kappa}_{w})}
\end{eqnarray*}
\end{prop}
\begin{proof} Let $f \in \HH^0(\cXI(p)(v), \omega^{\dag\kappa}_{w})$. We need to check that $$ \Theta_\alpha U_{p,i} f = \alpha(d_{g-i})^{<\kappa, \alpha^\vee> +1}U_{p,i} \Theta_\alpha f.$$  We  apply  proposition \ref{prop_Ud} to reduce to the results of section \ref{sect_class_GL}.
\end{proof}
\subsection{Classicity at the level of the sheaves}
We now assume only that $v$ is small enough  for the operators $U_{p,i}$ to be defined. We make no particular assumption on $w$.

\begin{prop} The submodule of $  \mathrm{M}^{\dag\kappa}_{w}(\cXI(p)(v))$ on which $U_{p,i}$ acts with slope strictly less than
$k_{g-i}-k_{g-i+1}+1$ for $1\leq i \leq g-1$ and $U_{p,g}$ acts with finite slope is contained in $\HH^0(\cXI(p)(v),\omega^\kappa)$.
\end{prop}
\begin{proof} Let $f \in   \mathrm{M}^{\dag\kappa}_{w}(\cXI(p)(v))$. For simplicity let us assume that $f$ is an eigenvector for  all operators $U_{p,i}$ with  corresponding eigenvalue $a_{i}$. The operator $\prod_{i=1}^{g-1} U_{p,i}$ increases analyticity. Since we have $f = \prod_{i=1}^{g-1} a_{p,i}^{-1} \prod_{i=1}^{g-1} U_{p,i} f$ we can assume that $w \in ]\frac{v}{p-1}, 1-v\frac{p}{p-1}]$. We  endow the space $\HH^0(\cXI(p)(v),\omega^{\dag\kappa}_{w})$ and $\HH^0(\cXI(p)(v),\omega^{\dag s_\alpha\bullet \kappa}_{w})$ for all simple positive  root $\alpha$ with the supremum norm  over the ordinary locus. This is indeed a norm by the analytic continuation principle (but of course $\HH^0(\cXI(p)(v),\omega^{\dag\kappa}_{w})$ may not be complete for this norm). For this  choice, the $U_{p,i}$ operators have norm less or equal to $1$. By the relative BGG exact sequence of proposition \ref{prop_BGGrel} it is enough to prove that $\Theta_{\alpha} f$ is $0$. Let  $\alpha$  be the character $(t_1,\ldots, t_g)\mapsto t_i.t_{i+1}^{-1}$. Since $U_{p,{g-i}} \Theta_\alpha(f) =p^{k_{i+1}-k_i-1} \Theta_\alpha U_{p,g-i}(f)$, we see that $\Theta_{\alpha}(f) $ is an eigenvector for $U_{p,g-i}$ for an  eigenvalue of negative valuation. Since    the norm of $U_{p,g-i}$ is less than $1$, $\Theta_{\alpha}(f)$ has to be zero.
\end{proof}

\section{Families}
Recall that the weight space $\mathcal{W} = \Hom (\T(\Z_p), \C_p^\times)$ was defined in section \ref{sect_ws}. For any  affinoid
open subset $\mathcal{U} $ of $\mathcal{W}$, by proposition \ref{prop_cara}, there exists $w_U >0$ such that the universal
character $\kappa^{\un}\colon
\T(\Z_p) \times \mathcal{W} \rightarrow \C_p^\times$ restricted to $\mathcal{U}$ extends to an analytic character
$\kappa^{\un}\colon  \T(\Z_p)(1+p^{w_\mathcal{U}} \ocal_{\C_p}) \times \mathcal{U}
\rightarrow \C_p^\times$.

\subsection{Families of overconvergent modular forms}
\subsubsection{The universal sheaves $\omega^{\dag\kappa^{\un}}_{w}$}\label{sec:universalomega}  Let $n \in \N, v \leq \frac{1}{2p^{n-1}}$
(resp.~ $\frac{1}{3p^{n-1}}$ if $p=3$)  and  $w \in ]n-1+ \frac{v}{p-1},  n- v \frac{p^n}{p-1}]$ satisfying $w\geq w_\mathcal{U}$.  We deduce
immediately from proposition \ref{prop_cara} that the construction given in section \ref{sect-ovmods} works in families:

\begin{prop2}  There exists a sheaf  $\omega^{\dag\kappa^{\un}}_{w}$ on  $\cXI(p)(v) \times \mathcal{U}$ such that  for any weight
$\kappa \in \mathcal{U}$, the fiber of $\omega^{\dag\kappa^{\un}}_w$ over $\cXI(p)(v) \times \{\kappa\}$ is $\omega^{\dag \kappa}_{w}$.
\end{prop2}

\begin{proof} We  consider the projection $\pi \times 1 \colon \mathcal{IW}^{o+}_w \times \mathcal{U} \rightarrow \cXI(p)(v)
\times \mathcal{U}$. We  take  $\omega^{\dag\kappa^{\un}}_w$ to be the  sub-sheaf of $(\pi\times 1)_\ast\oscr_{ \mathcal{IW}^{o+}_w \times
\mathcal{U} }$ of $(\kappa^{\un})'$-invariant sections for the action of $\T(\Z_p)$.
\end{proof}

Let $A$ be the ring of rigid analytic functions on $\mathcal{U}$. Let $M_{v,w}$ be the $A$-Banach module
$\HH^0(\cXI(v)\times \mathcal{U}, \omega^{\dag\kappa^{\rm un}}_w)$. Passing to the limit on $v$ and $w$  we get the $A$-Frechet space
$M^\dag=\displaystyle \lim_{v\rightarrow 0, w \rightarrow \infty} M_{v,w}$.

It is clear that the geometric definition of Hecke operators given in section \ref{sect_Hecke} works in families. We thus have an action of the
Hecke algebra of level prime to $Np$, $\mathbb{T}^{Np}$ on the space $M_{v,w}$. We also have an action of $\mathbb{U}_p$, the dilating Hecke
algebra at $p$, on $M_{v,w}$ for $v$ small enough.

Let $D$  be the boundary in $\cXI(v)$.  We let $\omega^{\dag\kappa^{\un}}_w(-D)$  be the  cuspidal sub-sheaf of $\omega^{\dag\kappa^{\un}}_w$ of
sections  vanishing along $D$. Let $M_{v,w,\cusp}$ be the $A$-Banach module $\HH^0(\cXI(p)(v)\times \mathcal{U}, \omega^{\dag\kappa}_w(-D))$ and
$M^\dag_{\cusp} = \displaystyle  \lim_{v\rightarrow 0, w \rightarrow \infty} M_{v,w,\cusp}$. All these modules are stable under the action of
the Hecke algebra.  We wish to construct an eigenvariety out of this data.

\subsubsection{Review of Coleman's Spectral theory}
\label{sec:colemantheory}

A convenient reference for the material in this section is \cite{Buzz}. The data we are given is:

\begin{itemize}
\item A reduced, equidimensional affinoid $\Spm~A$ and a continuous character $\kappa:\Z_p^\times A^\times$ (e.g. $\mathcal{U}=\Spm~A$ is an admissible
affinoid open of the weight space $\mathcal{W}$ and $\kappa$ is the universal character $\kappa^{\rm un}$),
\item A Banach $A$-module $M$ (e.g. the $A$- module of
$p$-adic families of modular forms $M_{v,w}$ defined above for suitable $v,w$),
\item A commutative endomorphism algebra $\mathbb{T}$ of $M$ over $A$ (e.g. the Hecke algebra),
\item A compact operator $U\in \mathbb{T}$ (e.g. the operator
$\prod_i U_{p,i}$).
\end{itemize}

\begin{defi2} \begin{enumerate}

\item Let $I$ be a set. Let $C(I) $ be  the $A$-module of functions
$ \{ f \colon I \rightarrow A,~ \lim_{i \rightarrow \infty} f(i) = 0\}$, where the limit is with respect to the filter
of complements of finite subsets of $I$. The module $C(I)$ is equipped with its supremum norm.
\item  A Banach $A$-module $M$ is orthonormalizable if  there is a set $I$ such that $M \simeq C(I)$.
\item A Banach-$A$ module $M$ is projective if there is a set $I$ and a Banach $A$-module $M'$  such that $M \oplus M' = C(I)$.
\end{enumerate}
\end{defi2}

The following lemma follows easily from the universal property of projective Banach modules given in \cite{Buzz}, p. 18.

\begin{lem2} Let $$0 \rightarrow M \rightarrow M_1 \rightarrow \cdots \rightarrow M_n \rightarrow 0$$ be an exact sequence  of Banach $A$-modules,
where the differentials are continuous and for all $1\leq i \leq n$, $M_i$ is projective. Then $M$ is projective.
\end{lem2}

We suppose now that $M$ is projective. Since $U$ is a compact operator, the following power series
$P(T):= \det(1-TU\vert M) \in A[\![T]\!]$ exists. It is known that $P(T) = 1 + \sum_{n\geq 1} c_n T^n$ where $c_n \in A$ and $\vert c_n\vert r^n
\rightarrow 0 $ when $n \rightarrow \infty$ for all positive $r \in \R$.  As a result, $P(T)$ is a rigid analytic function on
$\Spm~A\times \mathbb{A}^1_{\an}$.

\begin{defi2} The spectral variety   $\mathcal{Z} $ is the closed rigid sub-space of $\Spm~A\times \mathbb{A}^1_{\an}$ defined by the equation $P(T) = 0$.
\end{defi2}

 A pair $(x, \lambda) \in \Spm~A \times \mathbb{A}^1_{\an}$ is in $\mathcal{Z}$ if and only if there is an element
$m \in M\otimes_A \bar{k(x)}$ such that $U\cdot m = \lambda^{-1} m$.

\begin{prop2}[\cite{Buzz}, thm. 4.6] The map $p_1 \colon \mathcal{Z} \rightarrow \Spm~A$ is locally finite flat. More precisely, there is an
admissible cover of $\mathcal{Z}$ by open affinoids $\{ \cU_i\}_{i \in I}$  with the property that
the map $\cU_i  \rightarrow p_1(\cU_i)$ is finite flat.
 \end{prop2}

Let $i \in I$ and let $B$ be the ring of functions on $p_1(\cU_i)$. To $\cU_i$ is associated  a factorization $P(T) = Q(T)R(T)$ of $P$ over $B[\![T]\!]$
where $Q(T)$ is a polynomial and $R(T)$ is a power series co-prime to $Q(T)$. Moreover, $\cU_i$ is defined by the equation $Q(T) =0$ in $p_1(\cU_i) \times
\mathbb{A}^1_{\an}$.  To $\cU_i$, one can associate a direct factor  $\mathcal{M}(\cU_i)$ of $M$. This is the generalized eigenspace
of $M\otimes_A B$ for the eigenvalues of $U$ occurring in $Q(T)$. The rule $\cU_i \mapsto \mathcal{M}(\cU_i)$ gives a  coherent
sheaf $\mathcal{M}$ of $\oscr_{\mathcal{Z}}$-modules  which can be viewed as  the universal generalized eigenspace.

\begin{defi2} The eigenvariety $\cE$ is the affine rigid space  over $\mathcal{Z}$  associated to the  coherent $\oscr_{\mathcal{Z}}$-algebra
generated by the image of $\mathbb{T}$ in  $\mathrm{End}_{\oscr_{\mathcal{Z}}} \mathcal{M}$.
\end{defi2}

The map $w \colon \cE \rightarrow \Spm~A$ is locally finite and $\cE$ is equidimensional.  For each $x \in \Spm~A$, the geometric points of $w^{-1}(x)$ are in bijection with the set of eigenvalues of $\mathbb{T}$ acting on $M\otimes_A \bar{k(x)}$, which are of finite slope for $U$ (i.e the eigenvalue of $U$ is non-zero).

The space $\cE$ parametrizes eigenvalues. One may sometimes ask for a family of eigenforms. We have the following:

\begin{prop2} Let $x \in \cE$  and $f \in \mathcal{M}\otimes_Ak(x)$ be an eigenform corresponding to $x$. Assume that $w$ is unramified at $x$.
Then there is a family of eigenforms $F$ passing through $f$. More precisely, there
exist
$\Spm~B$, an admissible open affinoid   of $\cE$ containing $x$ and $F \in M \otimes_A B$ such that:
\begin{itemize}

\item $ \forall \phi \in \mathbb{T}, \phi \cdot F  = F \otimes \phi$,
\item the image of $F$ in $M \otimes_A k(x)$ is $f$.
\end{itemize}
\end{prop2}

\begin{demo} Let $\Spm~B$ be an admissible open affinoid of $\cE$ containing $x$ such that $w  \colon \Spm~B \rightarrow w(\Spm~B)$ is finite unramified.
Let $C$ be the ring of rigid analytic functions of $w(\Spm~B)$.  Let $e$ be the projector in $B\otimes_C B$ corresponding to the diagonal.
The projective $B$-module $e \mathcal{M}(B) \otimes_C B$ is the sub-module of $ \mathcal{M}(B) \otimes_C B$ of elements $m$ satisfying
$  b.m =  m \otimes b $ for all $b \in B$. We have a reduction map:
$$ \mathcal{M}(B) \otimes_C B \rightarrow \mathcal{M}(B) \otimes_A k (x)$$
and $f$ is in the image of $e\mathcal{M}(B) \otimes_C B$. Any element  $F$ of $e \mathcal{M}(B) \otimes_C B$, mapping to $f$ is a family of
eigenforms passing through $f$.
\end{demo}
%{}

\subsubsection{Properties of the module $M_{v,w,\cusp}$ }
\label{sec:moduleM}
In proposition \ref{prop:mainspecializationresult}  we will prove the following structure result about the
module $M_{v,w,\cusp}$:

\begin{prop2} a) The  Banach $A$-module $M_{v,w, \cusp}$ is projective.\smallskip

b) For any $\kappa \in U$,  the specialization map
$$M_{v,w,\cusp} \rightarrow \HH^0(\cXI(p)(v), \omega^{\dag\kappa}_w(-D)) $$ is surjective.
\end{prop2}

Granted this proposition, one can apply Coleman's spectral theory as described in section \S \ref{sec:colemantheory}
to construct an equidimensional
eigenvariety over the weight space. Thanks to theorem \ref{thm-class} we also get precise information
about the points of this eigenvariety.  This is enough to prove theorems \ref{thm1} and \ref{mainthm2} of the introduction.
\medskip

The rest of this chapter will be devoted to the proof of proposition \ref{prop:mainspecializationresult}.
Let us point out two main differences between the case $g=1$ treated in \cite{AIS} and  \cite{PiTW} and the
case $g \geq 2$ treated in the present article. First of all, the
ordinary locus in modular curves over a $p$-adic field is an affinoid,
whereas it is not an affinoid in the toroidal compactification of Siegel modular varieties  of genus $g \geq 2$.
Secondly, for  modular curves the classical modular sheaves are interpolated by coherent sheaves,
whereas for $g \geq 2$, the sheaves $\omega^{\dag\kappa}_w$ are only Banach sheaves.

In the modular curve case, because of the two reasons mentioned above, it is easy to see by a cohomological
argument that the proposition holds even in the
non-cuspidal case (see \cite{PiTW}, cor. 5.1). We believe that the cuspidality assumption is necessary when
$g \geq 2$.

\medskip

Because the proof of proposition \ref{prop:mainspecializationresult} is quite involved  we will first  explain the strategy of the proof
(for technical
reasons the actual proof of the proposition follows a slightly different line of arguments then the one
sketched below, but the ideas are presented faithfully). Let $\XI(p)^\star$ be the minimal compactification
of $Y_{\mathrm{Iw}}(p)$. Let
$X_{\mathrm{Iw}}(p)_{\rig}^\star$ be the rigid analytic fiber of $\XI(p)^\star$ and
$\xi \colon X_{\mathrm{Iw}}(p)_{\rig} \rightarrow X_{\mathrm{Iw}}(p)_{\rig}^\star$ be
the projection.   We define $\cXI(p)^\star(v)$ as the image of $\cXI(p)(v)$ in $\XI(p)_{\rig}^\star$.
If $v \in \qq$, this is an affinoid.   We have a Banach sheaf
$\omega^{\dag\kappa^{\un}}_w$ on $\cXI(p)(v)\times \mathcal{U}$ and we consider the sheaves
$(\xi\times 1)_\ast \omega^{\dag\kappa^{\un}}_w$ and  $(\xi\times 1)_\ast
\omega^{\dag\kappa^{\un}}_w(-D)$ on $\cXI(p)^\star(v)\times \mathcal{U}$. We will show that
$(\xi\times 1)_\ast \omega^{\dag\kappa^{\un}}_w(-D)$ is a Banach sheaf.   By
theorem \ref{thm_Khiel} of the appendix $(\xi\times 1)_\ast \omega^{\dag\kappa^{\un}}_w(-D)$ is the sheaf associated to
its global sections $M_{v,w,\cusp}$ that we
denote by $M$ for simplicity.  By the acyclicity property for Banach sheaves (proposition \ref{prop_acy}),  we have, for any affinoid
covering $\mathfrak{U}= (\cX_i)_{i\in I}$ of
$\cXI^\star(v)$, a  Cech exact sequence:

$$0 \rightarrow M  \rightarrow \oplus_{i\in I} \HH^0\big(\cX_i \times U,(\xi\times 1)_\ast \omega^{\dag\kappa^{\un}}_w(-D)\big) \rightarrow \oplus_{i,j \in I} \HH^0\big(\cX_{i,j}\times \mathcal{U},(\xi\times 1)_\ast \omega^{\dag\kappa^{\un}}_w(-D) \big) \cdots $$

To prove that $M$ is a projective $A$-module, it is enough to show that the Banach modules
$\HH^0\big(\cX_{\underline{i}} \times \mathcal{U},(\xi\times 1)_\ast
\omega^{\dag\kappa^{\un}}_w(-D)\big)$ are projective for all multi indexes $\underline{i} \in I^{t}$ for $1 \leq t \leq \sharp I$.   In other words,
the first part
of the proposition can be checked locally on the minimal compactification.  Of course, locally on the
toroidal compactification the projectivity holds. We are
reduced to study the sheaf $(\xi\times 1)_\ast \omega^{\dag\kappa^{\un}}_w(-D)$ along the boundary and we conclude
by some   explicit computations. The second point
of the proposition follows by similar arguments. In conclusion, the main step is to prove that
$(\xi\times 1)_\ast \omega^{\dag\kappa^{\un}}_w(-D)$ is a Banach
sheaf. We don't know in general if  direct images of Banach sheaves by proper
morphisms are Banach sheaves (see open problem \ref{opprobB} in the
appendix). Our approach to prove this property in this special situation is via formal models as
discussed in section \ref{sect_formal_ban}.

\subsubsection{An integral family}\label{sec:anintegralfamily}

We consider the map defined in section \ref{sec:ThesheafcalF}:

$$\upsilon=\pi_1 \circ \pi_2 \circ \pi_3 \colon  \mathfrak{IW}_w^+ \rightarrow \mathfrak{X}_1(p^n)(v).$$

There is an action of the torus $\mathfrak{T}_w$ on  $\mathfrak{IW}_w^+$ over $\mathfrak{X}_1(p^n)(v)$.
For all $\kappa^o \in \mathcal{W}(w)^o(K)$ we set $\tilde{\mathfrak{w}}^{\dag\kappa^o}_w =
\upsilon_\ast \oscr_{\mathfrak{IW}_w^+}[\kappa^{o'}]$.

Let $\kappa \in \mathcal{W}(w)$ mapping to $\kappa^o$.
Let $\pi_4 \colon \mathfrak{X}_1(p^n)(v) \rightarrow \mathfrak{X}_{\mathrm{Iw}}(p)(v)$ be the finite projection.
One recovers $\mathfrak{w}^{\dag\kappa}_w$ by taking the direct image $\pi_{4,\ast}\tilde{\mathfrak{w}}^{\dag\kappa^o}_w$
and the $\kappa'$-equivariant sections for the action of $\B(\Z_p)\mathfrak{B}_w$ or equivalently the
invariants under the action
$\B(\Z/p^n\Z)$ of the sheaf $\pi_{4,\ast}\tilde{\mathfrak{w}}^{\dag\kappa^o}_w(-\kappa')$
with twisted action by $-\kappa'$   (after this twist, the action of $\B(\Z_p)\mathfrak{B}_w$ factors
through its quotient $\B(\Z/p\Z)$). The group $\B(\Z/p^n\Z)$ is of order divisible by $p$ and has
higher  cohomology on $\Z_p$-modules. For this reason, we will implement the strategy of section \ref{sec:moduleM}
at the level of $\mathfrak{X}_1(p^n)$ for a while, and at the very end  invert $p$ and take into account the action of  $\B(\Z/p^n\Z)$; see proposition \ref{prop:mainspecializationresult}.

\medskip

The sheaves $\tilde{\mathfrak{w}}^{\dag\kappa^o}_w$ can be interpolated.  Consider the projection
$$ \upsilon\times 1 \colon \mathfrak{IW}_w^+ \times \mathfrak{W}(w)^o
\rightarrow \mathfrak{X}_1(p^n)(v) \times \mathfrak{W}(w)^o$$
and the family of formal Banach sheaves
$$\tilde{\mathfrak{w}}^{\dag\kappa^{o\un}}_w =\big( \upsilon\times
1\big)_{\ast} \oscr_{\mathfrak{IW}_w^+ \times \mathfrak{W}(w)^o}[{\kappa^{o\un}}'].$$

\subsubsection{Description of the sections}

We denote by $\Spf R$  an open affine sub-formal scheme of $\mathfrak{X}_1(p^n)$.   We let
$\psi \colon (\Z/p^n\Z)^g \rightarrow H_n^D(R[1/p])$ be the pull back of the
universal trivialization. We have an isomorphism:

$$ (\HT_w \circ \psi) \otimes 1 \colon \Z/p^n\Z^g \otimes_\Z R/p^wR \rightarrow \mathcal{F}/p^w\mathcal{F}.$$

We denote by  $e_1,\ldots,e_g$ the canonical basis of $(\Z/p^n\Z)^g$. Let $f_1,\ldots,f_g$ be a basis of $\mathcal{F}$ lifting the vectors
$\HT_w \circ \psi (e_1),
\ldots, \HT_w \circ \psi (e_g)$.  With these choices, $\mathfrak{IW}_w^+\vert_{\Spf R}$ is identified with   the set of matrices:
   $$  \begin{pmatrix} % or pmatrix or bmatrix or Bmatrix or ...
      1 & 0 &\ldots & 0\\
      p^w \mathfrak{B}(0,1) & 1  & \ldots& 0 \\
   \vdots & \vdots & \ddots & \vdots \\
     p^w\mathfrak{B}(0,1) & p^w\mathfrak{B}(0,1) &\ldots& 1\\
  \end{pmatrix}\times \begin{pmatrix} % or pmatrix or bmatrix or Bmatrix or ...
      1+ p^w \mathfrak{B}(0,1) \\
      1 +p^w \mathfrak{B}(0,1)  \\
   \vdots\\
     1+ p^w \mathfrak{B}(0,1)\\
   \end{pmatrix}\times_{\Spf~\ocal_K} \Spf~R.$$

where the first $g \times g$ matrix parametrizes the position of the flag and the second column vector the basis of the graded pieces.

For $1 \leq j < i \leq g$,  we let $X_{i,j}$ be the coordinate of the ball on the $i$-th line and $j$-th column in the $g\times g$ matrix and we let $X_1,\ldots,X_g$ be the coordinates on the column vector.

A function $f$ on $\mathfrak{IW}_w^+\vert_{\Spf~R}$  is a power series:
$$ f(X_{i,j}, X_k) \in R\langle\langle X_{i,j}, X_k, 1 \leq j < i \leq g, 1 \leq k \leq g \rangle\rangle.$$

Let $\kappa^o \in \mathfrak{W}(w)^o$. Then  $f \in \tilde{\mathfrak{w}}^{\dag \kappa^o}_{w}(R)$ if and only if:

$$ f(  X_{i,j}, \lambda. X_k )  = \kappa^{o'}(\lambda) f( X_{i,j}, X_k )  ~\forall \lambda \in
\mathfrak{T}_w(R').$$

We deduce the lemma:

\begin{lem2}\label{lem--1}  A section $f \in \tilde{\mathfrak{w}}^{\dag \kappa^o}_{w}(R)$ has a unique decomposition:

$$ f (X_{i,j}, X_k ) = g( X_{i,j}) \kappa^{o'} ( 1 + p^wX_{1},\ldots,1+ p^wX_{g}),$$

where $g(X_{i,j}) \in R\langle\langle X_{i,j}, 1\leq i < j \leq g \rangle\rangle$. This decomposition sets a
bijection:

$$\tilde{\mathfrak{w}}^{\dag \kappa^o}_{w}(R) \simeq R\langle\langle X_{i,j}, 1\leq i < j \leq g\rangle\rangle.$$

\end{lem2}

Similarly:

\begin{lem2}\label{lem--2}  A section $f \in \tilde{\mathfrak{w}}^{\dag \kappa^{o\un}}_{w}(R\hat{\otimes}
\ocal_K\langle\langle S_1,\ldots,S_g\rangle\rangle)$ has a unique decomposition:

$$ f ( X_{i,j}, X_k) = g( X_{i,j}) {\kappa^{o\un}}' ( 1 + p^wX_{1},\ldots,1+ p^wX_{g}),$$

where $g(X_{i,j}) \in R\langle\langle S_1,\ldots,S_g,X_{i,j}\rangle\rangle$. This decomposition sets a bijection:

$$\tilde{\mathfrak{w}}^{\dag \kappa^{o\un}}_{w}(R) \simeq  R\langle\langle S_1,\ldots,S_g,X_{i,j}\rangle\rangle.$$
\end{lem2}

\begin{lem2}\label{lem--3} Let $\varpi \in \ocal_K$ be the uniformizing element. We have
 $${\kappa^{o\un}} \big( (1+ p^w{{X}_{i}})_{1\leq i \leq g} \big)  \in 1 + \varpi \ocal_K\langle\langle
S_1,\ldots,S_g, X_{1}, \ldots, X_g \rangle\rangle^{\times}$$
 \end{lem2}
\begin{demo} We have $$  (1 + p^wX_{i})^{S_ip^{-w + \frac{2}{p-1}}} = \sum_{k\geq 0} \frac{S_ip^{-w + \frac{2}{p-1}}
(S_ip^{-w + \frac{2}{p-1}}-1)\cdots (S_ip^{-w + \frac{2}{p-1}}-k+1)}{k!} (p^wX_{i})^k .$$The constant term of this
series is $1$. Recall that for any integer $k\geq
1$, $v(k!)\leq  \frac{k -1}{p-1}$. As a result the $k$-th coefficient of the series, for $k>0$ has valuation at least $kw- kw + \frac{2k}{p-1}-
\frac{k-1}{p-1}>0$.
\end{demo}
\bigskip

For all $m \in \N$, we let $X_1(p^n)(v)_m$, $\mathcal{W}(w)^o_m$ be the schemes over
$\ocal_K/\varpi^m\ocal_K$ obtained by reduction modulo $\varpi^m$  from
$\mathfrak{X}_1(p^n)(v)$ and $\mathfrak{W}(w)^o$. We let $\tilde{\mathfrak{w}}^{\dag \kappa^{o\un}}_{w,m}$
(resp.~$\tilde{\mathfrak{w}}^{\dag \kappa^o}_{w,m}$),  be
the quasi-coherent sheaf over $X_1(p^n)(v)_m\times \mathcal{W}(w)^o_m$ (resp.~$X_1(p^n)(v)_m$)  obtained by
pull back.

\begin{coro2}\label{cor-triv}
The quasi-coherent family of sheaves $\tilde{\mathfrak{w}}^{\dag\kappa^{o\un}}_{w, 1}$ over
$X_1(p^n)(v)_1\times \mathcal{W}(w)^o_1$ is   constant: the sheaf
$\tilde{\mathfrak{w}}^{\dag\kappa^{o\un}}_{w,1}$ is the inverse image on $X_1(p^n)(v)_1\times \mathcal{W}(w)^o_1$
of a sheaf defined on $X_1(p^n)(v)_1$.
\end{coro2}

\begin{demo} In view of lemmas \ref{lem--1},  \ref{lem--2} and   \ref{lem--3}, the sheaf
$\tilde{\mathfrak{w}}^{\dag\kappa^{o\un}}_{w,1}$ equals the pull back  of the sheaf
$\omega^{\dag\kappa^o}_{w,1}$ for any $\kappa^o \in \mathcal{W}(w)^o(K)$.
\end{demo}

\subsubsection{D\'evissage of the sheaves}\label{sect-devissage}

 We have just given  a description of the local sections of $\tilde{\mathfrak{w}}^{\dag\kappa^o}_{w}$. This
description depends on the choice of a basis $f_1,\ldots,f_g$ of $\mathcal{F}$ lifting the universal basis
$e_1,\ldots,e_g$ of $H_1^D(R'[1/p])$.   We'd now like to investigate the dependence   on the choice of the basis $f_1,\ldots,f_g$.

Let $(f'_1,\ldots,f'_g)$ be an other compatible choice of basis for $\mathcal{F}$ with  $P = \mathrm{Id}_g + p^wM \in \GL(R')$
be the base change matrix from $f_1,\ldots,f_g$ to $f'_1,\ldots,f'_g$. This second trivialization of $\mathcal{F}$ determines new
coordinates $X'_{i,j}, X'_k$ on $\mathfrak{IW}_w^+\vert_{\Spf~R}$.
%The following lemma  will be enough for our purposes:

\begin{lem2}\label{lem-Pisic}  We have the congruences:
\begin{eqnarray*}
 X'_{i,j} &= & X_{i,j} + m_{i,j}  \mod p^w R\langle\langle X_{s,t}, X_u\rangle\rangle\\
 X'_k & = & X_k + m_{k} \mod p^w R\langle\langle X_{s,t}, X_u\rangle\rangle
 \end{eqnarray*}

for all $1 \leq j < i \leq g$ and all $1 \leq k \leq g$ where $m_{i,j}$ is the coefficient of $M$ on the $i$-th line and $j$-th column and $m_k$
the coefficient on the $k$-th diagonal entry.
\end{lem2}

\begin{demo}
Let $\underline{X}$ and  $\underline{X}'$ be the lower triangular  matrices with $X_{i,j}$ and $X'_{i,j}$ on the $i$-th line and
$j$-th column and $X_k$, $X'_k$ on the $k$-th diagonal entry.
We have $$(\mathrm{Id}_g + p^w M)(\mathrm{Id}_g + p^w\underline{X} )  = \mathrm{Id}_g + p^w (M + \underline{X})  + O(p^{2w}).$$
There is a unique  upper triangular matrix $N$ with $0$ on the diagonal  such that
$$(\mathrm{Id}_g + p^w M)(\mathrm{Id}_g + p^w\underline{X} )(\mathrm{Id}_g + p^w N) = \mathrm{Id}_g + p^w\underline{X}'.$$
We have $$(\mathrm{Id}_g + p^w M)(\mathrm{Id}_g + p^w\underline{X} )(\mathrm{Id}_g + p^w N) =  \mathrm{Id}_g + p^w (M + \underline{X} + N)  + O(p^{2w}).$$
We deduce that $N = (-m_{i,j})_{1 \leq i < j \leq g} + O(p^w)$ and that $M + \underline{X} + N = \underline{X}' \mod p^w$.
\end{demo}

\bigskip

\begin{coro2}\label{coro-dev}Let $\kappa^o \in \mathcal{W}(w)^o(K)$.
The quasi-coherent sheaf $\tilde{\mathfrak{w}}^{\dag\kappa^o}_{w,1}$ is  an inductive limit of coherent sheaves which are extensions of the  trivial sheaf.
\end{coro2}

\begin{demo}
Covering $\mathfrak{X}_1(p^n)(v)$ by affine open formal sub-schemes $\Spf~R$ and choosing a basis
$(f_1,\ldots,f_g)$ of $\mathcal{F}$ compatible with $\psi $, we can expand the
sections of $\tilde{\mathfrak{w}}^{\dag\kappa^o}_{w,1}\vert_{\Spf~R}$ as polynomials in  the variables $(X_{i,j})_{1 \leq j<i \leq g}$. By  lemma
\ref{lem-Pisic},
the total degree of a section is independent of the choice of the basis, so we can write
$\tilde{\mathfrak{w}}^{\dag\kappa^o}_{w,1}$ as the inductive limit as $r \in
\N$ grows of the sub-sheaves $\tilde{\mathfrak{w}}^{\dag\kappa^o}_{w,1}\vert^{\leq r}$ of sections of degree bounded by $r$. In
$\tilde{\mathfrak{w}}^{\dag\kappa^o}_{w,1}\vert^{\leq r}$, we can consider for all $1 \leq k, l \leq g$, the sub-sheaf
$\tilde{\mathfrak{w}}^{\dag\kappa^o}_{w,1}\vert^{\leq r, k,l}$ locally generated  by the polynomials of degree less than $r$ in the variables $X_{i,j}$ for $i \geq
k$ and $j \leq l$. This sub-sheaf is well defined by lemma \ref{lem-Pisic}. The sheaves
$$\tilde{\mathfrak{w}}^{\dag\kappa^o}_{w,1}\vert^{\leq r, k,l} ~ \mathrm{mod}~\tilde{\omega}^{\dag\kappa^o}_{w,1}\vert^{\leq r, k-1,l} +
\tilde{\mathfrak{w}}^{\dag\kappa^o}_{w,1}\vert^{\leq r, k,l-1}$$ are  isomorphic to the trivial sheaf.
\end{demo}

\bigskip

In general, one can always write $\tilde{\mathfrak{w}}^{\dag\kappa}_{w,m}$ as an inductive limit of coherent sheaves in a reasonable way as follows. Let $\cup_i
\mathfrak{U}_i$  be a finite Zariski cover of $\mathfrak{X}_1(p^n)$ by affine formal sub-schemes $\mathfrak{U}_i = \Spf~R_i$ such that over each $\mathfrak{U}_i$ we
have the description of the sections of  $\tilde{\mathfrak{w}}^{\dag \kappa^o}_{w}(R_i)$ as in lemma \ref{lem--1}.  We let $R_{i,m}$ be the reduction modulo $\varpi^m$
of $R_i$ and $U_{i,m} = \Spec~R_{i,m}$.  We have that $\tilde{\mathfrak{w}}^{\dag \kappa^o}_{w,m}(U_{i,m}) \simeq R_{i,m} [X_{i,j}, 1 \leq i < j \leq g ]$. We let
$\tilde{\mathfrak{w}}^{\dag \kappa^o\leq r}_{w,i,m}$ be the  coherent sheaf over $U_{i,m}$ associated to the sub-module of $\tilde{\mathfrak{w}}^{\dag
\kappa^o}_{w,m}(U_{i,m})$ of polynomials of degree bounded by $r$. We  also let $\tilde{\mathfrak{w}}^{\dag \kappa^o\leq r}_{w,m}$ be the sub-sheaf of
$\tilde{\mathfrak{w}}^{\dag \kappa^o}_{w,m}$ defined as the kernel of $$ \prod_i \tilde{\mathfrak{w}}^{\dag \kappa^o\leq r }_{w,i,m} \longrightarrow \prod_{i,j}
\tilde{\mathfrak{w}}^{\dag \kappa^o}_{w,m}\vert_{U_{i,m}\cap U_{j,m}}.$$It is a quasi-coherent sheaf as it is the kernel of a morphism of quasi-coherent sheaves.
Furthermore, for every $i$ the natural map $\tilde{\mathfrak{w}}^{\dag \kappa^o\leq r}_{w,m}\vert_{U_{i,m}}\to \tilde{\mathfrak{w}}^{\dag \kappa^o\leq r }_{w,i,m}$ is
injective. As $\tilde{\mathfrak{w}}^{\dag \kappa^o\leq r }_{w,i,m}$ is a coherent sheaf over $U_{i,m}$ for every $i$, we deduce that  $\tilde{\mathfrak{w}}^{\dag
\kappa^o\leq r}_{w,m}\vert_{U_{i,m}}$ is a coherent sheaf as well. Furthermore, as $\colim_{r}\tilde{\mathfrak{w}}^{\dag \kappa^o\leq
r}_{w,j,m}=\tilde{\mathfrak{w}}^{\dag \kappa^o}_{w,m}\vert_{U_{j,m}}$ for every $j$, we conclude that $\tilde{\mathfrak{w}}^{\dag \kappa^o }_{w,m} = \colim_{r}
\tilde{\mathfrak{w}}^{\dag \kappa^o \leq r}_{w,m}$. Of course, the sheaves $\tilde{\mathfrak{w}}^{\dag \kappa^o \leq r}_{w,m}$ depend on the choice of the cover
$\cup_i\mathfrak{U}_i$ and on the choice of a basis of $\mathcal{F}\vert_{\mathfrak{U}_i}$ if $m \geq 2$.  If we fix two choices of cover and basis, we get two
inductive limits:
$$\tilde{\mathfrak{w}}^{\dag \kappa^o }_{w,m} = \colim_{r} \tilde{\mathfrak{w}}^{\dag \kappa^o \leq r, 1}_{w,m} =  \colim_{r}
\tilde{\mathfrak{w}}^{\dag \kappa^o \leq r, 2}_{w,m}.$$ It
is easy to check that in any case, for all $r$, there is $r'\leq  r''$ such that

$$ \tilde{\mathfrak{w}}^{\dag \kappa^o \leq r', 2}_{w,m} \hookrightarrow \tilde{\mathfrak{w}}^{\dag \kappa^o \leq r, 1}_{w,m}
\hookrightarrow \tilde{\mathfrak{w}}^{\dag \kappa^o \leq r'', 2}_{w,m}.$$

The preceding  discussion still makes sense for $\tilde{\mathfrak{w}}^{\dag \kappa^{o\un}}_{w,m}$.

\subsection{The base change theorem}

\subsubsection{ The boundary of the compactification}

Let $V = \oplus_{i=1}^{2g} \Z e_i$ be a free $\Z$-module of rank $2g$ equipped with the symplectic form of matrix $J=\left(\begin{matrix} 0 & 1_g \cr -1_g & 0 \end{matrix} \right)$.
For all totally isotropic direct factor $V'$ we consider $C(V/{V'}^\bot)$ the cone of symmetric semi-definite bilinear
forms on $V/{V'}^\bot \otimes \R$ with rational radical.  If $V' \subset V''$ we have an inclusion
$C(V/{V''}^\bot) \subset C(V/{V'}^\bot)$. We let $\mathfrak{C}$ be the set of all totally isotropic direct factors
$V' \subset V$ and  $\mathcal{C}$ be the quotient of the disjoint union: $$ \coprod_{V' \in \mathfrak{C}} C(V/{V'}^\bot)$$
by the equivalence relation induced by the inclusions $C(V/{V''}^\bot) \subset C(V/{V'}^\bot)$.

The given basis of $V$ gives a ``principal level $N$ structure": $\psi_N\colon  \Z/N\Z^{2g} \simeq  V/NV $. The vectors $e_1,\ldots,e_g$ give a ``Siegel principal
level $p^n$ structure":
$$ \psi \colon (\Z/p^n\Z)^g \hookrightarrow V/p^nV.$$

Let $\Gamma$ be the congruence sub-group of $\G(\Z)$ stabilizing $\psi_N$ and $\Gamma_1(p^n)$ be the congruence subgroup stabilizing $\psi$ and $\psi_N$.
Let $\mathcal{S}$ be a rational polyhedral decomposition of $\mathcal{C}$ which is $\Gamma$-admissible
(see \cite{FC}, section IV. 2).

We now recall some facts about the toroidal compactifications following the presentation adopted in \cite{Str}
(see for example sections 1.4.3, 2.1 and 2.2). For any $V' \in \mathfrak{C}$ and $\sigma \in \mathcal{S}$ lying  in
the interior of $C(V/{V'}^\bot)$  there is a diagram:

\begin{eqnarray*}
\xymatrix{ \mathcal{M}_{V'} \ar[r] \ar[rd] & \mathcal{M}_{V',\sigma} \ar[r]\ar[d] & {\mathcal{M}}_{V',\mathcal{S}} \ar[ld] \\
& \mathcal{B}_{V'} \ar[d] &  \\
& Y_{V'}& }
\end{eqnarray*}
Where:
\begin{itemize}

\item $Y_{V'}$ is the moduli space of principally polarized abelian schemes of dimension $g-r$, with $r = \mathrm{rk}_\Z V'$, with principal level $N$ structure;
\item Let $A_{V'}$ is the universal abelian scheme over $Y_{V'}$. The scheme $\mathcal{B}_{V'} \rightarrow Y_{V'}$ is an abelian scheme. Moreover, there is an
isogeny  $i \colon \mathcal{B}_{V'} \rightarrow A_{V'}^{r}$ of degree a power of $N$.  Over $\mathcal{B}_{V'}$, there is a  universal semi-abelian scheme
$$ 0 \rightarrow \T_{V'} \rightarrow  \tilde{G}_{V'} \rightarrow A_{V'} \rightarrow 0$$
where $\T_{V'}$ is the torus  $V'\otimes_\Z \mathbb{G}_m$;
\item $\mathcal{M}_{V'}$ is a moduli space of principally polarized  $1$-motives with principal level $N$ structure, the map $\mathcal{M}_{V'} \rightarrow \mathcal{B}_{V'}$ is a torsor under a torus  with character group $S_{V'}$, isogenous to $\mathrm{Hom}(\mathrm{Sym}^2 V/{V'}^\bot, \mathbb{G}_m)$;
\item $\mathcal{M}_{V'} \rightarrow\mathcal{M}_{V',\sigma}$ is an affine toroidal embedding attached to the cone $\sigma \in C(V/{V'}^\bot)$ and
the $\Z$-module $S_{V'}$;
\item $\mathcal{M}_{V'} \rightarrow \mathcal{M}_{V',\mathcal{S}}$ is a toroidal embedding, locally of finite type associated to the polyhedral decomposition $\mathcal{S}$. The scheme $\mathcal{M}_{V',\sigma}$ is open affine in
$\mathcal{M}_{V',\mathcal{S}}$.
\end{itemize}

We shall denote by $Z_{\sigma}$ the closed stratum in $\mathcal{M}_{V',\sigma}$ and by $Z_{V'}$ the closed stratum in $\mathcal{M}_{V',\mathcal{S}}$. We let
$\Gamma_{V'}$ be  the stabilizer of $V'$ in $\Gamma$, it acts on $C(V/{V'}^\bot)$ and on the toroidal embedding  ${\mathcal{M}}_{V',\mathcal{S}}$. We assume that
$\{\mathcal{S}\}$ is a smooth and projective admissible polyhedral decomposition. The existence is guaranteed by the discussion in \cite[\S V.5]{FC}. Let $Y$ be the moduli space classifying principally polarized abelian varieties over $\ocal_K$ with principal level $N$ structure. Let $Y\subset X$ be
the toroidal compactification  associated to $\mathcal{S}$. The hypothesis that $\mathcal{S}$ is projective guarantees that $X$ is a projective scheme and not simply
an algebraic space.

\begin{thm2}[\cite{FC}]\label{FC}
\begin{enumerate}

\item The toroidal compactification $X$ carries a stratification indexed by $\mathfrak{C}/\Gamma$. For all $V' \in \mathfrak{C}$  the completion of $X$ along the
$V'$-stratum is isomorphic to the space $\widehat{\mathcal{M}}_{V',\mathcal{S}}/\Gamma_{V'}$  where $\widehat{\mathcal{M}}_{V',\mathcal{S}}$ is the completion of
$\mathcal{M}_{V',\mathcal{S}}$ along the strata $Z_{V'}$.

\item The toroidal compactification $X$ carries a finer stratification indexed by  $\mathcal{S}/\Gamma$. Let $\sigma \in \mathcal{S}$. The corresponding
stratum in $X$ is $Z_{\sigma}$. Let $Z$ be an affine open subset of $Z_{\sigma}$. Then the henselization of $X$ along $Z$ is  isomorphic to the
henselization of $\mathcal{M}_{V,\sigma}$ along $Z$.
\end{enumerate}
\end{thm2}

\medskip

 The Hasse invariant of the semi-abelian scheme on the special fiber of
$\widehat{\mathcal{M}}_{V',\mathcal{S}}$ is the Hasse invariant of  the abelian part of the semi-abelian scheme. We
can thus identify $\Ha$ with the Hasse invariant defined over the special fiber of $Y_{V'}$.

Recall from section \ref{sec:blowup} that we have define a formal scheme $\mathfrak{X}(v)$ with a morphism $\mathfrak{X}(v)\to \mathfrak{X}$ to the formal completion $\mathfrak{X}$ of $X$ along its special fiber. We can now describe the boundary of $\mathfrak{X}(v)$, namely the complement of the inverse image $\mathfrak{Y}(v)$ of the formal completion $\mathcal{Y}\subset \mathfrak{X}$ of $Y$. We will need some notations:
\begin{itemize}
\item $ \mathfrak{Y}_{V'}$ is the completion of $Y_{V'}$ along its special fiber;
\item  $\mathfrak{Y}_{V'}(v)$ is largest formal open subset of the formal admissible blow-up of
$\mathfrak{Y}_{V'}$ along the ideal $(p^v, \Ha)$ where the ideal $(p^v, \Ha)$ is generated by $\Ha$ (see  \ref{sec:blowup});

\item $\mathfrak{B}_{V'}$ is the completion of $\mathcal{B}_{V'}$ along its special fiber and
$\mathfrak{B}_{V'}(v)$ is the fiber product $\mathfrak{B}_{V'}\times_{\mathfrak{Y}_{V'}} \mathfrak{Y}_{V'}(v)$.
We define similarly  $\mathfrak{M}_{V'}(v)$, $\mathfrak{M}_{V',\sigma}(v)$ and $\mathfrak{M}_{V',\mathcal{S}}(v)$, $\mathfrak{Z}_{\sigma}(v)$, $\mathfrak{Z}_{V'}(v)$.
\end{itemize}

\begin{prop2} The formal scheme $\mathfrak{X}(v)$ has a fine stratification   indexed by $\mathcal{S}/\Gamma$
over a coarse stratification indexed by $\mathfrak{C}/\Gamma$.  For all $\sigma \in \mathcal{S}$ the corresponding strata is $\mathfrak{Z}_{V',\sigma}(v)$. For any
open affine sub-scheme $\mathfrak{Z}$ of $\mathfrak{Z}_{V',\sigma}(v)$ the henselization of $\mathfrak{X}(v)$ along $\mathfrak{Z}$ is isomorphic to the henselization
of $\mathfrak{M}_{V',\sigma}(v)$ along $\mathfrak{Z}$.  For all $V' \in \mathfrak{C}'$ the completion of $\mathfrak{X}(v)$ along the $V'$-strata  of
$\mathfrak{X}(v)$ is isomorphic to the completion  $\widehat{\mathfrak{M}}_{V',\mathcal{S}}(v)$ of $\mathfrak{M}_{V',\mathcal{S}}(v)$ along $\mathfrak{Z}_{V'}(v)$.
\end{prop2}
\begin{demo} As admissible blow-ups commute with flat base change, this follows easily from theorem \ref{FC}.
\end{demo}

\medskip

In section \ref{sec:blowup}  we have introduced a covering $\mathfrak{X}_{1}(p^n)(v) \to \mathfrak{X}(v)$. We now describe the boundary of  $\mathfrak{X}_{1}(p^n)(v)$, i.e., the complement of the inverse image of $\mathcal{Y}(v)\subset \mathfrak{X}(v)$. Let $\mathfrak{C}'$ be the subset of $\mathfrak{C}$ of totally isotropic spaces satisfying
$\psi((\Z/p^n\Z)^g ) \subset {V'}^\bot$.  We let $\mathcal{C}'$ be the   quotient of the disjoint union: $$ \coprod_{V' \in \mathfrak{C}'} C(V/{V'}^\bot)$$ by the
equivalence relation induced by the inclusions $C(V/{V''}^\bot) \subset C(V/{V'}^\bot)$.  Clearly, $\Gamma_1(p^n)$ acts on $\mathfrak{C}'$ and the polyhedral
decomposition $\mathcal{S}$ induces a polyhedral decomposition $\mathcal{S}'$ of $\mathcal{C}'$ which is $\Gamma_1(p^n)$-admissible.

Let $V' \in \mathfrak{C}'$ of rank $r$. We have an exact sequence
$$ 0 \rightarrow V'/p^nV'\rightarrow {V'}^{\bot}/p^n{V'}^\bot \rightarrow V'^{\bot}/V' + p^n{V'}^{\bot} \rightarrow 0.$$
The image of $\psi (\Z/p^n\Z^g )$ in $V'^{\bot}/V' + p^n{V'}^{\bot}$ is a totally isotropic subspace of rank $p^{g-r}$, that we denote by $W$. The map $\psi$ also
provides a section $s \colon W \hookrightarrow {V'}^{\bot}/p^n{V'}^{\bot}$. By duality, $\psi$ provides an isomorphism that we again denote by $\psi$:
$$ \psi \colon (\Z/p^n\Z)^g \simeq W^{\vee} \oplus V/(p^nV + {V'}^\bot)$$

To describe the local charts, we need some notations.

\begin{itemize}
\item $\cY_{V'}(v)$ is the rigid fibre of $\mathfrak{Y}_{V'}(v)$.
\item We let $H_{n, V'}$ be the canonical subgroup of the universal abelian scheme $\mathfrak{A}_{V'}$ over $\mathfrak{Y}_{V'}(v)$ and we denote  by  ${\cY}_1(p^n)_{V'}(v)$  the torsor $\mathrm{Isom}_{\cY_{V'}(v)}(W^{\vee} , H_{n,V'}^D)$. We let $\psi_{V'}$ be the universal trivialisation.

\item $ \mathfrak{Y}_1(p^n)_{V'}(v)$ is the normalization of  $\mathfrak{Y}_{V'}(v)$ in $\cY_1(p^n)_{V'}(v)$.

\item  Recall that there is an isogeny $i \colon  \mathfrak{B}(v) \rightarrow \mathfrak{A}_{V'}^r$ of degree a power of $N$. We let $i_{c\an} \colon  \mathfrak{A}_{V'}
\rightarrow \mathfrak{A}_{V'}/H_{n,V'}$ be the canonical projection.   We set $$\mathfrak{B}_1(p^n)_{V'} = \mathfrak{B}(v) \times_ {i,\mathfrak{A}_{V'}^r,
i_{c\an}^{D}} (\mathfrak{A}_{V'}/H_{n,V'})^r.$$ The abelian scheme  $\mathfrak{B}_1(p^n)_{V'} \rightarrow \mathfrak{Y}_1(p^n)_{V'}$ carries a universal diagram:
\begin{eqnarray*}
\xymatrix{ {V'}^{\vee} \ar[r] &  \mathfrak{A}_V'\\
{V'}^{\vee} \ar[r]\ar[u]^{\mathrm{Id}} &  \mathfrak{A}_{V'}/H_{n,V'}\ar[u]^{i_{c\an}^D}}
\end{eqnarray*}
which is  equivalent to the diagram:
\begin{eqnarray*}
\xymatrix{ 0 \ar[r] & \mathfrak{T}_{V'} \ar[r]\ar[d]^{\mathrm{Id} }& \tilde{\mathfrak{G}}_{V'}\ar[d]^{j} \ar[r] & \mathfrak{A}_{V'}\ar[d]^{i_{c\an} }\ar[r] & 0 \\
0 \ar[r] & \mathfrak{T}_{V'} \ar[r] & \tilde{\mathfrak{G}}'_{V'} \ar[r] & \mathfrak{A}_{V'}/H_{n,V'} \ar[r] & 0}
\end{eqnarray*}

The group $\mathrm{Ker} j $ is a lift to $\tilde{\mathfrak{G}}_{V'}[p^n]$ of $H_{n,V'}$.

Over the rigid fiber, we thus have $H_n = \T_{V'}[p^n] \oplus \mathrm{Ker} j $ and $H_n^D = V/(p^nV + {V'}^\bot) \oplus \mathrm{Ker} j^D$.  The map $\psi_{V'}$
provides an isomorphism $W^{\vee} \simeq {Ker} j^D$. The map $\psi$ now provides an isomorphism $\psi \colon \Z/p^n\Z^g \simeq H_n^D$.

\item We define $\mathfrak{M}_1(p^n)_{V'}(v)$, $\mathfrak{M}_1(p^n)_{V',\sigma}(v)$ and $\mathfrak{M}_1(p^n)_{V',\mathcal{S}'}(v)$,
$\mathfrak{Z}_1(p^n)_{\sigma}(v)$, $\mathfrak{Z}_1(p^n)_{V'}(v)$ by fiber product of  $\mathfrak{M}_{V'}(v)$, $\mathfrak{M}_{V',\sigma}(v)$,
$\mathfrak{M}_{V',\mathcal{S}'}(v)$, $\mathfrak{Z}_{\sigma}(v)$, $\mathfrak{Z}_{V'}(v)$ with $\mathfrak{B}_1(p^n)_{V'}$ over $\mathfrak{B}_{V'}$.
\end{itemize}

\begin{prop2}\label{desc-toro}
The formal scheme $\mathfrak{X}_1(p^n)(v)$ has a fine stratification   indexed by $\mathcal{S}'/\Gamma_1(p^n)$ over a coarse stratification indexed by
$\mathfrak{C}'/\Gamma_1(p^n)$.  For all $\sigma \in \mathcal{S}'$ the corresponding strata is $\mathfrak{Z}_1(p^n)_{V',\sigma}(v)$. For any open affine sub-scheme
$\mathfrak{Z}$ of $\mathfrak{Z}_1(p^n)_{V',\sigma}(v)$ the henselization of $\mathfrak{X}_1(p^n)(v)$ along $\mathfrak{Z}$ is isomorphic to the henselization of
$\mathfrak{M}_1(p^n)_{V',\sigma}(v)$ along $\mathfrak{Z}$.  For all $V' \in \mathfrak{C}'$ the completion of $\mathfrak{X}_1(p^n)(v)$ along the $V'$-strata  of
$\mathfrak{X}_1(p^n)(v)$ is isomorphic to the completion $\widehat{\mathfrak{M}_1(p^n)}_{V',\mathcal{S}'}(v)$ of $\mathfrak{M}_1(p^n)_{V',\mathcal{S}'}(v)$ along $\mathfrak{Z}_1(p^n)_{V'}(v)$.
\end{prop2}

\begin{demo} Over the rigid fiber, this is  a variant of theorem \ref{FC}.
In particular we know that the rigid fiber of the local charts of level $\Gamma_1(p^n)$ are correctly described. It is now easy to check that our formal  local
charts of level $\Gamma_1(p^n)$ are normal, and as a result,  they are the normalization of the formal  local charts of level $\Gamma$. Since normalization commutes
with \'etale localization, we conclude.
\end{demo}
\begin{rem2} The process of obtaining toroidal compactifications by normalization is studied in \cite{FC} p. 128,  in a different situation.
\end{rem2}

\subsubsection{Projection to the minimal compactification}\label{sec:mimimalcompact}

There is a projective scheme $X^\star$ called minimal compactification and a  proper morphism
$\xi\colon  X \rightarrow X^\star$. Let us recall some properties of the minimal compactification:

\begin{thm2}[\cite{FC},  Thm.~V.2.7]\label{thm-mini} The minimal compactification $X^\star$ is stratified by
$\mathfrak{C}/\Gamma$ and the morphism $\xi$ is compatible with the stratification.
\begin{itemize}
\item for any $V' \in \mathfrak{C}$ the $V'$-strata is $Y_{V'}$,
\item For a geometric point $\bar{x}$ of the $V'$ strata,  we have
$$ \widehat{\oscr_{X^\star,\bar{x}}} =\Big( \prod_{\lambda\in S_{V'} \cap C(V/{V'}^\bot)^\vee}
\HH^0\big(\widehat{\mathcal{B}_{V', \bar{x}}}, \mathcal{L}(\lambda)\big)\Big)^{\Gamma_{V'}}$$
where
\begin{itemize}
\item[$\bullet$] $ \widehat{\oscr_{X^\star,\bar{x}}}$ is the completion of the strict henselization of
$\oscr_{X^\star}$ along $\bar{x}$,
\item[$\bullet$] $\widehat{\mathcal{B}_{V', \bar{x}}}$ is the formal completion of $\mathcal{B}_{V'}$ along its fiber over $\bar{x}$,
\item[$\bullet$] $\mathcal{L}(\lambda)$ is the invertible sheaf over $\mathcal{B}_{V'}$ of $\lambda$-homogeneous  functions on $\mathcal{M}_{V'}$.
\end{itemize}
\end{itemize}
\end{thm2}

The Hasse invariant $\Ha$ descends to a function on the special fiber of $X^\star$. We let $\mathfrak{X}^\star$ be the completion of $\mathfrak{X}$ along its special fiber. We denote by $\mathfrak{X}^\star(v)$
the $p$-adic completion of the normalization of the greatest  open sub-scheme of the blow-up of
$\mathfrak{X}^\star$ along the ideal $(p^v,\Ha)$, on which this ideal is generated by $\Ha$.

\begin{prop2}
 For all $V' \in \mathfrak{C}'$ the $V'$-stratum of $\mathfrak{X}^\star(v)$ is $\mathfrak{Y}_{V'}(v)$.
\end{prop2}
\begin{demo}  This follows from the fact that  $(\Ha, p^v)$ is a regular sequence in $Y_{V'}$ and in $X^\star$. This implies that the blow-up  along $(\Ha, p^v)$ is
in both cases a closed sub-scheme of a relative $1$-dimensional projective space with equation $T \Ha - Sp^v$ (where $T$, $S$ are homogeneous coordinates).
\end{demo}
\bigskip

We have a diagram:

\begin{eqnarray*}
\xymatrix{ \mathfrak{X}_1(p^n)(v) \ar[rd]^{\eta}\ar[r]^{\pi_4} & \mathfrak{X}(v) \ar[d]^{\xi}\\
 & \mathfrak{X}^\star(v)}
 \end{eqnarray*}

We  let $X_1(p^n)(v)_m$, $X^\star(v)_m$,  $\mathcal{M}_1(p^n)_{V',\sigma}(v)_m$, $\mathcal{B}_1(p^n)_{V'}(v)_m$, $\ldots$   be the schemes  obtained by reduction modulo
$\varpi^m$ from $\mathfrak{X}_1(p^n)(v)$, $\mathfrak{M}_1(p^n)_{V',\sigma}(v)$, $\mathfrak{B}_1(p^n)_{V'}(v)$, $\ldots$. We will also consider the projection $\eta\times
1\colon \mathfrak{X}_1(p^n)(v)\times \mathfrak{W}(w)^o \rightarrow \mathfrak{X}^\star(v)\times \mathfrak{W}(w)^o$. Finally, we use $D$ to denote the boundary  in
$\mathfrak{X}_{1}(p^n)$, $\mathfrak{X}_1(p^n)(v)$, $\ldots$

\begin{thm2}\label{thm:baschange} Consider the following diagram for $l$, $m\in \N$ and $m\ge l$:
\begin{eqnarray*}
\xymatrix{X_1(p^n)(v)_{l} \ar[r]^{i}\ar[d]^{\eta_l} & X_1(p^n)(v)_{m} \ar[d]^{\eta_m}\\
X^\star(v)_l \ar[r]^{i'} & X^\star(v)_m}
\end{eqnarray*}
 Then we have the base change property:
$$ {i'}^\ast (\eta_m)_{\ast} \tilde{\mathfrak{w}}^{\dag\kappa^o}_{w,m}(-D) = \eta_{l,\ast} \tilde{\mathfrak{w}}^{\dag\kappa^o}_{w,l} (-D).$$
In particular,  $\bigl(\eta_\ast \tilde{\mathfrak{w}}^{\dag\kappa^o}_w(-D)\bigr)$ is a formal Banach sheaf over $\mathfrak{X}^\star (v)$. Similarly, also $(\eta\times
1)_\ast \bigl(\tilde{\mathfrak{w}}^{\dag \kappa^{o\un}}_w(-D)\bigr)$ is a  formal Banach sheaf over $\mathfrak{X}^\star (v)\times \mathfrak{W}(w)^o$.
\end{thm2}

\begin{demo} The property is local for the $fppf$-topology on $X^\star(v)_m$.  Let $\bar{x} \in X^\star(v)_m$ be a geometric point. We can write
$\tilde{\mathfrak{w}}^{\dag\kappa^o}_{w,m}$ as an inductive limit of coherent sheaves $\colim \tilde{\mathfrak{w}}^{\dag\kappa^o\leq r}_{w,m}$ by the discussion at
the end of section \ref{sect-devissage}.  By the theorem on formal functions \cite[\S4]{EGAIII}, and because direct images commute with inductive limits,
we have that
$$ \eta_{m,\ast} \tilde{\mathfrak{w}}^{\dag\kappa^o}_{w,m}(-D) \big( \widehat{X^\star(v)_{m,\bar{x}}}\big) =
\colim_r \HH^0\big( \widehat{X_1(p^n)(v)_{m,\bar{x}}}, \tilde{\mathfrak{w}}^{\dag\kappa^o\leq r}_{w,m}(-D)\big)$$ where
$\widehat{X^\star(v)_{m,\bar{x}}}$ is the
completion of the strict henselization of  $X^\star(v)_{m}$ at $\bar{x}$ and $\widehat{X_1(p^n)(v)_{m,\bar{x}}}$ is the completion of
$X_1(p^n)(v)_{m}$ along
$\eta^{-1}_m(\bar{x})$.  This completion is isomorphic to a finite disjoint union of spaces
$\widehat{\mathcal{M}_1(p^n)_{V',\mathcal{S}'}(v)}_{m,\bar{y}}/\Gamma_1(p^n)_{V'}$ where $\bar{y}$ is a geometric point in $Y_1(p^n)_{V'}(v)_m$.  This space fits in
the following  diagram:
\begin{eqnarray*}
\xymatrix{ \widehat{\mathcal{M}_1(p^n)_{V',\mathcal{S}'}}(v)_{m,\bar{y}} \ar[r]^{h_2} \ar[d]^{h_1} &
\widehat{\mathcal{M}_1(p^n)_{V',\mathcal{S}'}}(v)_{m,\bar{y}}/\Gamma_1(p^n)_{V'}\ar[d]_{h_3} \ar[r] & \widehat{X_1(p^n)(v)_{m,\bar{x}}}\\
\widehat{\mathcal{B}_1(p^n)_{V'}(v)}_{m,\bar{y}}\ar[r]^{h_4}& \widehat{Y_1(p^n)_{V'}(v)_{m,\bar{y}}}}
\end{eqnarray*}
where $\widehat{\mathcal{B}_1(p^n)_{V'}(v)}_{m,\bar{y}}$, $\widehat{\mathcal{M}_1(p^n)_{V',\mathcal{S}}(v)}_{m,\bar{y}}$ are the formal completions of
$\mathcal{B}_1(p^n)_{V'}(v)_{m}$ and $\mathcal{M}_1(p^n)_{V',\mathcal{S}}(v)_{m}$ over their  fibers at $\bar{y}$. We are thus reduced to prove the
following:\smallskip

 {\em Claim:}  the formation of the module $$\colim_r \HH^0\big( \widehat{\mathcal{M}_1(p^n)_{V',\mathcal{S}'}}(v)_{m,\bar{y}}/\Gamma_1(p^n)_{V'}
, \tilde{\mathfrak{w}}^{\dag\kappa^o\leq r}_{w,m}(-D)\big)$$commutes with reduction modulo $\varpi^l$ for $l \leq m$.\smallskip

\noindent We provide two proofs.

\smallskip

\noindent {\bf First Proof:}\enspace We identify the module in the claim with $$ \HH^0\Big(\Gamma_1(p^n)_{V'}, \colim_r \HH^0\big(
\widehat{\mathcal{M}_1(p^n)_{V',\mathcal{S}'}}(v)_{m,\bar{y}},h_2^\ast \tilde{\mathfrak{w}}^{\dag\kappa^o\leq r}_{w,m}(-D)\big)\Big).$$
We remark that we have a
formal semi-abelian scheme $\tilde{\mathfrak{G}}_{V'}$ over $\widehat{\mathcal{B}_1(p^n)_{V'}(v)}_{m,\bar{y}}$, extension of the
universal $g-r$-dimensional formal
abelian scheme $\mathfrak{A}_{V'}$ by the $r$-dimensional formal torus $\widehat{T}_{V'}:=V'\otimes \widehat{\mathbb{G}}_m$. It admits a canonical subgroup scheme
$H_n\subset \tilde{\mathfrak{G}}_{V'}[p^n]$ extension of the canonical subgroup $H_{n,V'}\subset \mathfrak{A}_{V'}[p^n]$ by $\widehat{T}_{V'}[p^n]$ and, as
explained in propositions \ref{prop:Filomega1} and \ref{prop-HTann}, the tautological principal $p^n$-level structure and the Hodge-Tate morphism  for $H_n$ provide
a morphism $\HT \colon (\Z/p^n \Z)^g \to \omega_{\tilde{\mathfrak{G}}_{V'}}/p^w$. Thus, proceeding as in proposition \ref{prop-FHT} we obtain a sheaf $\mathcal{F}\subset
\omega_{\tilde{\mathfrak{G}}_{V'}} $ and an isomorphism $\HT_w\colon (\Z/p^n \Z)^g\otimes \oscr_{\widehat{\mathcal{B}_1(p^n)_{V'}(v)}_{m,\bar{y}}} \longrightarrow
\mathcal{F}_w$. The Levi quotient of  $\Gamma_1(p^n)_{V'}$  is a subgroup of $\mathrm{GL}_g(V') \times \mathrm{GSp}({V'}^\bot/V')$. We let $\Gamma'_1(p^n)_{V'}$ be
the projection of $\Gamma_1(p^n)_{V'}$ onto its  $\mathrm{GL}_g(V')$ factor. As ${\widehat{\mathcal{B}_1(p^n)_{V'}(v)}_{m,\bar{y}}}$ classifies extensions by
$\mathfrak{A}_{V'}$ and $\widehat{T}_{V'}$ with a level $N$-structure, the group $\Gamma'_1(p^n)_{V'}$ acts on $\widehat{T}_{V'}$, on
${\widehat{\mathcal{B}_1(p^n)_{V'}(v)}_{m,\bar{y}}}$ and on $\mathcal{F}$ and we get an induced action of the group $\Gamma_1(p^n)_{V'}$ through the natural
morphism $\Gamma_1(p^n)_{V'} \to \Gamma'_1(p^n)_{V'}$. We thus obtain an action of $\Gamma_1(p^n)_{V'}$ on $\mathcal{F}$ so that $\HT_w$ is
$\Gamma_1(p^n)_{V'}$-equivariant. The functoriality of $\mathcal{F}$ and $\HT_w$ implies that their base change  via $h_1$ coincide with the base change via $h_2$
of the sheaf $\mathcal{F}$ and the map $ \HT_w$ for the universal degenerating semi-abelian scheme over $\widehat{X_1(p^n)(v)_{m,\bar{x}}}$. As in \S\ref{sect-main}
we get an affine formal scheme $\mathfrak{IW}^+_{w,m}\longrightarrow \widehat{\mathcal{B}_1(p^n)_{V'}(v)}_{m,\bar{y}}$, with an equivariant action of
$\Gamma_1(p^n)_{V'}$, such that its base change via $h_1$ is the reduction modulo $\varpi^m$ of the base change via $h_2$ of the formal scheme $\mathfrak{IW}^+_{w}$
naturally defined over $\widehat{X_1(p^n)(v)_{m,\bar{x}}}$. Taking $\kappa'$ invariant functions on $\mathfrak{IW}^+_{w,m}$ as in definition \ref{def:formalintegralsheaf} we
introduce a quasi-coherent sheaf $\underline{\tilde{\mathfrak{w}}}^{\dag\kappa^o}_{w,m}$ over $\widehat{\mathcal{B}_1(p^n)_{V'}(v)}_{m,\bar{y}}$, with an equivariant
action of $\Gamma'_1(p^n)_{V'}$ and hence of $\Gamma_1(p^n)_{V'}$. As explained in section \ref{sect-devissage} we can write
$\underline{\tilde{\mathfrak{w}}}^{\dag\kappa^o}_{w,m}$ as an inductive limit of coherent sheaves $\colim_r \underline{\tilde{\mathfrak{w}}}^{\dag\kappa^o\leq
r}_{w,m}$. Then, each $\colim_r h_1^\ast \underline{\tilde{\mathfrak{w}}}^{\dag\kappa^o\leq r}_{w,m}$ is naturally a $\Gamma_{1}(p^n)_{V'}$-equivariant sheaf
through the diagonal action of $\Gamma_{1}(p^n)_{V'}$ on $\colim_r \underline{\tilde{\mathfrak{w}}}^{\dag\kappa^o\leq r}_{w,m}$ and on
$\oscr_{\widehat{\mathcal{M}_1(p^n)_{V',\mathcal{S}}(v)}_{m,\bar{y}}} $. Due to the definition of $\tilde{\mathfrak{w}}^{\dag\kappa^o\leq r}_{w,m}$ in
section \ref{sect-devissage} it follows that we have a $\Gamma_1(p^n)_{V'}$-equivariant isomorphism   of quasi-coherent sheaves over
$\widehat{\mathcal{M}_1(p^n)_{V',\mathcal{S}'}}(v)_{m,\bar{y}}$:

$$\colim_r  h_2^\ast\left(\tilde{\mathfrak{w}}^{\dag\kappa^o\leq r}_{w,m}\right) = \colim_r
h_1^\ast \left(\underline{\tilde{\mathfrak{w}}}^{\dag\kappa^o\leq r}_{w,m}\right).$$By the projection formula, we have that

$$\colim_r \HH^0\big( \widehat{\mathcal{M}_1(p^n)_{V',\mathcal{S}'}}(v)_{m,\bar{y}}/\Gamma_1(p^n)_{V'} ,
\tilde{\mathfrak{w}}^{\dag\kappa^o\leq r}_{w,m}(-D)\big) ~~~~~~~~~~~~~~~$$

\begin{eqnarray*}
~~~~~~~~~& = & \HH^0\Big(\Gamma_1(p^n)_{V'}, \colim_r \HH^0\big( \widehat{\mathcal{M}_1(p^n)_{V',\mathcal{S}'}}(v)_{m,\bar{y}},
h_1^\ast \underline{\tilde{\mathfrak{w}}}^{\dag\kappa^o\leq r}_{w,m}(-D)\big)\Big) \\
&= & \Big( \colim_r \prod_{\lambda \in S_V' \cap C(V/{V'}^\bot)^\vee, \lambda >0} \HH^0\big(\widehat{\mathcal{B}_1(p^n)_{V'}(v)_{m,\bar{y}}},
\mathcal{L}(\lambda)\otimes \underline{\tilde{\mathfrak{w}}}^{\dag\kappa^o\leq r}_{w,m} \big)\Big)^{\Gamma_1(p^n)_{V'}}.
\end{eqnarray*}

The action of $\Gamma_1(p^n)_{V'}$ on $S_{V'}$ and on the product above factors via $\Gamma'_1(p^n)_{V'}$. Furthermore, $\Gamma'_1(p^n)_{V'}$ acts freely on  the
elements $\lambda \in S_V' \cap C(V/{V'}^\bot)^\vee$ which are definite positive (indeed, the stabilizer of an element would be a compact group, hence finite, but
$\Gamma'_1(p^n)_{V'}$ has no finite subgroups because  of the principal level $N$ structure). Let $S_0$ be a set of representative of the orbits. We then have

$$\Big( \colim_r \prod_{\lambda \in S_V' \cap C(V/{V'}^\bot)^\vee, \lambda >0} \HH^0\big(\widehat{\mathcal{B}_1(p^n)_{V'}(v)_{m,\bar{y}}}, \mathcal{L}(\lambda)\otimes
\underline{\tilde{\mathfrak{w}}}^{\dag\kappa^o\leq r}_{w,m} \big)\Big)^{\Gamma_1(p^n)_{V'}}~~~~~~$$$$~~~~~~ =  \colim_r \prod_{\lambda \in S_0}
\HH^0\big(\widehat{\mathcal{B}_1(p^n)_{V'}(v)_{m,\bar{y}}}, \mathcal{L}(\lambda)\otimes \underline{\tilde{\mathfrak{w}}}^{\dag\kappa^o\leq r}_{w,m} \big).$$

So it remains to see that the formation of $\HH^0\big(\widehat{\mathcal{B}_1(p^n)_{V'}(v)_{m,\bar{y}}}, \mathcal{L}(\lambda)\otimes
\underline{\tilde{\mathfrak{w}}}^{\dag\kappa^o\leq r}_{w,m} \big)$
 commutes with reduction modulo $\varpi^l$ for $l\leq m$.  We have an exact sequence of sheaves over  $\widehat{\mathcal{B}_1(p^n)_{V'}(v)_{m,\bar{y}}}$:

$$ 0 \rightarrow \mathcal{L}(\lambda)\otimes \underline{\tilde{\mathfrak{w}}}^{\dag\kappa^o\leq r}_{w,m-l} \stackrel{\varpi^{l}}\rightarrow
\mathcal{L}(\lambda)\otimes \underline{\tilde{\mathfrak{w}}}^{\dag\kappa^o\leq r}_{w,m} \rightarrow \mathcal{L}(\lambda)\otimes
\underline{\tilde{\mathfrak{w}}}^{\dag\kappa^o\leq r}_{w,l} \rightarrow 0$$

By induction, we may assume that $l= m-1$. It is enough to  show that  $$\HH^1\big(\widehat{\mathcal{B}_1(p^n)_{V'}(v)_{1,\bar{y}}}, \mathcal{L}(\lambda)\otimes
\underline{\tilde{\mathfrak{w}}}^{\dag\kappa^o\leq r}_{w,1} \big) \cong 0$$

Note that $\mathcal{L}(\lambda)$,  for $\lambda \in S_0$, is a very ample sheaf on the abelian scheme $\mathfrak{B}_1(p^n)_{V'}(v)$, due to the principal level
$N$-structure with $N\geq 3$; see the proof of \cite[Thm.~V.5.8]{FC}. The vanishing of the cohomology  follows  by  the vanishing theorem of \cite[\S III.16]{Mum},
the theorem of formal functions \cite[\S4]{EGAIII} and the fact that $\underline{\tilde{\mathfrak{w}}}^{\dag\kappa^o\leq r}_{w,1}$ is an iterated extension of the trivial sheaf  as
seen by corollary \ref{coro-dev} and its proof.

\

{\bf Second Proof:}\enspace  In order to prove the claim, it suffices to consider the case $l=m-1$. From the local description of the sheaves
$\tilde{\mathfrak{w}}^{\dag\kappa^o \leq r}_{w,m}(-D)$, see section \ref{sect-devissage}, it follows that the kernel of $\colim_r \tilde{\mathfrak{w}}^{\dag\kappa^o \leq
r}_{w,m} \to \colim_r \tilde{\mathfrak{w}}^{\dag\kappa^o\leq r}_{w,m-1}$ is isomorphic to $\tilde{\mathfrak{w}}^{\dag\kappa^o}_{w,1}$. The latter is  an inductive
limit of coherent sheaves which are extensions of the trivial sheaf $\oscr_{X_1(p^n)(v)_1}$ by corollary \ref{coro-dev}. To prove the claim it then suffices to show that
$R^1\eta_\ast \oscr_{X_1(p^n)(v)_1} (-D)=0$ and this follows from proposition \ref{prop:vanishinO(-D)} below.

The proof of the second part of the proposition goes exactly along the same lines since the family of sheaves $\tilde{\mathfrak{w}}^{\dag
\kappa^{o\un}}_{w,1}$ is
trivial over the weight space by corollary \ref{cor-triv}.

\end{demo}

Recall that  $\cX_1(p^n)(v)$ is defined by the choice of a smooth and projective admissible polyhedral decomposition in the sense of
\cite[Def.~V.5.1]{FC}. In particular for every $V'$ as above there exists a $\Gamma_1(p^n)_{V'}$-admissible {\em polarization function} $h\colon  C_{V'} \to \R$, i.e. a function satisfying:

(i) $h(x)>0$ if $x\neq 0$ and $h(tx)=t h(x)$ for all $t\in \R_{\geq 0}$ and every $x\in C_{V'} $;

(ii) $h$ is upper convex namely $h\bigl(t x +(1-t) y \bigr) \geq t h(x) + (1-t) h(y)$ for every $x$ and $y\in C_{V'}$ and every $0\leq t \leq 1$;

(iii) $h$ is $\mathcal{S}$-linear, i.e., $h$ is linear on each $\Sigma \in \mathcal{S}$;

(iv) $h$ is strictly upper convex for $\mathcal{S}$, i.e, $\mathcal{S}$  is the coarsest among the fans $\mathcal{S}'$ of $C_{V'}$ for which $h$ is
$\mathcal{S}'$-linear. Equivalently, the closure of the top dimensional cones of $\mathcal{S}$ are exactly the maximal polyhedral cones of $C_{V'}$ on which $h$ is
linear.

(v) $h$ is $\Z$-valued on the set of $N$ times the subset of $C_{V'}$ consisting of symmetric semi-definite bilinear and integral valued forms on
$V/{V'}^\bot$.\smallskip

Consider the morphism $\eta\colon \cX_1(p^n)(v) \to \cX^\star(v)$ from a toroidal to the minimal compactification. Then,

\begin{prop2}\label{prop:vanishinO(-D)}  We have $R^q \eta_{\ast} \oscr_{\cX_1(p^n)(v)}(-D)=0$ for every $q\geq 1$.

\end{prop2}

\begin{demo}
We use the notations of the proof of theorem \ref{thm:baschange}. We write $\widehat{Z}_{V'}:=\widehat{\mathcal{M}_1(p^n)_{V',
\mathcal{S}'}}(v)_{1,\bar{y}}$ and
$\widetilde{{Z}}_{V'}:=\widehat{Z}_{V'} /\Gamma_1(p^n)_{V'}$ to simplify the notation. By the theorem of formal functions
\cite[\S4]{EGAIII} it suffices to prove that
$\HH^q\big(\widetilde{{Z}}_{V'}, \oscr_{\widehat{Z}_{V'}}(-D)\big)=0$ for every $q\geq 1$.

We recall the construction of $\widehat{Z}_{V'}$. We have fixed a smooth $\Gamma_1(p^n)_{V'}$-admissible polyhedral decomposition
$\mathcal{S}$ of the cone
$C_{V'}:=C(V/{V'}^\bot)$ of symmetric semi-definite bilinear forms on $V/{V'}^\bot \otimes \R$ with rational radical. Every $\Sigma \in \mathcal{S}$ defines an
affine relative torus embedding $Z_{\Sigma}$ over the abelian scheme ${\mathcal{B}_1(p^n)_{V'}(v)_{1,\bar{y}}}$ which we view over the spectrum of the local ring
underlying $\widehat{Y_1(p^n)_{V'}(v)_{1,\bar{y}}}$. The $Z_{\Sigma}$'s glue to define a relative torus embedding $Z_{V'}$ stable for the action of
$\Gamma_1(p^n)_{V'}$. For every $\Sigma$ we let $W_\Sigma:= \sum_{\rho\in \Sigma(1)^o} D_\rho$ be the relative Cartier divisor defined by the set $\Sigma(1)^o$ of $1$-dimensional faces
of $\Sigma$ contained in the interior $C_{V'}^o$ of the cone $C_{V'}$. Put $W_{V'}:=\cup_\Sigma W_\Sigma$. Write $\widehat{Z}_{\Sigma}$ (resp.~$\widehat{Z}_{V'}$)
for the formal scheme given by  the completion of $Z_{\Sigma}$ (resp.~$Z_{V'}$) with respect to the ideal $\oscr_{Z_{\Sigma}}\bigl(-W_{\Sigma}\bigr)$
(resp.~$\oscr_{Z_{V'}}\bigl(-W_{V'}\bigr)$). Then, $\widehat{Z}_{V'}=\cup_{\Sigma} \widehat{Z}_{\Sigma}$.

Fix a $\Gamma_1(p^n)_{V'}$-admissible polarization function $h\colon C_{V'}\to \R_{\geq 0}$. As in \cite[Def.~V.5.6]{FC} we define
$$D_{\Sigma,h}':=-\sum_{\rho\in \Sigma(1)} a_\rho
D_\rho,\qquad D_h':=\cup_\Sigma D_{\Sigma,h},$$where the sum is over the set $\Sigma(1)$ of all $1$-dimensional faces  of $\Sigma$ (not simply over the set $\Sigma(1)^o$ as in the definition of $W_\Sigma$) . More explicitly, for every $\Sigma\in
\mathcal{S}$ and every $\rho\in \Sigma(1)$ there exists a unique primitive integral element $n(\rho)\in \Sigma$ such that $\rho=\R_{\geq 0} n_\rho$. We then set
$a_\rho:=h\big(n(\rho)\big)$. As $h(x)\neq 0$ if $x\neq 0$ and $h(x)\in \Z$ for integral elements $x\in C_{V'}$, we deduce that $a_\rho$ is a positive integer for
every $\Sigma$ and every $\rho\in \Sigma(1)$. Moreover, the Cartier divisor $D_h'$  of $Z_{V'}$ is $\Gamma_1(p^n)_{V'}$-invariant.

Take a positive integer $s$. As the set of integers $\{a_\rho\}$ is $\Gamma_1(p^n)_{V'}$-invariant, it is finite.  Thus there exists $\ell\in \Z$ such that $0< s
a_\rho < \ell$ for every $\rho$. Recall that given a $\qq$-divisor $E:=\sum_\rho e_\rho D_\rho$, with $e_\rho\in\qq$,  one defines the ``round down" Cartier divisor
$\lfloor E \rfloor:=\sum_\rho \lfloor e_\rho \rfloor D_\rho$ by setting $\lfloor e_\rho \rfloor$ to be the smallest integer $\leq e_\rho$. In particular, we compute
$$\lfloor \ell^{-1} s D'_h\rfloor=\sum_{\rho\in \Sigma(1)}-D_\rho:=-D,$$where $D$ defines the boundary of $Z_\Sigma$. Multiplication by $\ell$ on the cone $C_{V'}$
preserves the polyhedral decomposition $\mathcal{S}$ and for every $\Sigma\in \mathcal{S}$ defines a finite and flat morphism $\Phi_{\ell,\Sigma}\colon Z_\Sigma \to
Z_\Sigma$ over ${\mathcal{B}_1(p^n)_{V'}(v)_{1,\bar{y}}}$. As $\Phi_{\ell,\Sigma}^\ast(W_\Sigma)= \ell W_\Sigma$, the morphism $\Phi_{\ell,\Sigma}$ induces a
morphism on the completions with respect to the ideal $\mathcal{O}_{Z_\Sigma}(-W_\Sigma)$ and we get finite and flat morphisms $\widehat{\Phi}_{\ell,\Sigma}\colon
\widehat{Z}_\Sigma \longrightarrow \widehat{Z}_\Sigma$. They glue to provide a finite flat, $\Gamma_1(p^n)_{V'}$-equivariant morphism of formal schemes
$\widehat{\Phi}_\ell\colon \widehat{Z}_{V'} \rightarrow \widehat{Z}_{V'}$ over $\widehat{\mathcal{B}_1(p^n)_{V'}(v)_{1,\bar{y}}}$. After passing to the quotients by
$\Gamma_1(p^n)_{V'}$ we get a finite and flat morphism of formal schemes $$\widetilde{\Phi}_\ell\colon \widetilde{Z}_{V'}\to \widetilde{Z}_{V'}.$$As $s D_h'-
\Phi_\ell^\ast(-D)=\sum_\rho (\ell-s a_\rho) D_\rho$ is an effective Cartier divisor, we have by adjunction natural inclusions of invertible sheaves $$\iota_\ell
\colon \oscr_{Z_{V'}}(-D) \to {\Phi}_{\ell,\ast}\big(\oscr_{{Z}_{V'}}(s D_h')\big),\qquad \widehat{\iota}_\ell \colon \oscr_{\widehat{Z}_{V'}}(-D) \to
{\widehat{\Phi}}_{\ell,\ast}\big(\oscr_{\widehat{Z}_{V'}}(s D_h')\big).$$In \cite[Lemmas 9.2.6 and 9.3.4]{CS} a canonical splitting of $\iota_\ell$ as
$\oscr_{{Z}_{V'}}$-modules is constructed in terms of the cone  $C_{V'}$ and the integers $\{a_\rho\vert \rho\in \Sigma(1)\}$. In particular it is
$\Gamma_1(p^n)_{V'}$-equivariant and it defines a $\Gamma_1(p^n)_{V'}$-equivariant splitting of $\widehat{\iota}_\ell$ as $\oscr_{\widehat{Z}_{V'}}$-modules after
passing to completions. Taking the quotient under $\Gamma_1(p^n)_{V'}$ we get a split injective map $\widetilde{\iota}_{\ell} \colon
\oscr_{\widetilde{{Z}}_{V'}}(-D) \to \widetilde{{\Phi}}_{\ell,\ast} \oscr_{\widetilde{{Z}}_{V'}}(s D_h')$ of $\oscr_{\widetilde{{Z}}_{V'}}$-modules. Taking
cohomology for every $q\in \N$ we get a split injective map

$$\HH^q\bigl(\widetilde{{Z}}_{V'},\oscr_{\widetilde{{Z}}_{V'}}(-D)\big)
\to  \HH^q\bigl(\widetilde{{Z}}_{V'}, \widetilde{{\Phi}}_{\ell,\ast} \big(\oscr_{\widetilde{Z}_{V'}}(s D_h') \bigr)\cong
\HH^q\bigl(\widetilde{{Z}}_{V'},\oscr_{\widetilde{Z}_{V'}}(s D_h') \bigr).$$If we show that there exists $s\in\N$ such that
$\HH^q\bigl(\widetilde{{Z}}_{V'},\oscr_{\widetilde{Z}_{V'}}(s D_h') \bigr)=0$ for every $q\geq 1$, we are done. This follows if we prove that there exists $s$ such
that $\oscr_{\widetilde{Z}_{V'}}(s D_h') $ is a very ample invertible sheaf.

It follows from \cite[Thm.~V.5.8]{FC} that the map $\eta\colon \cX_1(p^n)(v) \to \cX^\star(v)$ is the normalization of the blow-up of $\cX^\star(v)$ defined by a sheaf of
ideals $\mathcal{J}$ such that $\eta^\ast(\mathcal{J})$ restricted to $\widetilde{{Z}}_{V'}$ is $\oscr_{\widetilde{Z}_{V'}}(d D_h')$ for a suitable $d$. In
particular, as $\eta$ is the composite of a finite map and a blow-up, the sheaf $\eta^\ast(\mathcal{J})$ is ample relatively to $\eta$. We conclude that
$\oscr_{\widetilde{Z}_{V'}}(d D_h')$ is an ample sheaf on  $\widetilde{Z}_{V'}$. In particular, there exists a large enough multiple $s$  of $d$ so that
$\oscr_{\widetilde{Z}_{V'}}(s D_h')$ is very ample as claimed.

Alternatively, to prove that $\oscr_{\widetilde{Z}_{V'}}(D_h')$ is ample, it suffices to prove that its restriction to the boundary $\partial  \widetilde{Z}_{V'}$
is ample. As $\partial  \widetilde{Z}_{V'}=W_{V'}/\Gamma_1(p^n)_{V'}$ is proper over $\widehat{Y_1(p^n)}_{V'}(v)_{m,\bar{y}}$ it suffices to prove ampleness after
passing to the residue field $k(\bar{y})$ of $\widehat{Y_1(p^n)}_{V'}(v)_{m,\bar{y}}$. It then follows from the Nakai-Moishezon criterion for ampleness \cite{Kl} that it
suffices to prove that the restriction of $\oscr_{\widehat{Z}_{V'}}(D_h')$ to the fiber of the boundary $W_{V'}$ of $\widehat{Z}_{V'}$ over $\bar{y}$ is ample in
the sense that the global sections of $\oscr_{\widehat{Z}_{V'}\otimes k(\bar{y})}(dD_h')$ for $d\geq 1$ form a basis of the topology of $W_{V'}\otimes k(\bar{y})$
(see the footnote to \cite[Def.~2.1, Appendix]{FC}). This follows if we prove the stronger statement that $\oscr_{Z_{V'}}(dD_h')$ is very ample for every $d\geq 1$,
i.e., that the elements of $\HH^0\big(Z_{V'},\oscr_{Z_{V'}}(dD_h')\big)$ form a basis of the Zariski topology of $Z_{V'}$. If $f\colon Z_{V'}\to
\mathcal{B}_1(p^n)_{V'}(v)_{1,\bar{y}}$ is the structural morphism, then $f_\ast\big(\oscr_{Z_{V'}}(d D_h')\big)= \oplus_{\lambda \geq d h} \mathcal{L}(\lambda)$,
where the sum is taken over all integral elements $\lambda\in S_{V'} \cap C(V/{V'}^\bot)^\vee$ such that for every $\Sigma\in \mathcal{S}$ and every $\rho\in
\Sigma$ we have $\lambda(\rho) \geq d h(\rho)$, see \cite[\S IV.4]{Dem}. In particular,
$$\HH^0\big(Z_{V'},\oscr_{Z_{V'}}(dD_h')\big)= \oplus_{\lambda \geq d  h} \HH^0\big({\mathcal{B}_1(p^n)_{V'}(v)_{1,\bar{y}}},
\mathcal{L}(\lambda)\big).$$Assuming conditions (i), (iii) and (v) in the definition of polarization function given above then condition
\cite[Cor.~IV.4.1(iii)]{Dem} is equivalent to conditions (ii) and (iv) above. As $dh$ is also a polarization function, we conclude from \cite[Cor.~IV.4.1(i) and Pf.~Thm.~IV.4.2]{Dem} that the morphism $Z_{V'} \to \mathbf{Proj} \left(\oplus_s f_\ast\big(\oscr_{Z_{V'}}( d D_h')\big)^{\otimes^s}\right) $ of schemes  over
${\mathcal{B}_1(p^n)_{V'}(v)_{1,\bar{y}}}$ defined by $f_\ast\big(\oscr_{Z_{V'}}( dD_h')\big)$ is a closed immersion. The sheaf $\mathcal{L}(\lambda)$ is very
ample on the abelian scheme ${\mathcal{B}_1(p^n)_{V'}(v)_{1,\bar{y}}}$ for every integral, non-zero element $\lambda\in S_{V'} \cap C(V/{V'}^\bot)^\vee$ due to the
principal level $N$-structure with $N\geq 3$; see the proof of \cite[Thm.~V.5.8]{FC}. As the condition $\lambda \geq d h$ implies $\lambda> 0$, the very ampleness
of $\oscr_{Z_{V'}}(d D_h')$ follows.

\end{demo}

\subsubsection{Applications of the base change theorem: the proof of proposition \ref{prop:mainspecializationresult}}
We let $\mathfrak{U} = ( \mathfrak{V}_i)_{1\leq i \leq r}$ be an affine covering of $\mathfrak{X}^\star(v)$. We let $\underline{i}= (i_1,i_2,\ldots,i_{n'})$ be a
multi index with $1\leq i_1 < \ldots< i_{n'} \leq r$. We let $\mathfrak{V}_{\underline{i}}$ be the intersection of $\mathfrak{V}_{i_1}$, $\mathfrak{V}_{i_2}$,
$\ldots$, $\mathfrak{V}_{i_{n'}}$. This is again an affine formal scheme. We denote by $V_{\underline{i},m}$ the scheme obtained by reduction modulo $\varpi^m$.  We
let $M_{\underline{i},m} =  \HH^0(V_{\underline{i},m}\times \mathcal{W}(w)^o_m, (\eta\times 1)_\ast \tilde{\mathfrak{w}}^{\dag\kappa^{o\un}}_{w,m}(-D))$ and
$M_{\underline{i},\infty} =  \HH^0(\mathfrak{V}_{\underline{i}}\times \mathfrak{W}(w)^o, (\eta\times 1)_\ast \tilde{\mathfrak{w}}^{\dag\kappa^{o\un}}_{w}(-D)) =
\lim_m M_{\underline{i},m}$. Finally let $A $ be the algebra of  $\mathfrak{W}(w)^o$.

\begin{coro2} The module $M_{\underline{i},\infty}$ is isomorphic to the $p$-adic completion of a free  $A$-module.
\end{coro2}
\begin{demo} The module $M_{\underline{i},\infty}$ is $p$-torsion free and the reduction map $M_{\underline{i},\infty}
\rightarrow M_{\underline{i},1}$ is surjective. The $A/\varpi A$-module $M_{\underline{i},1}$ is free by corollary \ref{cor-triv}.
Fix a basis $(\bar{e}_i)_{i\in I}$ for this module. We lift the vectors $\bar{e}_i$ to vectors $e_i$ in
$M_{\underline{i},\infty}$. We let $\widehat{A^I}$ be the $p$-adic completion of the module $A^I$. Now consider the map
$\widehat{A^I} \rightarrow M_{\underline{i},\infty}$ which sends $(a_i)_{i\in I}\in \widehat{A^I}$ to $\sum_i a_i e_i$.
We claim that this map is an isomorphism. It is surjective by the topological Nakayama lemma. It is injective, for if
$\sum_i a_i e_i$ is  $0$ and $(a_i)_{i\in I} \neq 0$, there is $n\in \N$, $(a_i' )_{i\in I}\in \widehat{A^I}$  such
that $\varpi^n a'_i=a_i$  and  an index $i_0$ such that $a_{i_0}' \notin \varpi A$. Since
$ M_{\underline{i},\infty} $ is $\varpi$-torsion free we have $\sum_i a'_i e_i=0$ and reducing this relation modulo
$\varpi$ we get a contradiction.
\end{demo}
\medskip

We set $M = \HH^0\left(\cX_1(p^n)(v) \times \mathcal{W}(w)^o, \tilde{\mathfrak{w}}^{\dag\kappa^{o\un}}_{w}(-D)\right)[p^{-1}]$, and $M_{\underline{i}} =
M_{\underline{i},\infty}[p^{-1}]$.

\begin{coro2}\label{cor:Mproj} The module $M$ is a projective Banach-$A[\frac{1}{p}]$-module. For any $\kappa \in\mathcal{W}(w)^o$,
the specialization map $M \rightarrow  \HH^0\bigl(\cX_1(p^n)(v), \tilde{\mathfrak{w}}^{\dag\kappa^o}_{w}(-D)\bigr)[p^{-1}]$ is surjective.
\end{coro2}

\begin{demo} Since $\cX^\star(v)\times \mathcal{W}^o$ is affinoid and $(\eta\times 1)_\ast \tilde{\mathfrak{w}}^{\dag\kappa^{o\un}}_{w}(-D)$ is a formal
Banach sheaf, we can apply  the acyclicity theorem to  the Chech complex associated to the rigid analytic fiber $\mathcal{U} = (\mathcal{V}_{\underline{i}})$ of the covering $\mathfrak{U}$. We thus get a resolution of
the module $M$ by the projective modules $M_{\underline{i}}$ and as a result $M$ is projective.

We now prove the surjectivity of the specialization map.   Let $P_{\kappa^o}$ be the maximal ideal of $\kappa^o$ in $A[p^{-1}]$. We consider the   Koszul resolution
of $A[p^{-1}]/P_{\kappa^o}$:
$$\mathrm{Ko}(\kappa^o)\colon   0 \rightarrow  A[p^{-1}] \rightarrow A[p^{-1}]^g  \cdots \rightarrow A[p^{-1}]^g\rightarrow A[p^{-1}] \rightarrow
A[p^{-1}]/P_{\kappa^o} \rightarrow 0.$$ The tensor product $\mathrm{Ko}(\kappa^o) \otimes (\eta \times 1)_\ast \tilde{\mathfrak{w}}^{\dag \kappa^{o\un}}_{w}(-D)$ is a
resolution of $\eta_\ast \tilde{\mathfrak{w}}^{\dag\kappa^o}_w(-D)[p^{-1}]$ by  sheaves isomorphic to direct sums of the sheaves $(\eta \times 1)_\ast
\tilde{\mathfrak{w}}^{\dag \kappa^{o\un}}_{w}(-D)[p^{-1}]$.

We consider the following double complex, obtained by taking the Cech complex of  $\mathrm{Ko}(\kappa) \otimes (\eta \times 1)_\ast \tilde{\mathfrak{w}}^{\dag
\kappa^{o\un}}_{w}(-D)$ attached to the covering $\mathcal{U} = (\mathcal{V}_{\underline{i}})$   (we think of $\mathrm{Ko}(\kappa) \otimes (\eta \times 1)_\ast \tilde{\mathfrak{w}}^{\dag \kappa^{o\un}}_{w}(-D)$
as a vertical complex of sheaves):

$$0 \rightarrow  \HH^0\left(\cX^\star(v)\times \mathcal{W}(w)^o, \mathrm{Ko}(\kappa^o) \otimes (\eta \times 1)_\ast\tilde{\mathfrak{w}}^{\dag \kappa^{o\un}}_{w}(-D)\right)
 \rightarrow $$
$$ ~~~~~~~~~~~~~~~~ \oplus_i \HH^0\left(\mathcal{V}_{\underline{i}}\times \mathcal{W}(v)^o, \mathrm{Ko}(\kappa^o)
\otimes (\eta \times 1)_\ast\tilde{\mathfrak{w}}^{\dag \kappa^{o\un}}_{w}(-D)\right) \rightarrow \cdots .$$For any multi-index $\underline{i}$ the complex
$\HH^0(\mathcal{V}_{\underline{i}} \times \mathcal{W}(w)^o , \mathrm{Ko}(\kappa^o) \otimes (\eta \times 1)_\ast \tilde{\mathfrak{w}}^{\dag \kappa^{o\un}}_{w}(-D))$ is exact. All the rows
of the double complex are exact by the acyclicity theorem. It follows that the first column is also exact.
\end{demo}

\bigskip

We now prove  proposition \ref{prop:mainspecializationresult}. Let $B$ be the algebra of rigid analytic functions on $\mathcal{W}(w)$.

\begin{prop2}\label{prop:mainspecializationresult}
a) The module $\HH^0(\cXI(v)\times \mathcal{W}(w), \omega^{\dag\kappa^\un}_w(-D))$ is a projective Banach-$B$-module.

and

b) For every $\kappa \in \mathcal{W}(w)$ the specialization map

$$ \HH^0(\cXI(v)\times \mathcal{W}(w), \omega^{\dag\kappa^\un}_w(-D)) \rightarrow \HH^0(\cXI(v), \omega^{\dag\kappa}_w(-D))$$is surjective.
\end{prop2}

\begin{demo}  We use the notations of the corollary \ref{cor:Mproj}. By definition,

$$\HH^0(\cXI(v)\times \mathcal{W}(w), \omega^{\dag\kappa^\un}_w(-D)) = \big(M \otimes_{A[\frac{1}{p}]} B(-\kappa^{\un'})\big)^{\B(\Z/p^n\Z)}$$

Here, $B(-\kappa^{\un'})$ is the free $B$-module with action of $B(\Z/p^n\Z)\mathcal{B}_w$ via the character $-\kappa^{\un'}$. Then, $M \otimes_{A} B(-\kappa^{\un'})$
is viewed as a $B(\Z/p^n\Z)\mathcal{B}_w$-module with diagonal action. This action factors through the group $\B(\Z/p^n\Z)$ and the invariants are precisely
$\HH^0(\cXI(v)\times \mathcal{W}(w), \omega^{\dag\kappa^{\un}}_w(-D))$. Thus, this module is a direct factor in a projective $B$-module by corollary \ref{cor:Mproj} so it
is projective. Now, let $\kappa \in \mathcal{W}(w)$. We let $\kappa^o$ be its image in $\mathcal{W}(w)^o$. Let $m_\kappa$ be the maximal ideal of $\kappa$ in $B$. Set
$M_{\kappa^o} = \HH^0\bigl(\cX_1(p^n)(v), \tilde{\mathfrak{w}}^{\dag\kappa^o}_{w}(-D)\bigr)[p^{-1}]$.  The specialization map $M \rightarrow M_{\kappa^o}$ is surjective
thanks to corollary \ref{cor:Mproj}. The map $M \otimes_A B(-\kappa^{\un}) \rightarrow M_{\kappa^o} \otimes_A B/m_\kappa(-\kappa)$ is  surjective and  the map
$$\big(M \otimes_A B(-\kappa^{\un'})\big)^{\B(\Z/p^n\Z)} \rightarrow \big(M_{\kappa^o} \otimes_A B/m_\kappa(-\kappa')\big)^{\B(\Z/p^n\Z)}$$
is still surjective since $\B(\Z/p^n\Z)$ has no higher cohomology on $\qq_p$-modules. As the $\HH^0(\cXI(v), \omega^{\dag\kappa}_w(-D))=\big(M_{\kappa^o}
\otimes_A B/m_\kappa(-\kappa')\big)^{\B(\Z/p^n\Z)}$, the claim concerning the specialization follows.
\end{demo}

\appendix

\section{Banach sheaves}

\subsection{Definitions and Acyclicity}
Let $K$ be a complete field for a non archimedean valuation $\vert. \vert$, $A$  a $K$-affinoid algebra
equipped with a quotient norm $\vert. \vert$ and let $M$ be an $A$-module. We say that $M$ is a
normed $A$ module if
there is a  norm function $\vert. \vert\colon  M \rightarrow \R_{\geq 0}$ such that:
\begin{enumerate}
\item $\vert m \vert = 0  \Leftrightarrow m=0$,
\item $\vert m + n \vert  \leq \sup \{ \vert m \vert, \vert n\vert\}$,
\item $\vert a m \vert \leq \vert a\vert \vert m \vert$, for $a \in A$ and $m\in M$.
\end{enumerate}
If $\vert . \vert$ only satisfies $2.$ and $3.$, we call it a semi-norm. We say that $M$ is a Banach $A$-module if $M$
is a complete normed $A$-module. The open mapping theorem will be  frequently used:

\begin{thm}[\cite{Bou}, Chap. I, sect. 3.3, thm. 1]A surjective continuous map $\phi\colon  M\rightarrow N$  between Banach $A$-modules is
open.
\end{thm}

In the situation of the theorem, the topology on $N$ is  the image of the topology on $M$, and the norm on $N$ is
equivalent to the image norm defined by: $\vert n \vert = \inf_{m \in \phi^{-1}(n)} \vert m\vert$.

If $(M, \vert. \vert)$ is a Banach $A$-module then any other norm $\vert. \vert '$ on $M$ inducing the same topology on $M$ is  equivalent to  $\vert.\vert$.

The following lemma will also be useful:
\begin{lem}\label{lem-completion} Let $$ 0 \rightarrow N \rightarrow M \rightarrow L \rightarrow 0$$ be an exact sequence of
$A$-modules. We assume that $M$ is equipped with a semi-norm $\vert. \vert $,  and that $N$ and $L$ are equipped with the induced semi-norms respectively.
Then there is a natural exact sequence of the
separation-completions:
$$0 \rightarrow \hat{N} \rightarrow \hat{M} \rightarrow \hat{L} \rightarrow 0$$
\end{lem}

\begin{demo} The map $M \rightarrow \hat{L}$ is bounded so it extends uniquely to a map
$p\colon  \hat{M} \rightarrow \hat{L}$. Let $l \in \hat{L}$. We can write $l = \sum_n l_n$ with $l_n \in L$ and $\vert l_n\vert \rightarrow 0 $
when $n \rightarrow \infty$. Take $m_n \in M$ such that $p(m_n)=l_n$ and $\vert m_n\vert \leq \vert l_n \vert + \frac{1}{n}$. Then
$m= \sum_n m_n \in \hat{M}$ maps to $l$. The completion of $N$ is clearly isomorphic to   the closure in $\hat{M}$ of the image of $N$. Moreover the
composition
$\hat{N} \rightarrow \hat{M} \rightarrow \hat{L}$ is clearly $0$. Let us take $x \in \hat{M}$ such that $p(m)=0$. We have $x = \lim_{n \rightarrow \infty} x_n$ where $x_n \in M$, and $\lim_{n \rightarrow \infty} p(x_n) =0$. There is $y_n \in M$ such that $\vert y_n \vert \leq \vert p(x_n) \vert + \frac{1}{n}$ and $p(y_n) = p(x_n)$. So we have $x_n = x_n -y_n + y_n$ where $x_n -y_n \in \hat{N}$. Since $\lim_{n \rightarrow \infty} y_n = 0$, we deduce that $x \in \hat{N}$.
 \end{demo}

\bigskip
%{}

Let $\cU$ be an admissible affinoid open  of $\cX = {\rm Spm} ~A$, and let $A_\cU$ be the ring of rigid
analytic functions on $\cU$,
equipped with a residue norm. Let $(M,\vert\_ \vert)$ be a Banach $A$-module. We define the complete
localization $M_\cU = M\hat{\otimes}_A A_\cU$. It is the separation and completion of the semi-normed $A_\cU$-module $M\otimes_A A_\cU$ where the semi-norm of an element $x$ is the infimum over all the expressions  $x= \sum_i m_i \otimes a_i$ of the supremum $\sup_i \vert a_i\vert \vert m_i \vert $.

Let $M$ be a Banach $A$-module. We  define the pre-sheaf $\mathcal{M}$ on the category of affinoid open subsets of $\cX$ by $\mathcal{M} (\cU) = M\hat{\otimes}_A A_\cU$.

We now have the proposition:

\begin{prop}\label{prop_acy} Let $\mathfrak{U}= \{\cU_j\}_{j\in J}$ be a finite admissible affinoid covering  of $\cX$. The augmented Cech complex:

$$0 \rightarrow \mathcal{M}(\cX) \rightarrow \oplus_j \mathcal{M}(\cU_j) \rightarrow \oplus_{k,j} \mathcal{M}(\cU_{k,j}) \rightarrow \ldots$$
 is exact.
 \end{prop}
\begin{demo} Let $I$ be a set. We denote by $C(I)$ the $A$-module of functions $f\colon  I \rightarrow A$ with the property that $\lim_{i \rightarrow \infty } f(i) =0$, where the convergence is with respect to the filter
of complements of finite sub-sets of $I$. The module $C(I)$ is endowed with the norm
$\vert f \vert = \sup_i \vert f(i)\vert $ for which it is a Banach module. We remark that $C(I)$ is the
completion of the free  $A$-module with basis indexed by $I$. An othonormalizable Banach module is a
module isomorphic (with its  norm) to $C(I)$ for some $I$.   We first establish the proposition for
orthonormalizable Banach modules. By Tate's acyclicity theorem \cite[Thm.~8.2]{TateRAS}, the sequence
$$ 0 \rightarrow \oscr_\cX(\cX) \rightarrow \oplus_j \oscr_\cX(\cU_j) \rightarrow \ldots$$ is exact. If we tensor
by the free module with basis indexed by $I$, we still get an exact sequence, and by lemma
\ref{lem-completion}, the sequence remains exact after taking the separation-completion.

Let $q \geq 1$. Assume that  $\HH^t (\mathfrak{U}, \mathcal{N})=0$ for all $t > q$ and for all Banach modules
$N$. We show that the vanishing holds for $q$. Let $M$ be a Banach module. We can find a surjection
$p\colon  C(I) \rightarrow M$ obtained by taking a set of bounded topological generators of $M$.
We can also replace the norm on $M$ by the image norm because of the open mapping theorem. Let $K$ be the
kernel of the map, it is a closed Banach sub-module of $C(I)$. We now claim that for all $\cU$ affinoid
open, the sequence

$$0 \rightarrow K \hat{\otimes}_A A_\cU  \rightarrow C(I) \hat{\otimes}_A A_\cU \rightarrow  M \hat{\otimes}_{A} A_\cU \rightarrow 0$$ is exact.
The sequence of localizations (without completion) is exact because $A_\cU$ is flat over $A$ and we can
apply  lemma \ref{lem-completion}.
As a result we get  a long exact  sequence in Cech cohomology:
$$ \cdots \rightarrow \HH^q(\mathfrak{U}, \mathcal{K}) \rightarrow \HH^q( \mathfrak{U}, \mathcal{C}(I)) \rightarrow \HH^q ( \mathfrak{U}, \mathcal{M}) \rightarrow 0 \cdots$$
and we have $\HH^q( \mathfrak{U}, \mathcal{C}(I))=0$ so we get $\HH^q ( \mathfrak{U}, \mathcal{M})=0$.

\end{demo}

We are led to the following definition:

\begin{defi} Let $\cX$ be a rigid space and $\mathscr{F}$ be a sheaf on $\cX$. We say that
$\mathscr{F}$ is a Banach sheaf if:
\begin{enumerate}
\item for all affinoid open sub-set $\cU$ of $\cX$, $\mathscr{F}(\cU)$ is a Banach $\oscr_\cX(\cU)$-module,
\item the restriction maps are continuous,
\item there exists an  admissible affinoid covering $\mathfrak{U} = \{\cU_i\}_{i\in I}$ of $\cX$ such that
for all $i \in I$ and all affinoid  $\mathcal{V} \subset \cU_i$, the map:
$$ \oscr_\cX(\mathcal{V}) \hat{\otimes}_{\oscr_\cX(\cU_i)} \mathscr{F}(\cU_i) \rightarrow \mathscr{F}(\mathcal{V})$$ is an isomorphism.
\end{enumerate}
\end{defi}

 A coherent sheaf is a Banach sheaf and thus  Banach sheaves  generalize  coherent sheaves. We also
introduce a special  sub-category of Banach sheaves. We  say that a Banach module $M$ over $A$ is
projective if it is a direct factor of an orthonormalizable Banach module. We say that a Banach sheaf
$\mathscr{F}$ over a rigid space $\cX$ is  a projective Banach sheaf if $\mathscr{F}$ is  a Banach sheaf
and if there is an admissible affinoid covering $\mathfrak{U}$ of $\cX$ such that for
each $\cU\in \mathfrak{U}$, the $\oscr_\cX(\cU)$-module $\HH^0(\cU,\mathscr{F})$ is a  projective $\oscr_\cX(\cU)$-Banach module.

\subsection{Kiehl's theorem for Banach sheaves}

We can now prove the analogue of Kiehl's theorem (see \cite{Bos}, p. 93, thm. 4) for Banach sheaves.

\begin{thm}\label{thm_Khiel} Let $\cX = \Spm~A$ be an  affinoid  over $K$ and $\mathscr{F}$  a
Banach sheaf
on $\cX$. Then $\mathscr{F}$ is the Banach sheaf associated to the Banach module $\HH^0(\cX,\mathscr{F})$.
\end{thm}

\begin{demo} We'll follow the outline of the proof given in \cite{Bos}, p. 93 to 96. Let $\mathfrak{U}$ be an admissible affinoid
covering of $\cX$ such that $\mathscr{F}\vert_{\cU_i}$ is the sheaf associated to the
Banach modules $M_i$ for all $\cU_i \in \mathfrak{U}$. We say that $\mathscr{F}$ is
$\mathfrak{U}$-quasi-coherent. By a standard reduction we can assume that $\mathfrak{U}$ is a Laurent
covering (we can always refine $\mathfrak{U}$ by a Laurent covering). By induction, we can assume that the
Laurent covering is given by a single invertible  function $f \in A$. So $\mathfrak{U} = \{\cU_+, \cU_-\}$
where $\cU_ + =\{ x \in \cX,  \vert f \vert \leq 1\} =  \Spm~A<f>$ and
$\cU_ - =\{ x \in \cX,  \vert f \vert \geq 1\} =  \Spm~A<f^{-1}>$. Also set $\cU = \cU_+\cap \cU_-$.

\begin{lem} The Cech complex
$$0 \rightarrow \HH^0(\cX, \mathcal{F}) \rightarrow \HH^0(\cU_+,\mathcal{F}) \oplus
\HH^0(\cU_-,\mathcal{F} ) \rightarrow \HH^0(\cU, \mathcal{F})\rightarrow 0$$ is exact.
\end{lem}

We set $M_{+/-} = H^0(\cU_{+/-}, \mathcal{F})$ and $M= H^0(\cU, \mathcal{F})$. We need to prove the surjectivity
of the map $M_+ \oplus M_- \rightarrow M$.

Let $\vert\_ \vert$ be the norm on $A$. We deduce from it  a norm on the rings of restricted power series
$A\langle X
\rangle$, $A\langle Y\rangle$, $A\langle X,Y\rangle $.  Using  the surjections
$A\langle X\rangle \rightarrow
A\langle f\rangle$, $A\langle Y\rangle \rightarrow A\langle f^{-1}\rangle $, $A\langle X,Y\rangle\rightarrow
A\langle f,f^{-1}\rangle$, we equip $A\langle f\rangle $, $A\langle f^{-1}\rangle $ and
$A\langle f,f^{-1}\rangle$
with the residue norms.
  By definition of the norm,  for all $\beta>1$,  any element $g \in A\langle f,f^{-1}\rangle $ can be
written in the
form $\sum_{n \in \Z} c_n f^n$ with $\sup_n \vert c_n \vert \leq \beta \vert g\vert$. As a result, we deduce
that there
exist   $g_{+/-} \in A\langle f^{+/-1}\rangle$ such that  $\vert g_{+/-}\vert \leq \beta \vert g \vert $ and $g = g_+-g_-$.
Let $\{v_i\}_{i\in I}$ and $\{w_j\}_{j\in J}$ be  bounded topological generators of $M_+$ and $M_-$ as
$A\langle f\rangle $ and $A\langle f^{-1}\rangle $ modules. We get surjections $C(I) \rightarrow M_+$ and
$C(J) \rightarrow M_-$. We denote by $\vert. \vert$ the corresponding norms on $M_+$ and
$M_-$. We let $v'_i$ and $w'_j$
be the images of $v_i$ or $w_j$ in $M$. The $\{v_i\}_{i\in I}$ and $\{w_j\}_{j\in J}$ are bounded
topological generators
of $M$. It follows from the open mapping theorem that the residue norms induced by both surjections
$C(I) \rightarrow M$ and $C(J) \rightarrow M$ are equivalent. We call them $\vert . \vert_I$ and
$\vert. \vert_J$.
There exists $m,n \in \R_{>0}$ such that $m\vert. \vert_I \leq \vert. \vert_J \leq n \vert.\vert_I$. In
the sequel we
simply put $\vert. \vert$ for $\vert.\vert_I$.

To prove the lemma , it suffices to show that for any $\epsilon >0$, there is $\alpha>1$ such that for
any $u \in M$, there is $u_{+/-} \in M_{+/-}$ with  $\vert u_{+/-} \vert \leq \alpha \vert u \vert $
and $\vert u- u_+-u_-\vert \leq
\epsilon \vert u \vert $.

 There exist elements $a_{i,j}, b_{i,j} \in A\langle f,f^{-1}\rangle$ uniformly bounded, say by
$L \in \R$, such that
\begin{eqnarray*}
 v'_i& = &\sum_j a_{i,j} w'_j \\
 w'_j &=& \sum_i b_{i,j} v'_j
\end{eqnarray*}

Since $A\langle f^{-1}\rangle $ is dense in $A\langle f,f^{-1}\rangle$, for any $L'>0$, we can find $x_{i,j} \in
A\langle f\rangle $ such that $\vert x_{i,j} - a_{i,j}\vert \leq L'$. Let $u \in M$. We can write
$u = \sum_i a_i v_i$ where $\sup_i \vert a_i \vert \leq \beta \vert u\vert$. Now write $a_i = a_i^+- a_i^{-}$
where $a_i^{+/-} \in A\langle f^{+/-1}\rangle$ satisfies $\vert a_i^{+/-}\vert\leq \beta \vert a_i \vert$.
Now we set
\begin{eqnarray*}
 u_+ &=& \sum_i a_i^{+} v_i\\
 u_- &=& \sum_{i,j} a_i^{-} x_{i,j} w_j\\
 \end{eqnarray*}

 We compute that $$u-u_++u_- = - \sum_i a_i^{-} v_i' + \sum_{i,j} a_i^{-} x_{i,j} w'_j = -
\sum_{i,j} a_i^{-}(x_{i,j}-a_{i,j}) w'_{j} $$
 It follows that
 \begin{eqnarray*}
 \vert u-u_++u_- &\vert& \leq n \beta\vert u \vert L' \\
 \vert u_+ \vert \leq \beta \vert u \vert \\
\vert u_- \vert \leq \beta (L+L')\vert u \vert \\
\end{eqnarray*}

We now take $L' = \frac{\epsilon}{n\beta}$ and $\alpha = \beta (L+L')$. This finishes the proof of the lemma.

We now show that $\mathscr{F}$ is a Banach sheaf. Let $\mathcal{V} = \Spm~A_{\mathcal{V}} $ be an affinoid open
sub-set of $\cX$. If we
take the completed tensor product of the sequence of the lemma with $A_{\mathcal{V}}$, we get the exact sequence
$$0 \rightarrow  \HH^0(\cX, \mathscr{F}) \hat{\otimes}_A A_{\mathcal{V}} \rightarrow \HH^0(\cU_+ \cap \mathcal{V}, \mathscr{F}) \oplus
\HH^0(\cU_-\cap \mathcal{V} , \mathscr{F}) \rightarrow \HH^0(\cU\cap \mathcal{V}, \mathscr{F}) \rightarrow 0,$$ which proves that
$\HH^0(\cX,\mathscr{F}) \hat{\otimes}_A A_{\mathcal{V}} = \HH^0(\mathcal{V},\mathscr{F})$.
\end{demo}

\bigskip

We would now like to raise a question about proper morphisms. Let $\phi\colon  \cX \rightarrow \cY$ be a
proper morphism.

%Recall that a morphism $\phi: X \rightarrow Y$ between two rigid spaces is proper if:

%\begin{itemize}

%\item $\phi$ is separated,

%\item There is an admissible affinoid covering $\{Y_i\}_{i\in I}$ of $Y$ and for all $i \in I$, two finite admissible coverings by affinoids $\{U_{i,j}\}_{j\in J}$ and $\{V_{i,j}\}_{j\in J}$ of $X\vert_{Y_i}$ such that $U_{i,j}$ is relatively compact in $V_{i,j}$ over $Y_i$.

%\end{itemize}

\begin{opprob}\label{opprobB} Let $\mathscr{F}$ be a   Banach sheaf over $\cX$. Is $\phi_\ast \mathscr{F}$ a
Banach sheaf over $\cY$?
\end{opprob}

When $\mathscr{F}$ is a coherent module, the answer is positive (this is a theorem of Kiehl, see \cite{Bos}, p. 105,  thm. 9).

%It is possible to define a Cech cohomology for Banach sheaves. We are actually able to prove that for all $q$,
$R^{q}\phi_\ast \mathscr{F}$ is a Banach sheaf.

%The statement of the theorem is local over $Y$, so we can assume that $I=\{pt\}$.

%It is also possible to assume that for all $U_i$, $\HH^0(\mathfrak{U}_i,\mathscr{F})$ is projective.  We will first show the following:
%All I can say for the moment is that this would follow from the following proposition, but I am not able to prove it.
%\begin{prop}[????] Consider the Chech resolution associated to $\mathfrak{U} = \{U_j\}$:
%$$ C^{0} (  \mathfrak{U}, \mathscr{F}) \rightarrow C^{1}(\mathfrak{U}, \mathscr{F})\rightarrow \ldots$$
%Then the $q$-boundaries $B^q$ are closed in the $q$-cocycles $Z^q$.
%\end{prop}

%\begin{coro} The modules $R\Gamma^{q}(Y, \mathscr{F})$ are naturally Banach modules.
%\end{coro}

\subsection{Formal geometry and Banach sheaves}\label{sect_formal_ban}

 Let $\mathfrak{X}$ be a formal scheme over $\Spf~\ocal_K$. Let $\varpi$ be a uniformizing element in
$\ocal_K$. We denote
by $\cX$ the rigid-analytic fiber of $\mathfrak{X}$ and by $X_n$ the scheme over
$\Spec~\ocal_{K}/\varpi^n\ocal_K$ deduced from $\mathfrak{X}$ by reduction modulo $\varpi^n$.
 \begin{defi}\label{def:formalBanach}  A  formal Banach sheaf on $\mathfrak{X}$ is a family of quasi coherent sheaves
$\mathfrak{F} = (\mathscr{F}_n)_{n\in \N}$ where:
 \begin{enumerate}
 \item $\mathscr{F}_n$ is a sheaf over $X_n$,
 \item  for all $n \geq m$, if $i\colon  X_m \hookrightarrow X_n$ is the closed immersion, we have
$i^\ast \mathscr{F}_n = \mathscr{F}_m$.
 \end{enumerate}
 \end{defi}

   Let $\mathfrak{F}$ be a formal Banach sheaf. We can associate to $\mathfrak{F}$ a Banach sheaf
$\mathscr{F}$ on
$\cX$ as follows. Let $\Spf ~A$ be an open affine sub-set of $\mathfrak{X}$.
Let $M_n = \HH^0(A/\varpi^n A, \mathscr{F}_n)$.
We set $\mathscr{F}(\Spm~A_{\qq_p}) = ( \mathrm{lim}_n M_n) \otimes_{\ocal_K} K$. This is a Banach
$A_{\qq_p}$-module with
unit ball the image of $  \mathrm{lim}_n M_n$ in $( \mathrm{lim}_n M_n) \otimes_{\ocal_K} K$.

We now give a simple criterion for the direct image of a Banach sheaf to be a Banach sheaf.
Let $\Phi\colon
\mathfrak{X} \rightarrow \mathfrak{Y}$ be a quasi-compact and quasi-separated morphism between two
formal schemes.
As before we denote
by $X_n$ and $Y_n$ the schemes obtained by reduction modulo $\varpi^n$ and by $\phi_n\colon  X_n
\rightarrow Y_n$ the induced map. Let $\mathfrak{F}$ be a formal Banach sheaf on $\mathfrak{X}$.
Let $\cX$ and $\cY$ be the rigid spaces associated to $\mathfrak{X}$ and $\mathfrak{Y}$ and $\phi\colon
\cX \rightarrow \cY$ the induced map.  For all $n \geq m$, we have a cartesian diagram:

\begin{eqnarray*}
\xymatrix{ X_n \ar[r]^{i}\ar[d]^{\phi_n} & X_m \ar[d]^{\phi_m}\\
Y_n \ar[r]^{j} &Y_m}
\end{eqnarray*}
\begin{prop}\label{last-prop} Assume that for all $n \geq m$ we have the base change property:
$$ j^\ast \bigl( \phi_{m,\ast}  \mathscr{F}_m\bigr) \simeq \phi_{n,\ast} \mathscr{F}_n $$

then $\phi_\ast \mathscr{F}$ is a Banach sheaf.
\end{prop}
\begin{demo} Indeed $\phi_\ast \mathscr{F}$ is the Banach sheaf associated to the formal Banach sheaf $(\phi_{n,\ast} \mathscr{F}_n)_{n\in \N}$.
\end{demo}

\section{\bf List of symbols}

$B$ standard Borel in $\GL$, \S\ref{sec:algrep}

$U\subset B$ unipotent radical, \S\ref{sec:algrep}

$T\subset B$ standard torus, \S\ref{sec:algrep}

$B^0$ Borel opposite to $B$, \S\ref{sec:algrep}

$U^0\subset B^0$ unipotent radical, \S\ref{sec:algrep}

$I$ Iwahori subgroup of $\GL(\Z_p)$, \S\ref{sec:analrep}

$\mathfrak{T}$, $\mathfrak{T}_w$, formal torus, \S\ref{sect-main}

$\mathfrak{B}_w$, $\mathfrak{U}_w$, formal groups, \S\ref{sec:ThesheafcalF}

$\mathcal{W}$ weight space, \S\ref{sect_ws}

$\mathcal{W}(w)$, $\mathcal{W}(w)^o$, \S\ref{sect_ws}

$\kappa^{\un}$, universal character,  proposition \ref{prop_cara}

$\kappa\mapsto\kappa'$, involution on weights, \S\ref{sec-classicalmf}

$Y$ moduli space of principally polarized abelian
schemes $(A,\lambda)$ of dimension $g$ equipped with a  principal level $N$, \S\ref{sec-classicalmf}

$Y\subset X$ toroidal compactification, \S\ref{sec-classicalmf}

$Y\subset X^\star$, minimal compactification, \S\ref{sec:mimimalcompact}

$\mathfrak{X}$,  formal scheme associated to $X$, \S\ref{sec:blowup}

$Y_{\mathrm{Iw}}$ moduli space with principal level $N$ structure and Iwahori structure at $p$, \S\ref{sec-classicalmf}

$Y_{\mathrm{Iw}}\subset \XI$  toroidal compactification, \S\ref{sec-classicalmf}

$\mathfrak{X}_1(p^n)(v)$, $\mathfrak{X}_{\mathrm{Iw}^+}(p^n)(v)$, $\mathfrak{X}_{\mathrm{Iw}}(p^n)(v)$, $\mathfrak{X}(v)$ formal schemes,  \S\ref{sec:blowup}

$\cX(v)$ rigid space, neighbourhood of ordinary locus of width $v$, \S\ref{sec:blowup}

$\cX_1(p^n)(v)$, $\cX_{\mathrm{Iw}^+}(p^n)(v)$, $\cXI(p^n)(v)$, rigid spaces, \S\ref{sect-rigide}

$\mathcal{F}$, proposition \ref{prop-FHT}

$\mathfrak{IW}_w$, Grassmannian of $w$-compatible flags in $\mathcal{F}$,   \S\ref{sect-main}

$\mathfrak{IW}^+_w$, Grassmannian of $w$-compatible flags in $\mathcal{F}$ and bases elements of the graded pieces, \S\ref{sect-main}

$\mathcal{IW}_w^+$,  rigid space over $ {\cX}_1(p^n)(v)$ associated to $\mathfrak{IW}_w^+$, \S\ref{sect-rigide}

$\mathcal{IW}_w$, rigid space over $ {\cX}_1(p^n)(v)$ associated to  $\mathfrak{IW}_w$, \S\ref{sect-rigide}

$\mathcal{IW}^{o+}_w$, descent of $\mathcal{IW}_w^+$ to  $\cX_{\mathrm{Iw}^+}(p^n)(v)$, \S\ref{sect-rigide}

$\mathcal{IW}^o_w$,  descent of $\mathcal{IW}_w$ to  $\cXI(p^n)(v)$, \S\ref{sect-rigide}

$\mathcal{IW}^{o+}_{\underline{w}}$, rigid space with dilations parameters, \S\ref{sec:dilations}

$\mathfrak{w}^{\dag\kappa}_w$ the formal Banach sheaf of $w$-analytic, $v$-overconvergent modular forms of weight $\kappa$,
definition \ref{def:formalintegralsheaf}

$\omega^{\dag\kappa}_w$
Banach sheaf of $w$-analytic, $v$-overconvergent weight $\kappa$ modular forms, \S\ref{sect-rigide}

$\mathrm{M}_w^{\dag\kappa}(\cXI(p)(v))$, $\mathrm{M}^{\dag\kappa}(\XI(p))$ space of overconvergent modular forms of weight $\kappa$, definition \ref{def:overconvmodforms}

$\omega^{\dag\kappa^{\un}}_{w}$, $M_{v,w}$, $M^\dag$ families of overconvergent modular forms, \S\ref{sec:universalomega}

$\tilde{\mathfrak{w}}^{\dag\kappa^{o\un}}_w$, family of integral overconvergent modular forms, \S\ref{sec:anintegralfamily}

$\omega^{\dag\kappa}_{\underline{w}}$, $\mathrm{M}_{\underline{w}}^{\dag\kappa}(\cXI(p)(v))$, variants with dilations parameters, \S\ref{sec:dilations}

$U_{p,g}$, $U$ operator, \S\ref{sec:Upg}

$U_{p,i}$, $U$ operator, \S\ref{sec:Upi}

\end{document}